\documentclass[11pt]{amsart}
\usepackage[margin=1.2in]{geometry}
\usepackage[T1]{fontenc}
\usepackage[utf8]{inputenc}
\usepackage{amsmath,amssymb,amsthm,mathtools,bm,mathrsfs}
\usepackage{enumitem}
\usepackage{microtype}
\usepackage{booktabs,longtable,array}
\usepackage{float}
\usepackage{graphicx}
\usepackage{xcolor}
\usepackage[numbers]{natbib}
\usepackage[colorlinks=true,linkcolor=blue,citecolor=blue,urlcolor=blue]{hyperref}
\usepackage[capitalise,nameinlink]{cleveref}
\hypersetup{
	pdftitle={Nonparametric Inference for Semigroup Blocks of Switching Diffusions},
	pdfauthor={Yuzhong Cheng}
}

\newif\ifshowonlycitedequations
\showonlycitedequationstrue
\ifshowonlycitedequations
\mathtoolsset{showonlyrefs=true}
\fi

\newtheorem{theorem}{Theorem}[section]
\newtheorem{assumption}{Assumption}
\renewcommand{\theassumption}{A\arabic{assumption}}
\newtheorem{proposition}[theorem]{Proposition}
\newtheorem{corollary}[theorem]{Corollary}
\newtheorem{lemma}[theorem]{Lemma}
\theoremstyle{definition}
\newtheorem{definition}[theorem]{Definition}
\theoremstyle{remark}
\newtheorem{remark}[theorem]{Remark}
\crefname{assumption}{Assumption}{Assumptions}
\Crefname{assumption}{Assumption}{Assumptions}
\crefname{definition}{Definition}{Definitions}
\Crefname{definition}{Definition}{Definitions}
\crefname{theorem}{Theorem}{Theorems}
\Crefname{theorem}{Theorem}{Theorems}
\crefname{proposition}{Proposition}{Propositions}
\Crefname{proposition}{Proposition}{Propositions}
\crefname{corollary}{Corollary}{Corollaries}
\Crefname{corollary}{Corollary}{Corollaries}
\crefname{lemma}{Lemma}{Lemmas}
\Crefname{lemma}{Lemma}{Lemmas}
\crefname{remark}{Remark}{Remarks}
\Crefname{remark}{Remark}{Remarks}

\newcommand{\Pp}{\mathbb P}

\title{Nonparametric Inference for Semigroup Blocks of Switching Diffusions}
\author{Yuzhong Cheng}
\address{Institute of Mathematics for Industry, Kyushu University,
744 Motooka, Fukuoka, Japan}
\email{cheng.yuzhong.451@m.kyushu-u.ac.jp}
\date{July 14, 2026}

\begin{document}

	\begin{abstract}
		{Regime-conditioned transition probabilities and moments are basic inputs for prediction and decision making in hybrid systems, but their short-time infinitesimal structure is not directly observable. We study nonparametric estimation of regime-indexed semigroup blocks for switching diffusions observed together with their regimes. A Dynkin--Taylor expansion identifies the first two block-generator coefficients, and localized probes recover switching intensities, drift, and diffusion. Under shrinking meshes, a common-design difference cancels the localized level and yields a martingale array; nonoverlapping second differences preserve within-block covariance and produce the variance factor two. The resulting first- and second-order estimators are asymptotically normal, admit feasible studentization, and support short-horizon approximation and specification checks for interactions among primitive coefficients. Fixed-mesh local-polynomial theory and a numerical study complement the recovery results.}
	\end{abstract}
	\maketitle

	\section{Introduction}
	Hybrid systems with continuous dynamics and discrete events arise naturally in stochastic control, queueing, signal processing, mathematical finance, and biological regulation. A mathematically flexible class of such models is provided by hybrid switching diffusions, in which a continuous diffusion component evolves together with a finite regime process; see \citet{yinzhu2010} for a systematic treatment.

	Let \((\Omega,\mathcal F,(\mathcal F_t)_{t\ge0},\mathbb P)\) be a filtered probability space satisfying the usual conditions. We study hybrid switching diffusions of the following form:
	\begin{equation}\label{eq:model}
		\begin{aligned}
			dX_t&=b(X_t,\Lambda_{t-})\,dt+\sigma(X_t,\Lambda_{t-})\,dW_t,\\
			d\Lambda_t&=\int_0^\infty \kappa(X_t,\Lambda_{t-},z)\,N(dt,dz),
		\end{aligned}
	\end{equation}
	where \(W\) is an \(r\)-dimensional Wiener process and \(N(dt,dz)\) is an independent Poisson random measure on \([0,\infty)\times[0,\infty)\) with intensity \(dt\,dz\). The regime process \(\Lambda_t\) takes values in a finite set and
	$\kappa(x,i,z)=\sum_{j\ne i}(j-i)\mathbf 1_{\Gamma_{ij}(x)}(z)$.
	The intervals \(\Gamma_{ij}(x)\) are specified in \Cref{sec:model}.

	Regime-indexed sampled quantities appear directly in applications. In the regime-switching jump-diffusion option model of \citet{goswamimanjarekaranjana2018}, prices depend on regime transition probabilities and conditional terminal payoffs, while the ecological management model of \citet{yoshiokayaegashitsujimura2020} uses forecasts under changing environmental regimes. Both require terminal-regime probabilities or conditional moments indexed by the present and future regimes. We observe \(\{(Y_k,\Lambda_{t_k}):0\le k\le n\}\) on the grid \(t_k=k\Delta\), with \(Y_k=X_{t_k}\), and estimate the block switching functional
		\begin{equation}\label{eq:intro-block}
			P_\Delta^{ij}g(x)
			:=\mathbb E\!\left[
			g(X_{t+\Delta})\mathbf 1_{\{\Lambda_{t+\Delta}=j\}}
			\mid X_t=x,\Lambda_t=i
			\right].
		\end{equation}
		For \(g\equiv1\), this reduces to the one-step transition probability from regime \(i\) to regime \(j\). For general \(g\), it gives the numerator of the regime-specific conditional moment; whenever \(P_\Delta^{ij}1(x)>0\), the normalized conditional moment is
		\begin{equation}\label{eq:intro-conditional-moment}
			\mathbb E\!\left[
			g(X_{t+\Delta})\mid X_t=x,\Lambda_t=i,\Lambda_{t+\Delta}=j
			\right]
			=\frac{P_\Delta^{ij}g(x)}{P_\Delta^{ij}1(x)}.
		\end{equation}
		{The short-time expansion
			$P_\Delta^{ij}g(x)
			=\delta_{ij}g(x)+\Delta B_{ij}g(x)
			+\frac{\Delta^2}{2}C_{ij}g(x)+o(\Delta^2)$
		makes precise how finite-horizon blocks encode infinitesimal dynamics. Evaluating the first-order coefficient \(B_{ij}g\) at localized constant, linear, and quadratic probes identifies the primitive switching intensities, drift, and diffusion matrix.}

	Together with \(B_{ij}g\), the second-order coefficient \(C_{ij}g\) gives a second-order approximation of short-horizon block expectations and, by applying the expansion to both \(g\) and \(1\), the corresponding short-horizon correction to the normalized conditional moment above. It also directly targets aggregate drift--switching and diffusion--intensity interactions; for \(i\ne j\), these include \((\nabla q_{ij})^\top a(\cdot,i)\nabla g\), where \(a(x,i)=\sigma(x,i)\sigma(x,i)^\top\). Comparing a direct estimate of \(C_{ij}g\) with the value implied by a separately fitted primitive-coefficient model therefore provides a specification check that is unavailable from first-order point identification alone.

	To our knowledge, nonparametric estimation of observed-regime semigroup blocks \(P_\Delta^{ij}g\) and finite-difference recovery of their generator coefficients have not previously been studied. For ordinary diffusions, \citet{bandiphillips2003} develop fully nonparametric drift and diffusion estimation, \citet{gobethoffmannreiss2004} use spectral methods for low-frequency data, and \citet{comtegenoncatalotrozenholc2007} give penalized estimators with nonasymptotic risk bounds. These works have neither a terminal-regime block nor a switching indicator in the response. The closest observed-regime inference result is \citet{cheng2025statistical}, who use Gaussian quasi likelihood under high-frequency sampling to estimate parametric drift and diffusion coefficients and the generator of the switching chain. Recent adjacent work treats different observation schemes and targets: \citet{stumpffetizonetal2025exact} develop exact MCMC and MCEM methods for diffusion parameters and a latent switching path, while \citet{cheng2026semiparametric} combines truncated Gaussian quasi likelihood with regime-wise kernel smoothing to estimate parametric continuous coefficients and unknown L\'evy densities in switching jump diffusions. None of these papers estimates an observed-regime semigroup block or recovers its generator coefficients by nonparametric finite differences.

	{To address this estimation problem, we combine established semigroup, smoothing, and limit-theorem tools with problem-specific response constructions. Repeated Dynkin formulas applied to \(f_g^{(j)}(x,\ell)=g(x)\mathbf 1_{\{\ell=j\}}\) produce the block-generator hierarchy. The fixed-mesh analysis adapts local-polynomial regression for stationary mixing time series from \citet{robinson1983,masry1996,masryfan1997}; its triangular-array step uses Berbee coupling and absolute-regularity bounds as in \citet{berbee1979random,bradley2007introduction}. Under shrinking meshes, the common-design differences turn the localized responses into martingale-difference arrays, and the central limit argument applies \citet[Theorem~2.3]{mcleish1974}. The common-design cancellation and the nonoverlapping second-difference construction are the problem-specific ingredients: the former removes the localized level before smoothing, while the latter retains the covariance of the two within-block innovations and yields the factor two in the second-order variance.}

	{The fixed-mesh consistency and central limit theorems extend the classical mixing local-polynomial results through the regime-specific design and block response. The shrinking-mesh recovery theorems have a different target from the ordinary-diffusion procedures of \citet{bandiphillips2003,gobethoffmannreiss2004} and the parametric switching-diffusion inference of \citet{cheng2025statistical}: they estimate observed-regime semigroup blocks nonparametrically and recover their first two finite-difference generator coefficients under state-dependent switching. In particular, the results cover both diagonal diffusion-and-exit noise and off-diagonal rare-switch noise, give Gaussian limits at rates \(\sqrt{n\Delta_nh_n^d}\) and \(\sqrt{N_n\Delta_n^3h_n^d}\), and provide feasible studentization for the second-order coefficient, which is not a target of the cited coefficient-estimation theories.}

	{\Cref{sec:model,sec:assumptions} introduce the model and coefficient hierarchy; \Cref{sec:estimation,sec:recovery} give the fixed- and shrinking-mesh theory; and \Cref{sec:examples,sec:discussion} give the numerical study and discussion. The supplement contains arbitrary-order consistency and detailed auxiliary proofs.}

	\section{Model and observation scheme}
	\label{sec:model}
	Throughout, \(a_n\lesssim b_n\) means \(a_n\le Cb_n\) for all large \(n\), with \(C\) independent of \(n\), \(\Delta_n\), and \(h_n\). Subscripts on \(\lesssim\) indicate additional fixed dependence.
	Throughout the paper and online supplement, every estimator or studentized statistic is set equal to zero on the complement of its defining invertibility or positivity event.

	\subsection*{Hybrid switching diffusion}
	Let \(S=\{1,\ldots,m\}\) with \(m\ge 2\), and put \(E:=\mathbb R^d\times S\). {The joint process \((X,\Lambda)\) takes values in \(E\), and \(\mathbb P_{(x,i)}\) denotes its law when initialized at \((x,i)\in E\).} {Expectation and variance under \(\mathbb P_{(x,i)}\) are denoted by \(\mathbb E_{(x,i)}\) and \(\operatorname{Var}_{(x,i)}\), respectively.} We use the filtered probability space, noises, and process \((X_t,\Lambda_t)_{t\ge0}\) introduced in \eqref{eq:model}.
	For rate functions \(q_{ij}:\mathbb R^d\to[0,\infty)\), \(i\ne j\), choose for each \(i\) disjoint left-closed, right-open intervals
	\[
	\Gamma_{ij}(x)
	:=
	\left[
	\sum_{\ell<j,\ \ell\ne i}q_{i\ell}(x),
	\sum_{\ell\le j,\ \ell\ne i}q_{i\ell}(x)
	\right),
	\qquad j\ne i,
	\]
	where empty sums are zero.
	The Poisson representation specifies the switching intensities in the conditional sense: for \(i\ne j\),
	\begin{equation}\label{eq:qij}
		\mathbb P(\Lambda_{t+h}=j\mid\mathcal F_t)
		=q_{ij}(X_t)h+o(h)
		\quad\text{on }\{\Lambda_t=i\},
		\qquad h\downarrow0.
	\end{equation}
	When \(q_{ij}\) are constant in \(x\), one recovers the Markovian switching case treated in much of the classical literature. The state-dependent case is more delicate because the continuous and discrete dynamics are coupled already at the level of the jump intensities; see \citet{yinzhu2010}. Write
	\[
	q_{ii}(x)=-\sum_{j\ne i}q_{ij}(x),\qquad q_i(x)=\sum_{j\ne i}q_{ij}(x),\qquad a(x,i)=\sigma(x,i)\sigma(x,i)^\top.
	\]
	For a smooth function \(g:\mathbb R^d\to\mathbb R\), define the within-regime diffusion operator
	\begin{equation}\label{eq:Ai}
		\mathscr L_i g(x):=b(x,i)^\top\nabla g(x)+\frac12\operatorname{tr}\bigl(a(x,i)\nabla^2 g(x)\bigr).
	\end{equation}
	For bounded measurable \(f:E\to\mathbb R\), write
	$P_t f(x,i):=\mathbb E[f(X_t,\Lambda_t)\mid X_0=x,\Lambda_0=i]$
	for the Markov semigroup of the hybrid process. We use \(\mathcal L\) as the classical differential--jump operator. For a function \(f:E\to\mathbb R\) whose regime components \(f(\cdot,i)\) are twice continuously differentiable, set
	\[
	\mathcal D_0:=\{f:E\to\mathbb R: f(\cdot,i)\in C_c^2(\mathbb R^d)\ \text{for every }i\in S\},
	\]
	and define
	\begin{equation}\label{eq:generator}
		\mathcal L f(x,i)=b(x,i)^\top\nabla_x f(x,i)+\frac12\operatorname{tr}\bigl(a(x,i)\nabla_x^2 f(x,i)\bigr)
		+\sum_{j\ne i}q_{ij}(x)\bigl(f(x,j)-f(x,i)\bigr).
	\end{equation}

	\subsection*{Observation scheme}
	We observe the pair \((X_t,\Lambda_t)\) on the grid
	$t_k=k\Delta$, $Y_k=X_{t_k}$, $k=0,1,\ldots,n$.
	Thus the data consist of
	$\{(Y_k,\Lambda_{t_k}):0\le k\le n\}$.

	\subsection*{Block switching functional}
	{The definition below applies to bounded measurable \(g\). The fixed-mesh theory fixes \(g\in C_b(\mathbb R^d)\), while the model-level order-\(k\) expansion fixes \(g\in C_c^{2k}(\mathbb R^d)\).}

	\begin{definition}
		For \(\Delta>0\), a regime pair \((i,j)\in S^2\), and a bounded measurable function \(g:\mathbb R^d\to\mathbb R\), define
		\begin{equation}\label{eq:block}
			P_\Delta^{ij}g(x)
			:=
			\mathbb E\!\left[
			g(X_{t+\Delta})
			\mathbf 1_{\{\Lambda_{t+\Delta}=j\}}
			\,\middle|\,
			X_t=x,\Lambda_t=i
			\right].
		\end{equation}
		The family \(\{P_\Delta^{ij}g:i,j\in S\}\) is called the block switching functional associated with \(g\).
	\end{definition}

	\section{Assumptions and preliminaries}
	\label{sec:assumptions}
	All statistical results are pointwise: fix a design point \(x\in\mathbb R^d\) and a bounded open neighborhood \(U_x\) such that \(x\in U_x\Subset\mathbb R^d\).

	{\begin{assumption}[Coefficient regularity]\label[assumption]{ass:wellposedness}
			For each regime \(i\in S\), the coefficients \(b(\cdot,i)\) and \(\sigma(\cdot,i)\) belong to \(C^4(\mathbb R^d)\) and have bounded derivatives of orders \(1,\ldots,4\). For \(i\ne j\), the switching intensities satisfy \(q_{ij}\ge0\), belong to \(C^4(\mathbb R^d)\), and have bounded derivatives of orders \(1,\ldots,4\). Moreover, for some \(C_q<\infty\),
			$\sup_{i\in S}\sup_{x\in\mathbb R^d}q_i(x)\le C_q$.
	\end{assumption}}

	The bounded first derivatives imply global Lipschitz continuity and linear-growth estimates. Together with the bounded total switching rate, these conditions give a pathwise unique nonexplosive strong solution and the strong Markov property by the standard state-dependent switching-diffusion results in \citet[Ch.~2]{yinzhu2010}. Under the same coefficient conditions, \citet[Theorem~3.2]{xi2008feller} gives Feller continuity.

	{\begin{assumption}[Exponential ergodicity]\label[assumption]{ass:mixing}
			The joint process \((X_t,\Lambda_t)_{t\ge0}\) admits an invariant probability measure \(\nu\) on \(\mathbb R^d\times S\). Moreover, there exist a measurable function \(V:\mathbb R^d\times S\to[1,\infty)\) with \(\nu(V)<\infty\), and constants \(C_{\rm erg},\lambda_{\rm erg}>0\), such that
			\[
			\|P_t((x,i),\cdot)-\nu\|_V
			\le C_{\rm erg}e^{-\lambda_{\rm erg}t}V(x,i),
			\qquad t\ge0,\quad (x,i)\in\mathbb R^d\times S,
			\]
			where
			$\|\mu\|_V:=\sup_{|f|\le V}|\mu(f)|$
			for every finite signed measure \(\mu\) on \(\mathbb R^d\times S\).
			In addition,
			\[
			\sup_{(y,\ell)\in U_x\times S}V(y,\ell)<\infty.
			\]
			{Throughout the paper, the joint process is initialized under its invariant law: \((X_0,\Lambda_0)\sim\nu\).}
	\end{assumption}}

	{The initialization in Assumption~A2 makes the continuous-time process strictly stationary. For a fixed mesh \(\Delta>0\), put \(Z_k:=(Y_k,\Lambda_{t_k})\), \(k\ge0\). Then the sampled chain is strictly stationary.}
	Its absolute-regularity coefficients are
	\[
	\beta_\Delta(r)
	:=\sup_{k\ge0}
	\mathbb E\!\left[
	\sup_{B\in\mathcal F_{k+r}^\infty}
	\bigl|\mathbb P(B\mid\mathcal F_0^k)-\mathbb P(B)\bigr|
	\right]
	,\qquad r\ge1,
	\]
	where \(\mathcal F_a^b:=\sigma(Z_k:a\le k\le b)\) and
	\(\mathcal F_a^\infty:=\sigma(Z_k:k\ge a)\).
	Since \(E\) is a standard Borel space, the standard transition-kernel representation of the absolute-regularity coefficients of a strictly stationary Markov chain applies; see \citet[Proposition~3.22\textup{(III.5)}]{bradley2007introduction}. Therefore, since \(V\ge1\),
		\[
		\begin{aligned}
			\beta_\Delta(r)
			&=
			\int_E
			\sup_{A\in\mathcal B(E)}
			\left|P_{r\Delta}(z,A)-\nu(A)\right|
			\,\nu(dz)
			\le
			\frac12\int_E
			\left\|P_{r\Delta}(z,\cdot)-\nu\right\|_V
			\,\nu(dz) \\
			&\le
			\frac12 C_{\rm erg}\nu(V)e^{-\lambda_{\rm erg}r\Delta},
			\qquad r\ge1,
		\end{aligned}
		\]
		where the last inequality follows from \Cref{ass:mixing}.
	This is the probabilistic input needed for the law of large numbers and triangular-array central limit theorem arguments for the kernel sums.

	The final standing assumptions record the local design positivity and sampled-block smoothness.

	{\begin{assumption}[Local stationary design and fixed-mesh small-ball regularity]\label[assumption]{ass:local-densities}
			\textup{(i)} For every \(i\in S\), there exists a stationary sublaw density \(\varpi_i\) such that
				$\nu(dy,\{i\})=\varpi_i(y)\,dy$,
				where \(\varpi_i\) is continuous on \(U_x\) and, for constants \(0<c_{i,x}\le C_{i,x}<\infty\),
				\[
				c_{i,x}\le\varpi_i(y)\le C_{i,x},
				\qquad y\in U_x.
				\]

			\textup{(ii)} For every fixed sampling mesh \(\Delta>0\) and every bounded Borel set \(D\subset\mathbb R^d\), there exist \(C_{\Delta,D,x}<\infty\) and \(h_{\Delta,D,x}>0\) such that
				\[
				\sup_{r\ge1}\sup_{a,b\in S}
				\mathbb P\!\left(
				Y_0\in x+hD,\ Y_r\in x+hD,\
				\Lambda_{t_0}=a,\ \Lambda_{t_r}=b
				\right)
				\le C_{\Delta,D,x}h^{2d},
				\qquad 0<h\le h_{\Delta,D,x}.
				\]
	\end{assumption}}

	{\begin{remark}\label[remark]{rem:mixing-density-distinct}
			Assumption~A3\textup{(ii)} is stated in the local form used by the fixed-mesh central limit theorem. A convenient sufficient condition is that, for every fixed \(\Delta>0\), the two-point stationary sublaws admit densities \(p_{r,\Delta}^{ab}\) on \(U_x\times U_x\) satisfying
			\[
			\sup_{r\ge1}\sup_{a,b\in S}\sup_{(y,z)\in U_x\times U_x}
			p_{r,\Delta}^{ab}(y,z)<\infty.
			\]
	\end{remark}}

	{\begin{assumption}[Fixed-mesh block regularity]\label[assumption]{ass:fixedmesh-smoothness}
			Fix a sampling mesh \(\Delta>0\), a polynomial degree \(p\ge0\), a bounded test function \(g\in C_b(\mathbb R^d)\), and a regime pair \((i,j)\). The sampled-block regression function
			$m_{\Delta,g}^{ij}(y):=P_\Delta^{ij}g(y)$
			satisfies the following conditions, invoked separately below:
			\begin{enumerate}[label=\textup{(\roman*)}]
				\item \(m_{\Delta,g}^{ij}\) is continuous at \(x\).
				\item \(m_{\Delta,g}^{ij}\in C^{p+1}(U_x)\), and
				$\max_{|\alpha|\le p+1}\sup_{y\in U_x}
				\left|\partial^\alpha m_{\Delta,g}^{ij}(y)\right|<\infty$.
				\item The conditional second-moment function
				$s_{\Delta,g}^{ij}(y):=P_\Delta^{ij}(g^2)(y)$
				is continuous at \(x\).
				\item The conditional variance is positive at the design point:
				\[
				v_{\Delta,g}^{ij}(x)
				:=s_{\Delta,g}^{ij}(x)
				-\bigl\{m_{\Delta,g}^{ij}(x)\bigr\}^2>0.
				\]
			\end{enumerate}
	\end{assumption}}

	{\Cref{ass:fixedmesh-smoothness} may be verified by a model-specific semigroup regularity theorem. The ergodicity condition follows, for example, from a Foster--Lyapunov drift condition together with irreducibility and minorization; see \citet{xi2008feller,xiyin2011uniform,tongmajda2016ergodic,cloezhairer2015,lixi2022tv}. Wasserstein contraction alone does not imply the required \(V\)-norm or \(\beta\)-mixing bounds.}

	\subsection*{Block-generator coefficient hierarchy}
	Finally, we record the block-generator coefficient hierarchy used by the shrinking-mesh recovery theory. The block semigroup \(P_\Delta^{ij}g\) admits block-generator coefficients indexed by the expansion order, and the later statistical target at order \(k\) is the \(k\)-th coefficient in this hierarchy.

		Let \(g:\mathbb R^d\to\mathbb R\) be a test function, let \(j\in S\), and define
		\(f_g^{(j)}(y,\ell):=g(y)\mathbf 1_{\{\ell=j\}}\).
		Set \(\mathcal L^0 f_g^{(j)}:=f_g^{(j)}\). Whenever the formal iterates \(\mathcal L^\ell f_g^{(j)}\), \(\ell=1,\ldots,m\), are well defined, set
		\[
		A_m^{ij}g(x):=(\mathcal L^m f_g^{(j)})(x,i),\qquad i,j\in S.
		\]
		In particular,
		$A_0^{ij}g(x)=f_g^{(j)}(x,i)=\delta_{ij}g(x)$.

	{We first verify that the first two block-generator coefficients are well defined. The subsequent proposition establishes the general order-\(k\) expansion under the corresponding higher-order regularity.}
	\begin{proposition}\label[proposition]{prop:block-domain}
		{Under \Cref{ass:wellposedness}, fix \(g\in C_c^4(\mathbb R^d)\) and \(j\in S\). Then the block test \(f_g^{(j)}\) belongs to \(\mathcal D_0\), \(\mathcal Lf_g^{(j)}\in\mathcal D_0\), and \(\mathcal L^2f_g^{(j)}\) is bounded and continuous with compact support. Consequently, \(A_1^{ij}g\) and \(A_2^{ij}g\) are well defined for all \(i,j\in S\).}
	\end{proposition}

	\begin{proof}
		Since \(g\in C_c^4(\mathbb R^d)\), \(f_g^{(j)}\in\mathcal D_0\). By the operator formula \eqref{eq:generator},
		\[
		(\mathcal L f_g^{(j)})(x,\ell)
		=
		\begin{cases}
			\mathscr L_j g(x)-q_j(x)g(x),& \ell=j,\\
			q_{\ell j}(x)g(x),& \ell\ne j.
		\end{cases}
		\]
		Each component is \(C_c^2(\mathbb R^d)\). Indeed, \(g\) is compactly supported and has four continuous derivatives, while \(b\), \(\sigma\), and the switching rates have the derivatives required in \Cref{ass:wellposedness}. Hence \(\mathcal Lf_g^{(j)}\in\mathcal D_0\). Applying \eqref{eq:generator} once more gives a bounded continuous function with compact support, because \(\mathcal Lf_g^{(j)}\) has compact support and two continuous derivatives. This proves the claim.
	\end{proof}

	For functions \(F_t,F:E\to\mathbb R\), we write \(F_t\to F\) compact-uniformly as \(t\downarrow0\) when, for every compact \(K\subset\mathbb R^d\),
	\[
	\lim_{\delta\downarrow0}
	\sup_{0\le t\le\delta}\max_{i\in S}\sup_{x\in K}|F_t(x,i)-F(x,i)|=0.
	\]
	{\Cref{prop:kth} is a model-level sufficient condition for the sampled-block expansion; the recovery theorem later assumes that local expansion directly.}
	\begin{proposition}[Block expansion criterion]\label[proposition]{prop:kth}
		Fix \(k\ge1\), \(j\in S\), and \(g\in C_c^{2k}(\mathbb R^d)\). Suppose \Cref{ass:wellposedness} holds. If \(k>2\), assume in addition that, for every regime pair, \(b(\cdot,i)\), \(\sigma(\cdot,i)\), and \(q_{ij}\) belong to \(C^{2k}(\mathbb R^d)\) with bounded derivatives of orders \(1,\ldots,2k\). Then, for every \(i\in S\),
		\begin{equation}\label{eq:kth-expansion}
			P_\Delta^{ij}g(x)=\sum_{m=0}^k\frac{\Delta^m}{m!}A_m^{ij}g(x)+o(\Delta^k),\qquad \Delta\downarrow 0,
		\end{equation}
		uniformly on compact subsets of \(\mathbb R^d\).
	\end{proposition}

	\begin{proof}
		Let \(u_m:=\mathcal L^m f_g^{(j)}\), \(m=0,\ldots,k\). We first verify that these formal iterates are well defined. By the stated regularity, induction in \(m\) gives
		$u_m(\cdot,\ell)\in C_c^{2(k-m)}(\mathbb R^d)$ for
		 $m=0,\ldots,k,\quad \ell\in S$.
		Indeed, applying \(\mathcal L\) differentiates the current component at most twice and multiplies it by coefficients and switching rates with the required differentiability. The compact support is preserved because each term in \(\mathcal L u_m\) contains either \(u_m\), one of its derivatives, or a difference of its regime components. In particular, \(u_0,\ldots,u_{k-1}\in\mathcal D_0\), while \(u_k\) is bounded and continuous with compact support.

		By Itô's formula for jump diffusions, for \(m=0,\ldots,k-1\), the process
		\[
		M_t^{(m)}
		:=
		u_m(X_t,\Lambda_t)-u_m(X_0,\Lambda_0)
		-\int_0^t u_{m+1}(X_s,\Lambda_s)\,ds
		\]
		is a local martingale. For every \(T<\infty\),
		$\sup_{0\le t\le T}|M_t^{(m)}|
		\le
		2\|u_m\|_\infty+T\|u_{m+1}\|_\infty$.
		Hence \(M^{(m)}\) is bounded on \([0,T]\) and is therefore a uniformly integrable martingale on that interval. Taking expectation under \(\mathbb P_z\) yields the Dynkin identity
		\[
		P_tu_m(z)-u_m(z)=\int_0^t P_su_{m+1}(z)\,ds,
		\qquad z\in E,\quad t\ge0.
		\]
		Iterating these identities gives
		\[
		P_\Delta f_g^{(j)}
		=\sum_{m=0}^{k-1}\frac{\Delta^m}{m!}\mathcal L^m f_g^{(j)}
		+\int_0^\Delta \frac{(\Delta-s)^{k-1}}{(k-1)!}
		P_s\mathcal L^k f_g^{(j)}\,ds.
		\]
		Since
		\[
		\frac{\Delta^k}{k!}\mathcal L^k f_g^{(j)}
		=\int_0^\Delta \frac{(\Delta-s)^{k-1}}{(k-1)!}
		\mathcal L^k f_g^{(j)}\,ds,
		\]
		the remainder after subtracting the \(m=k\) term is
		\[
		\int_0^\Delta \frac{(\Delta-s)^{k-1}}{(k-1)!}
		\left(P_s\mathcal L^k f_g^{(j)}-\mathcal L^k f_g^{(j)}\right)\,ds.
		\]
		It remains to prove \(P_tu_k\to u_k\) compact-uniformly. Let \(K\subset\mathbb R^d\) be compact. By the linear growth of \(b,\sigma\), the Burkholder--Davis--Gundy inequality and Gronwall's inequality give, for \(0<t\le1\),
		\[
		\sup_{(x,i)\in K\times S}
		\mathbb E_{(x,i)}
		\left[
		\sup_{0\le s\le t}\|X_s-x\|^2
		\right]
		\lesssim_K t.
		\]
		Moreover, the uniform bound on the total switching rate implies
		\[
		\sup_{(x,i)\in K\times S}
		\mathbb P_{(x,i)}(\Lambda_t\ne i)
		\le \left(\max_{\ell\in S}\sup_y q_\ell(y)\right)t.
		\]
		Since \(u_k\) is bounded and uniformly continuous on \(\mathbb R^d\times S\), with modulus of continuity \(\omega_k(\delta)\), we have
		\[
		\sup_{(x,i)\in K\times S}|P_tu_k(x,i)-u_k(x,i)|
		\le
		\omega_k(\delta)
		+2\|u_k\|_\infty
		\left(
		\delta^{-2}C_Kt+
		\left(\max_{\ell\in S}\sup_y q_\ell(y)\right)t
		\right).
		\]
		Letting first \(t\downarrow0\) and then \(\delta\downarrow0\) proves compact-uniform convergence. The displayed remainder is therefore \(o(\Delta^k)\) uniformly on compact subsets of \(E\). Evaluating at \((x,i)\) gives \eqref{eq:kth-expansion}.
	\end{proof}

	{\subsection*{First two block-generator coefficients}
		For every test function for which the displayed expressions are defined, put
		\begin{equation}\label{eq:Bops}
			B_{ij}g(x):=
			\begin{cases}
				(\mathscr L_i-q_i)g(x), & i=j,\\
				q_{ij}(x)g(x), & i\ne j.
			\end{cases}
		\end{equation}
		Applying \(\mathcal L\) once more to \(f_g^{(j)}\) gives, for \(i\ne j\),
		\begin{equation}\label{eq:Coff}
			C_{ij}g:=\mathscr L_i(q_{ij}g)+q_{ij}\mathscr L_j g
			+\sum_{\ell\ne i,j}q_{i\ell}q_{\ell j}g-q_{ij}(q_i+q_j)g,
		\end{equation}
		and
		\begin{equation}\label{eq:Cdiag}
			C_{ii}g:=\mathscr L_i^2 g-\mathscr L_i(q_i g)-q_i\mathscr L_i g
			+q_i^2g+\sum_{\ell\ne i}q_{i\ell}q_{\ell i}g.
		\end{equation}
		The term \(\mathscr L_i(q_{ij}g)\) contains the cross-gradient interaction
		\((\nabla q_{ij})^\top a(\cdot,i)\nabla g\). Under \Cref{ass:wellposedness}, \Cref{prop:block-domain} yields, for every fixed \(g\in C_c^4(\mathbb R^d)\),
		\begin{equation}\label{eq:coeff-through-two}
			A_1^{ij}g(x)=B_{ij}g(x),
			\qquad
			A_2^{ij}g(x)=C_{ij}g(x).
		\end{equation}
		Consequently, whenever the order-two expansion holds locally,
		\begin{equation}\label{eq:second-order}
			P_\Delta^{ij}g(x)
			=\delta_{ij}g(x)+\Delta B_{ij}g(x)
			+\frac{\Delta^2}{2}C_{ij}g(x)+o(\Delta^2),
			\qquad \Delta\downarrow0,
		\end{equation}
		{uniformly on compact sets.}}

	The transition-probability target corresponds to the non-compactly supported test function $g\equiv1$. The following cutoff lemma localizes $g\equiv1$ near a fixed compact set and transfers the expansion to regime transition probabilities.
	\begin{lemma}[Cutoff localization]\label[lemma]{lem:cutoff-localization}
		Fix \(k\ge1\) and a compact set \(K\subset\mathbb R^d\). Assume the coefficient hypotheses of \Cref{prop:kth} through order \(k\). Let
		$\chi\in C_c^{2k}(\mathbb R^d)$,
		$0\le\chi\le1$,
		and suppose that \(\chi=1\) on an open neighborhood of \(K\).
		Then, for every integer \(M\ge1\), there exists \(C_{K,M,\chi}<\infty\) such that
		\begin{equation}\label{eq:cutoff-small-time}
			\max_{i,j\in S}
			\sup_{x\in K}
			\left|
			P_t^{ij}1(x)-P_t^{ij}\chi(x)
			\right|
			\le
			C_{K,M,\chi}t^M,
			\qquad 0<t\le1.
		\end{equation}
		{Moreover, for each \(i,j\in S\) and \(m=0,\ldots,k\), the definition
			$A_m^{ij}1(x):=A_m^{ij}\chi(x)$, $x\in K$,
			is independent of the choice of \(\chi\in C_c^{2k}(\mathbb R^d)\) satisfying \(\chi=1\) on an open neighborhood of \(K\). The coefficients \(A_m^{ij}1\) satisfy
			\begin{equation}\label{eq:transition-probability-expansion}
				P_t^{ij}1(x)
				=
				\sum_{m=0}^k
				\frac{t^m}{m!}A_m^{ij}1(x)
				+
				o(t^k)
			\end{equation}
			uniformly for \(x\in K\).}
	\end{lemma}

	\begin{proof}
			Choose \(\rho>0\) so that \(\chi=1\) on the \(\rho\)-neighborhood of \(K\). The standard estimate obtained from the Burkholder--Davis--Gundy inequality and Gronwall's inequality gives, for every \(M\ge1\),
			\[
			\sup_{(x,i)\in K\times S}
			\mathbb E_{(x,i)}\!\left[\sup_{0\le s\le t}\|X_s-x\|^{2M}\right]
			\le C_{K,M}t^M,
			\qquad 0<t\le1.
			\]
			Consequently,
			\[
			0\le P_t^{ij}1(x)-P_t^{ij}\chi(x)
			\le \mathbb P_{(x,i)}\!\left(\sup_{s\le t}\|X_s-x\|\ge\rho\right)
			\le C_{K,M,\chi}t^M,
			\]
			which proves \eqref{eq:cutoff-small-time}. If two cutoffs equal one near \(K\), their difference and all its derivatives vanish there. Locality of the continuous part of \(\mathcal L\), together with the fact that its jump part preserves the continuous state, gives equality of their coefficients by induction on \(m\). Finally choose \(M>k\), apply \Cref{prop:kth} to either cutoff, and absorb the \(O(t^M)\) cutoff error into \(o(t^k)\).
		\end{proof}
		\paragraph{Constant probe convention.}
	{The fixed-mesh theory treats \(g\equiv1\) directly. Throughout the model-level and shrinking-mesh theory, every statement for \(g\equiv1\) that invokes a short-time expansion is understood through the cutoff localization in \Cref{lem:cutoff-localization}. We use this convention without repeating it in individual result statements.}

	\begingroup
	\subsection*{Primitive coefficients encoded by the first block-generator family}
	Fix \(x\in\mathbb R^d\), and choose \(\chi_x\in C_c^\infty(\mathbb R^d)\) that equals one on an open neighborhood of \(x\). For \(r,s\in\{1,\ldots,d\}\), define the localized coordinate and quadratic probes
	\[
	\begin{aligned}
	g_{r,x}(y)&:=\chi_x(y)y_r,
	&g_{rs,x}(y)&:=\chi_x(y)y_ry_s,\\
	\widetilde g_{r,x}(y)&:=\chi_x(y)(y_r-x_r),
	&\widetilde g_{rs,x}(y)&:=\chi_x(y)(y_r-x_r)(y_s-x_s).
	\end{aligned}
	\]

	\begin{proposition}[Primitive-coefficient identification]\label[proposition]{prop:primitive-identification}
		Under \Cref{ass:wellposedness}, the first block-generator coefficients identify the switching intensities, drift, and diffusion matrix at \(x\). More precisely, for \(i\ne j\),
		\[
		q_{ij}(x)=A_1^{ij}1(x),
		\qquad
		q_i(x)=-A_1^{ii}1(x),
		\]
		and, for \(r,s\in\{1,\ldots,d\}\),
		\begin{align*}
		b_r(x,i)
		&=A_1^{ii}g_{r,x}(x)+q_i(x)x_r
		=A_1^{ii}\widetilde g_{r,x}(x),\\
		a_{rs}(x,i)
		&=A_1^{ii}g_{rs,x}(x)+q_i(x)x_rx_s
		-x_sb_r(x,i)-x_rb_s(x,i)\\
		&=A_1^{ii}\widetilde g_{rs,x}(x).
		\end{align*}
		These identities are independent of the choice of \(\chi_x\). Consequently, the family of first block-generator coefficients generated by locally constant, linear, and quadratic probes point-identifies the primitive coefficient vector
		\[
		\left(
		\{q_{ij}(x):i\ne j\},
		\{b(x,i):i\in S\},
		\{a(x,i):i\in S\}
		\right).
		\]
	\end{proposition}

	\begin{proof}
		By \eqref{eq:Bops} and the cutoff construction in \Cref{lem:cutoff-localization},
		\[
		A_1^{ij}1(x)=q_{ij}(x),\quad i\ne j,
		\qquad
		A_1^{ii}1(x)=\mathscr L_i1(x)-q_i(x)=-q_i(x).
		\]
		Because \(\chi_x=1\) near \(x\), its derivatives do not contribute at \(x\). Hence
		\[
		\mathscr L_i g_{r,x}(x)=b_r(x,i),
		\qquad
		\mathscr L_i g_{rs,x}(x)
		=x_sb_r(x,i)+x_rb_s(x,i)+a_{rs}(x,i).
		\]
		Substitution in \(A_1^{ii}g=\mathscr L_i g-q_i g\) gives the two uncentered identities. For the centered probes,
		\[
		\widetilde g_{r,x}(x)=0,
		\quad
		\nabla\widetilde g_{r,x}(x)=e_r,
		\quad
		\widetilde g_{rs,x}(x)=0,
		\quad
		\nabla\widetilde g_{rs,x}(x)=0.
		\]
		Their local Hessians therefore give
		\(\mathscr L_i\widetilde g_{r,x}(x)=b_r(x,i)\)
		and
		\(\mathscr L_i\widetilde g_{rs,x}(x)=a_{rs}(x,i)\),
		proving the centered identities. Since the generator is local in the continuous coordinate and regime switching does not change that coordinate, every displayed value depends only on the probes in a neighborhood of \(x\). This also proves independence of \(\chi_x\).
	\end{proof}
	\endgroup

	\section{Fixed-mesh local-polynomial estimation}
	\label{sec:estimation}
	\begingroup
	{Fix \(\Delta>0\). We estimate \(P_\Delta^{ij}g(x)\) by degree-\(p\) local-polynomial regression with a regime-specific response. Let \(K:\mathbb R^d\to[0,\infty)\) be bounded, measurable, compactly supported, and symmetric. Assume that it integrates to one and is strictly positive on a nonempty open ball \(B_K\). Throughout this section, \((h_n)\) is a deterministic bandwidth sequence satisfying
	\begin{equation}\label{eq:fixed-mesh-bandwidth}
		h_n\to0,\qquad nh_n^d\to\infty,
	\end{equation}
	and \(K_{h_n}(z):=h_n^{-d}K(z/h_n)\). Fix a polynomial degree \(p\ge0\), write
	\[
	\mathcal A_p:=\{\alpha\in\mathbb N_0^d:|\alpha|\le p\},\qquad q_p:=|\mathcal A_p|=\binom{p+d}{d},
	\]
	and let \(\psi_p(u):=(u^\alpha)_{\alpha\in\mathcal A_p}\in\mathbb R^{q_p}\), \(u^\alpha:=\prod_{r=1}^d u_r^{\alpha_r}\). We order the multi-indices so that the constant monomial appears first, and denote by \(e_0\in\mathbb R^{q_p}\) the first canonical basis vector. For \(k=1,\ldots,n\), define \(R_k^{j,g}:=g(Y_k)\mathbf 1_{\{\Lambda_{t_k}=j\}}\).
	The local-polynomial fit is built from the localized design and response moments
	\begin{equation}\label{eq:lp-S}
		\widehat S_{n,h_n}^{i,(p)}(x):=\frac{1}{nh_n^d}\sum_{k=1}^n
		K\!\left(\frac{Y_{k-1}-x}{h_n}\right)\mathbf 1_{\{\Lambda_{t_{k-1}}=i\}}
		\psi_p\!\left(\frac{Y_{k-1}-x}{h_n}\right)
		\psi_p\!\left(\frac{Y_{k-1}-x}{h_n}\right)^\top,
	\end{equation}
	\begin{equation}\label{eq:lp-T}
		\widehat T_{n,\Delta,h_n}^{ij,g,(p)}(x):=\frac{1}{nh_n^d}\sum_{k=1}^n
		K\!\left(\frac{Y_{k-1}-x}{h_n}\right)
		\mathbf 1_{\{\Lambda_{t_{k-1}}=i\}}
		R_k^{j,g}
		\psi_p\!\left(\frac{Y_{k-1}-x}{h_n}\right).
	\end{equation}
	\begin{definition}[Fixed-mesh local-polynomial estimator]
		On the event that \(\widehat S_{n,h_n}^{i,(p)}(x)\) is invertible, define
		\begin{equation}\label{eq:lp-est}
			\widehat P_{\Delta,h_n}^{ij}g(x)
			:=
			e_0^\top\bigl(\widehat S_{n,h_n}^{i,(p)}(x)\bigr)^{-1}
			\widehat T_{n,\Delta,h_n}^{ij,g,(p)}(x).
		\end{equation}
	\end{definition}
	The identifying conditional-mean relation is
	\begin{equation}\label{eq:lp-identification}
		\mathbb E\!\left[
		R_k^{j,g}
		\,\middle|\,
		Y_{k-1}=x,\Lambda_{t_{k-1}}=i
		\right]
		=
		P_\Delta^{ij}g(x).
	\end{equation}
	Thus the estimator targets the block functional itself, not the product of the block functional with the stationary regime probability.
	}

	The current-regime indicator in \eqref{eq:lp-S}--\eqref{eq:lp-T} is essential. Without it, the regression function is
	\(
	\pi_i(x)P_\Delta^{ij}g(x)
	\), where
	\(
	\pi_i(x)=\varpi_i(x)/\sum_{\ell\in S}\varpi_\ell(x)
	\), rather than the block functional itself. Apart from this regime-specific design and the block response, the estimator is an ordinary local-polynomial regression estimator for a stationary mixing sequence; compare \citet{robinson1983,masry1996,masryfan1997}.

	{The dependence on \(p\) is suppressed in the estimator notation. For the stochastic expansion it is convenient to introduce the deterministic surrogate
	\begin{equation}\label{eq:lp-surrogate}
		R_{\Delta,h_n,p}^{ij}g(x):=e_0^\top\bigl(\mathbb E[\widehat S_{n,h_n}^{i,(p)}(x)]\bigr)^{-1}
		\mathbb E[\widehat T_{n,\Delta,h_n}^{ij,g,(p)}(x)].
	\end{equation}}

	Write
	\[
	M_p(K):=\int_{\mathbb R^d}K(u)\psi_p(u)\psi_p(u)^\top\,du,\qquad
	\eta_p(K):=\int_{\mathbb R^d}K(u)\psi_p(u)\,du.
	\]
	The kernel conditions imply that \(M_p(K)\) is positive definite. Thus the moment matrix is invertible for every polynomial degree \(p\).
	Since  \(\psi_p(u)^\top e_0=1\), we have
	$M_p(K)e_0
	=\eta_p(K)$.
	{Define the squared-kernel moment matrix
		\[
		Q_p(K):=\int_{\mathbb R^d}K(u)^2\psi_p(u)\psi_p(u)^\top\,du.
		\]
		With the conditional variance \(v_{\Delta,g}^{ij}\) from \Cref{ass:fixedmesh-smoothness}, the scalar asymptotic variance is
		\begin{equation}\label{eq:fixedmesh-scalar-variance}
			\sigma_{ij,\Delta,p}^{2,g}(x)
			:=
			\frac{v_{\Delta,g}^{ij}(x)}{\varpi_i(x)}
			e_0^\top M_p(K)^{-1}Q_p(K)M_p(K)^{-1}e_0.
	\end{equation}}

	\begingroup
	\begin{theorem}[Fixed-mesh consistency]\label{thm:lp-fixedmesh-consistency}
		Fix \(\Delta>0\), \(p\ge0\), \(g\in C_b(\mathbb R^d)\), a regime pair \((i,j)\), and a design point \(x\in\mathbb R^d\). Suppose \Cref{ass:wellposedness,ass:mixing}, part~\textup{(i)} of \Cref{ass:local-densities}, and part~\textup{(i)} of \Cref{ass:fixedmesh-smoothness} hold on \(U_x\). Then
		$\mathbb P\!\left(\widehat S_{n,h_n}^{i,(p)}(x)\ \text{is invertible}\right)\to1$
		and
		$\widehat P_{\Delta,h_n}^{ij}g(x)
		\xrightarrow{\mathbb P}P_\Delta^{ij}g(x)$.
	\end{theorem}

	\begin{theorem}[Fixed-mesh asymptotic normality]\label{thm:lp-fixedmesh-clt}
		Fix \(\Delta>0\), \(p\ge0\), \(g\in C_b(\mathbb R^d)\), a regime pair \((i,j)\), and a design point \(x\in\mathbb R^d\). Suppose \Cref{ass:wellposedness,ass:mixing}, parts~\textup{(i)}--\textup{(ii)} of \Cref{ass:local-densities}, and parts~\textup{(ii)}--\textup{(iv)} of \Cref{ass:fixedmesh-smoothness} hold on \(U_x\). If
		$\frac{(\log n)^2}{nh_n^d}\to0$, and
		$\sqrt{nh_n^d}\,h_n^{p+1}\to0$,
		then
		\[
		\sqrt{n h_n^d}
		\left(\widehat P_{\Delta,h_n}^{ij}g(x)-P_\Delta^{ij}g(x)\right)
		\xrightarrow{d}
		N\!\left(0,\sigma_{ij,\Delta,p}^{2,g}(x)\right).
		\]
	\end{theorem}
	\endgroup

	{Because the fixed-mesh proof does not involve the shrinking-mesh cancellations, we state its probabilistic input as Lemma~\ref{lem:appendix-kernel-clt} and place the proof of that lemma and the complete fixed-mesh calculations in Sections~S.1--S.2 of the online supplement. At fixed \(\Delta\), the sampled one-step blocks form a stationary absolutely regular sequence with exponentially decaying coefficients. After verifying the required moment and covariance bounds, the argument follows the standard localized law of large numbers, triangular-array central limit theorem, and smoothing-bias analysis of \citet{robinson1983,masry1996,masryfan1997}.} For the main stochastic expansion, put
		\[
		S_{h_n,p}^{i}(x):=\mathbb E[\widehat S_{n,h_n}^{i,(p)}(x)],
		\qquad
		m_{\Delta,g}^{ij}(x):=P_\Delta^{ij}g(x),
		\]
		and define
		\[
		\begin{aligned}
			\zeta_{k,n}^{ij,g,(p)}(x)
			&:=
			K\!\left(\frac{Y_{k-1}-x}{h_n}\right)
			\mathbf 1_{\{\Lambda_{t_{k-1}}=i\}}
			e_0^\top\bigl(S_{h_n,p}^{i}(x)\bigr)^{-1}
			\psi_p\!\left(\frac{Y_{k-1}-x}{h_n}\right)\\
			&\qquad\times
			\left[
			R_k^{j,g}-m_{\Delta,g}^{ij}(x)
			\right].
		\end{aligned}
		\]

	\begin{proof}[Proof sketch of \Cref{thm:lp-fixedmesh-consistency}]
		By stationarity and the kernel change of variables,
		\[
		\mathbb E[\widehat S_{n,h_n}^{i,(p)}(x)]
		\to \varpi_i(x)M_p(K),
		\qquad
		\mathbb E[\widehat T_{n,\Delta,h_n}^{ij,g,(p)}(x)]
		\to
		\varpi_i(x)P_\Delta^{ij}g(x)\eta_p(K).
		\]
		The localized law of large numbers proved in Section~S.2.1 of the supplement makes both centered empirical moments \(O_{\mathbb P}((nh_n^d)^{-1/2})=o_{\mathbb P}(1)\). Since \(\varpi_i(x)M_p(K)\) is positive definite and \(M_p(K)e_0=\eta_p(K)\), the continuous mapping theorem gives invertibility with probability tending to one and the asserted consistency.
	\end{proof}

		\begin{proof}[Proof sketch of \Cref{thm:lp-fixedmesh-clt}]
			\Cref{lem:appendix-kernel-clt} applies to the centered array
			\[
			\zeta_{k,n}^{ij,g,(p)}(x)
			-\mathbb E[\zeta_{k,n}^{ij,g,(p)}(x)].
			\]
			Supplementary Section~S.2.2 verifies its hypotheses, including the near-lag small-ball estimate and the far-lag mixing bound. The asserted scalar convergence follows. The inverse-matrix linearization in Supplementary Section~S.2.4 then yields
		\begin{equation}\label{eq:fixedmesh-main-stochastic-expansion}
		\begin{aligned}
		\widehat P_{\Delta,h_n}^{ij}g(x)
		-R_{\Delta,h_n,p}^{ij}g(x)
		=
		\frac{1}{nh_n^d}
		\sum_{k=1}^n
		\left\{
		\zeta_{k,n}^{ij,g,(p)}(x)
		-\mathbb E[\zeta_{k,n}^{ij,g,(p)}(x)]
		\right\}
		+o_{\mathbb P}\!\left((nh_n^d)^{-1/2}\right).
		\end{aligned}
		\end{equation}
		The zero-lag conditional variance in this array converges to \eqref{eq:fixedmesh-scalar-variance}, whereas the summed nonzero-lag covariances are \(o(h_n^d)\). Finally, Section~S.2.3 proves
		\(
		R_{\Delta,h_n,p}^{ij}g(x)-P_\Delta^{ij}g(x)=O(h_n^{p+1})
		\).
		The undersmoothing condition and Slutsky's theorem therefore give the stated limit. Section~S.2.4 contains the complete matrix and remainder calculation.
	\end{proof}

	The normalized conditional moment follows from a joint two-response limit. Put
	\[
	\pi_{\Delta}^{ij}(x):=P_\Delta^{ij}1(x),
	\qquad
	\mu_{\Delta,g}^{ij}(x):=P_\Delta^{ij}g(x),
	\qquad
	M_\Delta^{ij}g(x):=\frac{\mu_{\Delta,g}^{ij}(x)}{\pi_\Delta^{ij}(x)},
	\]
	and define
	\[
	\kappa_p(K):=e_0^\top M_p(K)^{-1}Q_p(K)M_p(K)^{-1}e_0.
	\]
	For \(\pi_\Delta^{ij}(x)>0\), set
	\[
	\Omega_{\Delta,g}^{ij}(x)
	:=
	\begin{pmatrix}
	P_\Delta^{ij}(g^2)(x)-\{\mu_{\Delta,g}^{ij}(x)\}^2
	& \mu_{\Delta,g}^{ij}(x)\{1-\pi_\Delta^{ij}(x)\}\\
	\mu_{\Delta,g}^{ij}(x)\{1-\pi_\Delta^{ij}(x)\}
	& \pi_\Delta^{ij}(x)\{1-\pi_\Delta^{ij}(x)\}
	\end{pmatrix},
	\]
	\[
	\Sigma_{\Delta,g}^{ij}(x)
	:=\frac{\kappa_p(K)}{\varpi_i(x)}\Omega_{\Delta,g}^{ij}(x),
	\qquad
	\tau_{ij,\Delta,g}^2(x)
	:=
	\frac{\kappa_p(K)}{\varpi_i(x)\{\pi_\Delta^{ij}(x)\}^2}
	P_\Delta^{ij}\!\left[\{g-M_\Delta^{ij}g(x)\}^2\right](x).
	\]
	On the event that the common design matrix is invertible and
	\(\widehat P_{\Delta,h_n}^{ij}1(x)>0\), define
	\begin{equation}\label{eq:fixedmesh-normalized-estimator}
	\widehat M_{\Delta,h_n}^{ij}g(x)
	:=
	\frac{\widehat P_{\Delta,h_n}^{ij}g(x)}
	{\widehat P_{\Delta,h_n}^{ij}1(x)}.
	\end{equation}

	\begin{theorem}[Fixed-mesh normalized conditional moment]\label{thm:fixedmesh-normalized-moment}
		Fix \(\Delta>0\), \(p\ge0\), \(g\in C_b(\mathbb R^d)\), a regime pair \((i,j)\), and a design point \(x\). Suppose \Cref{ass:wellposedness,ass:mixing} and parts~\textup{(i)}--\textup{(ii)} of \Cref{ass:local-densities} hold on \(U_x\). Suppose parts~\textup{(ii)}--\textup{(iii)} of \Cref{ass:fixedmesh-smoothness} hold for \(g\), and part~\textup{(ii)} holds for \(1\). Assume \(\pi_\Delta^{ij}(x)>0\), \(\tau_{ij,\Delta,g}^2(x)>0\), and
		$\frac{(\log n)^2}{nh_n^d}\to0$,
		$\sqrt{nh_n^d}\,h_n^{p+1}\to0$.
		Then
		\[
		\sqrt{nh_n^d}
		\left[
		\begin{pmatrix}
		\widehat P_{\Delta,h_n}^{ij}g(x)\\
		\widehat P_{\Delta,h_n}^{ij}1(x)
		\end{pmatrix}
		-
		\begin{pmatrix}
		\mu_{\Delta,g}^{ij}(x)\\
		\pi_\Delta^{ij}(x)
		\end{pmatrix}
		\right]
		\xrightarrow{d}
		N_2\!\left(0,\Sigma_{\Delta,g}^{ij}(x)\right).
		\]
		Moreover,
		\[
		\sqrt{nh_n^d}
		\left\{\widehat M_{\Delta,h_n}^{ij}g(x)-M_\Delta^{ij}g(x)\right\}
		\xrightarrow{d}
		N\!\left(0,\tau_{ij,\Delta,g}^2(x)\right).
		\]
	\end{theorem}

	\begin{proof}[Proof sketch]
		For \(c=(c_1,c_2)^\top\), the two-response projection equals
		\(c_1R_k^{j,g}+c_2R_k^{j,1}=R_k^{j,c_1g+c_2}\). Applying \eqref{eq:fixedmesh-main-stochastic-expansion} to the combined probe and then the Cram\'er--Wold device gives the joint limit; direct conditional covariance calculation gives \(\Omega_{\Delta,g}^{ij}(x)\). Consistency for the probe \(1\) and \(\pi_\Delta^{ij}(x)>0\) make the ratio well defined with probability tending to one. The multivariate delta method for \((u,v)\mapsto u/v\) gives the last display. Section~S.2.4 of the supplement contains the full projection, covariance, and delta-method calculation.
	\end{proof}
	\endgroup

	\section{Shrinking-mesh recovery of block-generator coefficients}
	\label{sec:recovery}
	{The inferential focus of this section is recovery of the first two block-generator coefficients under \(\Delta_n\downarrow0\). The section proceeds directly to the first- and second-order central limit theorems in \Cref{thm:first-order-recovery-clt,thm:second-order-recovery-clt}; the preliminary arbitrary-order separate-lag theory is placed in the online supplement.} {When the polynomial degree must be emphasized, we write
	$\widehat P_{\Delta,h_n}^{ij,(p)}g(x)$
	for the estimator in \eqref{eq:lp-est} constructed from the basis \(\psi_p\).} Throughout this section, the observations form a triangular array:
	\(t_{k,n}=k\Delta_n\),
	\(Z_{k,n}:=(Y_{k,n},\Lambda_{t_{k,n}})\), \(k=0,\ldots,n\).
	The shrinking mesh and bandwidth satisfy
	\begin{equation}\label{eq:recovery-mesh-bandwidth}
		\Delta_n\to0,\qquad h_n\to0,\qquad
		n\Delta_n\to\infty,\qquad n\Delta_n h_n^d\to\infty.
	\end{equation}
	\begingroup
	We first record the localization consequence of exponential \(V\)-ergodicity that is used throughout the shrinking-mesh proofs. It supplies the factor \(h_n^d\) that cannot be obtained from an unweighted mixing inequality alone.
	\begin{lemma}[Localized \(V\)-norm covariance bound]\label[lemma]{lem:shrinking-localization-vnorm}
		Suppose \Cref{ass:mixing} and part~\textup{(i)} of \Cref{ass:local-densities} hold on \(U_x\). Let \(\phi_n:E\to\mathbb R\) satisfy \(|\phi_n|\le C\), with a constant independent of \(n\), and
		\(\phi_n(y,\ell)=0\) unless
		\(\ell=i\) and \(y\in x+h_nD\),
		where \(D\subset\mathbb R^d\) is a fixed compact set. Then, for every \(m\ge1\),
		\[
		\left|\operatorname{Cov}\bigl(\phi_n(Z_{0,n}),\phi_n(Z_{m,n})\bigr)\right|
		\lesssim h_n^d e^{-\lambda_{\rm erg}m\Delta_n}.
		\]
		Consequently,
		\[
		\operatorname{Var}\!\left(
		\frac{1}{nh_n^d}\sum_{s=0}^{n-1}\phi_n(Z_{s,n})
		\right)
		\lesssim \frac{1}{n\Delta_n h_n^d}.
		\]
	\end{lemma}

	\begin{proof}
		Write \(\bar\phi_n:=\phi_n-\nu(\phi_n)\). By the Markov property and stationarity,
		\[
		\operatorname{Cov}\bigl(\phi_n(Z_{0,n}),\phi_n(Z_{m,n})\bigr)
		=
		\mathbb E\!\left[
		\bar\phi_n(Z_{0,n})
		\{P_{m\Delta_n}\phi_n(Z_{0,n})-\nu(\phi_n)\}
		\right].
		\]
		Since \(V\ge1\), \(\|\phi_n\|_{V,\infty}:=\sup_z|\phi_n(z)|/V(z)\le C\). The definition of the \(V\)-norm and \Cref{ass:mixing} therefore give
		\[
		\left|P_{m\Delta_n}\phi_n(z)-\nu(\phi_n)\right|
		\le
		\|\phi_n\|_{V,\infty}
		\|P_{m\Delta_n}(z,\cdot)-\nu\|_V
		\lesssim e^{-\lambda_{\rm erg}m\Delta_n}V(z).
		\]
		Moreover, the support restriction, the local boundedness of \(V\) in \Cref{ass:mixing}, and {the stationary density in part~\textup{(i)} of \Cref{ass:local-densities}} imply
		\[
		\mathbb E\!\left[|\phi_n(Z_{0,n})|V(Z_{0,n})\right]
		\lesssim
		\int_{x+h_nD}\varpi_i(y)V(y,i)\,dy
		\lesssim h_n^d.
		\]
		Also \(|\nu(\phi_n)|\lesssim h_n^d\) and \(\nu(V)<\infty\), so
		\(\mathbb E[|\bar\phi_n(Z_{0,n})|V(Z_{0,n})]\lesssim h_n^d\).
		Combining the preceding estimates proves the covariance bound. The zero-lag variance is \(O(h_n^d)\); hence the stationary variance decomposition and
		\[
		\sum_{m\ge1}e^{-\lambda_{\rm erg}m\Delta_n}\lesssim\Delta_n^{-1}
		\]
		yield the stated variance estimate.
	\end{proof}
	\endgroup

	{For an arbitrary fixed order \(k\), let \(\widehat P_{r\Delta_n,h_n}^{ij,(p)}g(x)\) denote the degree-\(p\) local-polynomial estimator formed from the separately sampled \(r\)-step blocks, \(1\le r\le k\), and set \(\widehat P_{0,h_n}^{ij,(p)}g(x):=\delta_{ij}g(x)\). With \(c_{k,r}:=(-1)^{k-r}\binom{k}{r}\), define
	\[
	\widehat A_{k,n}^{(p),ij}g(x)
	:=
	\frac{1}{\Delta_n^k}
	\sum_{r=0}^k c_{k,r}\widehat P_{r\Delta_n,h_n}^{ij,(p)}g(x).
	\]
	Under sufficient conditions stated in Section~S.3 of the online supplement, this estimator is consistent:
	\[
	\widehat A_{k,n}^{(p),ij}g(x)
	\xrightarrow{\mathbb P}
	A_k^{ij}g(x).
	\]
	The precise construction and proof are in Supplementary Section~S.3. The main paper now focuses on the first two orders and their common-design central limit theorems.}

		\subsection{First-order recovery and studentization}

		Throughout this subsection, let \(\widehat S_n^{i,(p)}(x)\) denote the design matrix in \eqref{eq:lp-S}, evaluated on the triangular observations \(Z_{s,n}\), \(s=0,\ldots,n-1\).

		For \(\Delta>0\), define the normalized first-order block
		\[
		b_{\Delta,g}^{ij}(y)
		:=
		\frac{P_\Delta^{ij}g(y)-\delta_{ij}g(y)}{\Delta},
		\]
		and define the local infinitesimal variance
		\begin{equation}\label{eq:first-order-gamma}
		\gamma_{ij,g}(y)
		:=
		\begin{cases}
		\nabla g(y)^\top a(y,i)\nabla g(y)+q_i(y)g(y)^2, & i=j,\\[1mm]
		q_{ij}(y)g(y)^2, & i\ne j.
		\end{cases}
		\end{equation}
		This is the carré-du-champ coefficient
		\(\mathcal L\{(f_g^{(j)})^2\}(y,i)-2f_g^{(j)}(y,i)\mathcal Lf_g^{(j)}(y,i)\).

		\begin{assumption}[First-order local smoothness]\label[assumption]{ass:first-order-smoothness}
			For the polynomial degree \(p\), regime pair \((i,j)\), test function \(g\), and neighborhood \(U_x\),
			\[
			b_{\Delta_n,g}^{ij}\in C^{p+1}(U_x),
			\qquad
			\sup_{n\ge1}\left\|b_{\Delta_n,g}^{ij}\right\|_{C^{p+1}(U_x)}<\infty.
			\]
		\end{assumption}

		\begin{lemma}[First-order short-time moments]\label[lemma]{lem:first-order-short-time-moments}
			Suppose \Cref{ass:wellposedness} holds and either \(g\in C_c^4(\mathbb R^d)\) or \(g\equiv1\). Then, uniformly on \(U_x\),
			\begin{align}
			b_{\Delta,g}^{ij}(y)-B_{ij}g(y)&=O(\Delta),\label{eq:first-order-temporal-bias}\\
			\frac{P_\Delta^{ij}(g^2)(y)-\{P_\Delta^{ij}g(y)\}^2}{\Delta}
			&\to \gamma_{ij,g}(y),\label{eq:first-order-variance-limit}\\
			\kappa_{\Delta,g}^{ij}&\to\gamma_{ij,g}, \notag
			\end{align}
			where
			\[
			\kappa_{\Delta,g}^{ij}(y)
			:=
			\frac{1}{\Delta}
			\mathbb E_{(y,i)}\!\left[
			\left\{f_g^{(j)}(Z_\Delta)-f_g^{(j)}(y,i)\right\}^2
			\right].
			\]
		\end{lemma}

		\begin{proof}
			For \(g\in C_c^4(\mathbb R^d)\), apply \Cref{prop:kth} with \(k=2\). For \(g\equiv1\), apply the cutoff argument in \Cref{lem:cutoff-localization} through order two. In either case, uniformly on \(\overline U_x\),
			\[
			P_\Delta^{ij}g
			=\delta_{ij}g+\Delta B_{ij}g+\frac{\Delta^2}{2}C_{ij}g+o(\Delta^2).
			\]
			Since \(C_{ij}g\) is bounded on \(\overline U_x\), division by \(\Delta\) proves \eqref{eq:first-order-temporal-bias}.

			Apply the corresponding order-one expansion to \(g^2\), using \Cref{prop:kth} for compactly supported \(g\) and \Cref{lem:cutoff-localization} for \(g\equiv1\). If \(i=j\), then
			\[
			\begin{aligned}
			\frac{P_\Delta^{ii}(g^2)-\{P_\Delta^{ii}g\}^2}{\Delta}
			\to
			B_{ii}(g^2)-2gB_{ii}g
			=\nabla g^\top a(\cdot,i)\nabla g+q_i g^2.
			\end{aligned}
			\]
			If \(i\ne j\), then
			\[
			P_\Delta^{ij}(g^2)=\Delta q_{ij}g^2+o(\Delta),
			\qquad
			\{P_\Delta^{ij}g\}^2=O(\Delta^2),
			\]
			which gives the second line of \eqref{eq:first-order-gamma}. This proves \eqref{eq:first-order-variance-limit}.

			Let \(J_\Delta\) be the event that at least one regime switch occurs on \((0,\Delta]\). The bounded total switching rate gives
			\(\sup_{y\in U_x}\mathbb P_{(y,i)}(J_\Delta)\lesssim\Delta\). On \(J_\Delta\), the fourth power of the block increment is bounded by \(16\|g\|_\infty^4\). On \(J_\Delta^c\), the increment vanishes when \(i\ne j\). When \(i=j\), couple \(X\) with the fixed-regime-\(i\) diffusion driven by the same Wiener process; the two paths agree up to the first switch. Itô's formula for the fixed-regime diffusion, the boundedness of \(\mathscr L_i g\) and \(\nabla g^\top\sigma(\cdot,i)\), and the Burkholder--Davis--Gundy inequality give
			\[
			\sup_{y\in U_x}
			\mathbb E_{(y,i)}\!\left[
			|g(X_\Delta)-g(y)|^4\mathbf 1_{J_\Delta^c}
			\right]
			\lesssim\Delta^2.
			\]
			Combining the switching and no-switch contributions proves the following fourth-moment bound: with \(z=(y,i)\),
			\[
			\sup_{y\in U_x}
			\mathbb E_z\!\left[
			\left|f_g^{(j)}(Z_\Delta)-f_g^{(j)}(z)\right|^4
			\right]
			\lesssim\Delta.
			\]
			Finally,
			\[
			\kappa_{\Delta,g}^{ij}
			=
			\frac{P_\Delta^{ij}(g^2)-\{P_\Delta^{ij}g\}^2}{\Delta}
			+\Delta\{b_{\Delta,g}^{ij}\}^2.
			\]
			The first term converges uniformly to \(\gamma_{ij,g}\), while the second is \(O(\Delta)\) by \eqref{eq:first-order-temporal-bias}. This proves the last assertion.
		\end{proof}

		For \(s=0,\ldots,n-1\), set
		$D_{s,n}^{ij,g}
		:=
		f_g^{(j)}(Z_{s+1,n})-f_g^{(j)}(Z_{s,n})$.
		On \(\{\Lambda_{t_{s,n}}=i\}\), this common-design response is exactly
		\[
		D_{s,n}^{ij,g}
		=
		g(Y_{s+1,n})\mathbf 1_{\{\Lambda_{t_{s+1,n}}=j\}}
		-\delta_{ij}g(Y_{s,n}),
		\]
		which is the first forward difference of the terminal-regime block response.
		\begingroup
		Define
		\begin{equation}\label{eq:first-order-weighted-innovation}
		\begin{aligned}
		\xi_{s,n}^{ij,g}
		&:=
		\mathbf 1_{\{\Lambda_{t_{s,n}}=i\}}
		\left[
		f_g^{(j)}(Z_{s+1,n})
		-P_{\Delta_n}f_g^{(j)}(Z_{s,n})
		\right]\\
		&=
		\mathbf 1_{\{\Lambda_{t_{s,n}}=i\}}
		\left[
		g(Y_{s+1,n})\mathbf 1_{\{\Lambda_{t_{s+1,n}}=j\}}
		-P_{\Delta_n}^{ij}g(Y_{s,n})
		\right].
		\end{aligned}
		\end{equation}
		Then \(\xi_{s,n}^{ij,g}\) is \(\mathcal F_{t_{s+1,n}}\)-measurable and
		\(\mathbb E[\xi_{s,n}^{ij,g}\mid\mathcal F_{t_{s,n}}]=0\). Moreover,
		\begin{equation}\label{eq:first-order-weighted-decomposition}
		\mathbf 1_{\{\Lambda_{t_{s,n}}=i\}}D_{s,n}^{ij,g}
		=
		\Delta_n\mathbf 1_{\{\Lambda_{t_{s,n}}=i\}}
		b_{\Delta_n,g}^{ij}(Y_{s,n})
		+\xi_{s,n}^{ij,g}.
		\end{equation}
		\endgroup
		On the event that \(\widehat S_n^{i,(p)}(x)\) is invertible, define the first-order estimator
		\begin{equation}\label{eq:first-order-estimator}
		\widehat B_n^{ij,(p)}g(x)
		:=
		e_0^\top\left(\widehat S_n^{i,(p)}(x)\right)^{-1}
		\frac{1}{nh_n^d\Delta_n}
		\sum_{s=0}^{n-1}
		K\!\left(\frac{Y_{s,n}-x}{h_n}\right)
		\mathbf 1_{\{\Lambda_{t_{s,n}}=i\}}
		D_{s,n}^{ij,g}
		\psi_p\!\left(\frac{Y_{s,n}-x}{h_n}\right).
		\end{equation}
			Thus \eqref{eq:first-order-estimator} is the ordinary first forward difference after applying one common local-polynomial fit to the lag-zero and lag-one responses.

		Define
		\begin{equation}\label{eq:first-order-asymptotic-variance}
		\sigma_{B,ij,p}^{2,g}(x)
		:=
		\frac{\gamma_{ij,g}(x)}{\varpi_i(x)}
		e_0^\top M_p(K)^{-1}Q_p(K)M_p(K)^{-1}e_0.
		\end{equation}

		\begin{theorem}[First-order recovery central limit theorem]\label{thm:first-order-recovery-clt}
			Fix \(p\ge0\), a regime pair \((i,j)\), a test function \(g\in C_c^4(\mathbb R^d)\) or \(g\equiv1\), and a design point \(x\). Suppose \Cref{ass:wellposedness,ass:mixing}, part~\textup{(i)} of \Cref{ass:local-densities}, and \Cref{ass:first-order-smoothness} hold on \(U_x\). {Under the standing mesh and bandwidth conditions \eqref{eq:recovery-mesh-bandwidth},}
			$\widehat B_n^{ij,(p)}g(x)\xrightarrow{\mathbb P}B_{ij}g(x)$.
			If, in addition,
			\begin{equation}\label{eq:first-order-clt-rates}
			\sqrt{n\Delta_nh_n^d}\bigl(\Delta_n+h_n^{p+1}\bigr)\to0,
			\end{equation}
			then
			\[
			\sqrt{n\Delta_nh_n^d}
			\left\{\widehat B_n^{ij,(p)}g(x)-B_{ij}g(x)\right\}
			\xrightarrow{d}
			N\!\left(0,\sigma_{B,ij,p}^{2,g}(x)\right).
			\]
		\end{theorem}

		\begin{proof}
			Write \(\widehat S_n:=\widehat S_n^{i,(p)}(x)\). By the kernel change of variables,
			\[
			\mathbb E[\widehat S_n]\to\varpi_i(x)M_p(K).
			\]
			Applying \Cref{lem:shrinking-localization-vnorm} coordinatewise gives
			\(\widehat S_n-\mathbb E[\widehat S_n]=O_{\mathbb P}((n\Delta_nh_n^d)^{-1/2})\). Hence \(\widehat S_n\) is invertible with probability tending to one and
			\begin{equation}\label{eq:first-order-design-limit}
			\widehat S_n^{-1}\xrightarrow{\mathbb P}
			\varpi_i(x)^{-1}M_p(K)^{-1}.
			\end{equation}

			Taylor's theorem and \Cref{ass:first-order-smoothness} give, uniformly for \(u\in\operatorname{supp}(K)\),
			\[
			b_{\Delta_n,g}^{ij}(x+h_nu)
			=\psi_p(u)^\top\beta_n+b_{n}^{\rm rem}(u),
			\qquad
			\sup_{u\in\operatorname{supp}(K)}|b_n^{\rm rem}(u)|\lesssim h_n^{p+1},
			\]
			where the intercept of \(\beta_n\) is \(b_{\Delta_n,g}^{ij}(x)\). Since \(\widehat S_n^{-1}=O_{\mathbb P}(1)\), substitution of \eqref{eq:first-order-weighted-decomposition} into \eqref{eq:first-order-estimator} yields
			\begin{equation}\label{eq:first-order-linearization}
			\begin{aligned}
			\widehat B_n^{ij,(p)}g(x)-B_{ij}g(x)
			&=
			e_0^\top\widehat S_n^{-1}
			\frac{1}{nh_n^d\Delta_n}
			\sum_{s=0}^{n-1}
			K\!\left(\frac{Y_{s,n}-x}{h_n}\right)
			\psi_p\!\left(\frac{Y_{s,n}-x}{h_n}\right)
			\xi_{s,n}^{ij,g}\\
			&\quad+O_{\mathbb P}(h_n^{p+1}+\Delta_n),
			\end{aligned}
			\end{equation}
			where \eqref{eq:first-order-temporal-bias} controls the temporal approximation.

			By \eqref{eq:first-order-variance-limit},
			\[
			\mathbb E\!\left[
			(\xi_{s,n}^{ij,g})^2
			\,\middle|\,\mathcal F_{t_{s,n}}
			\right]
			=
			\mathbf 1_{\{\Lambda_{t_{s,n}}=i\}}
			v_{\Delta_n,g}^{ij}(Y_{s,n})
			=O(\Delta_n)
			\]
			uniformly on the kernel support. Localization and martingale orthogonality therefore give
			\[
			\mathbb E\!\left[
			\left\|
			\frac{1}{nh_n^d\Delta_n}
			\sum_{s=0}^{n-1}
			K\!\left(\frac{Y_{s,n}-x}{h_n}\right)
			\psi_p\!\left(\frac{Y_{s,n}-x}{h_n}\right)
			\xi_{s,n}^{ij,g}
			\right\|^2
			\right]
			\lesssim\frac{1}{n\Delta_nh_n^d}.
			\]
			Hence the first term in \eqref{eq:first-order-linearization} is
			\(O_{\mathbb P}((n\Delta_nh_n^d)^{-1/2})\). This proves consistency.

			For the distributional limit, set
			\[
			\mathcal M_n
			:=
			\frac{1}{\sqrt{n\Delta_nh_n^d}}
			\sum_{s=0}^{n-1}
			K\!\left(\frac{Y_{s,n}-x}{h_n}\right)
			\psi_p\!\left(\frac{Y_{s,n}-x}{h_n}\right)
			\xi_{s,n}^{ij,g}.
			\]
			Its predictable quadratic-variation matrix is
			\[
			\frac{1}{nh_n^d}
			\sum_{s=0}^{n-1}
			K\!\left(\frac{Y_{s,n}-x}{h_n}\right)^2
			\mathbf 1_{\{\Lambda_{t_{s,n}}=i\}}
			\psi_p\!\left(\frac{Y_{s,n}-x}{h_n}\right)
			\psi_p\!\left(\frac{Y_{s,n}-x}{h_n}\right)^\top
			\frac{v_{\Delta_n,g}^{ij}(Y_{s,n})}{\Delta_n},
			\]
			where \(v_{\Delta,g}^{ij}:=P_\Delta^{ij}(g^2)-\{P_\Delta^{ij}g\}^2\). By \eqref{eq:first-order-variance-limit}, its entries are uniformly bounded on the kernel support. The change-of-variables formula identifies their expectations, while \Cref{lem:shrinking-localization-vnorm} makes their centered empirical averages negligible. Hence the matrix converges in probability to $\varpi_i(x)\gamma_{ij,g}(x)Q_p(K)$.

			Fix \(a\in\mathbb R^{q_p}\), and let \(X_{s,n}(a)\) denote the \(s\)-th summand of \(a^\top\mathcal M_n\). Supplementary Section~S.4 verifies, for these projected summands, the shifted-filtration, maximum-increment, realized-square, and zero-variance conditions, and also records the conditional Lindeberg check. Lemma~\ref{lem:appendix-martingale-clt} yields the limit for each Cramér--Wold projection, and the Cramér--Wold device gives
			\[
			\mathcal M_n\xrightarrow{d}
			N\!\left(0,\varpi_i(x)\gamma_{ij,g}(x)Q_p(K)\right).
			\]
			Combining this limit with \eqref{eq:first-order-design-limit}, \eqref{eq:first-order-linearization}, and \eqref{eq:first-order-clt-rates} and applying Slutsky's theorem proves the claimed central limit theorem.
		\end{proof}

		\begingroup
		For \(i\ne j\), define the direct switching-intensity estimator
		\begin{equation}\label{eq:offdiagonal-intensity-estimator}
		\widehat q_{ij,n}^{(p)}(x)
		:=\widehat B_n^{ij,(p)}1(x)
		=\frac{\widehat P_{\Delta_n,h_n}^{ij,(p)}1(x)}{\Delta_n}.
		\end{equation}
		The equality holds because the lag-zero response is identically zero on the regime-\(i\) design when \(i\ne j\).

		\begin{corollary}[Off-diagonal switching-intensity inference]\label[corollary]{cor:offdiagonal-intensity}
			Fix \(p\ge0\), \(i\ne j\), and \(x\in\mathbb R^d\). Suppose \Cref{ass:wellposedness,ass:mixing}, part~\textup{(i)} of \Cref{ass:local-densities}, and \Cref{ass:first-order-smoothness} for \(g\equiv1\) hold on \(U_x\). {Under \eqref{eq:recovery-mesh-bandwidth},}
			\[
			\widehat q_{ij,n}^{(p)}(x)\xrightarrow{\mathbb P}q_{ij}(x).
			\]
			If, in addition, \eqref{eq:first-order-clt-rates} holds and \(q_{ij}(x)>0\), then
			\[
			\sqrt{n\Delta_nh_n^d}
			\left\{\widehat q_{ij,n}^{(p)}(x)-q_{ij}(x)\right\}
			\xrightarrow{d}
			N\!\left(
			0,
			\frac{q_{ij}(x)}{\varpi_i(x)}\kappa_p(K)
			\right).
			\]
		\end{corollary}

		\begin{proof}
			Apply \Cref{thm:first-order-recovery-clt} with \(g\equiv1\) and use \(B_{ij}1=q_{ij}\) and \(\gamma_{ij,1}=q_{ij}\) for \(i\ne j\).
		\end{proof}
		\endgroup

		For a fixed regime \(i\), retain \(\widehat q_{ij,n}^{(p)}\) from \eqref{eq:offdiagonal-intensity-estimator} and define the remaining primitive-coefficient estimators
		\[
		\begin{aligned}
		\widehat q_{i,n}^{(p)}(x)&:=-\widehat B_n^{ii,(p)}1(x),\\
		\widehat b_{r,n}^{(p)}(x,i)&:=\widehat B_n^{ii,(p)}\widetilde g_{r,x}(x),
		&&1\le r\le d,\\
		\widehat a_{rs,n}^{(p)}(x,i)&:=\widehat B_n^{ii,(p)}\widetilde g_{rs,x}(x),
		&&1\le r,s\le d,
		\end{aligned}
		\]
		where the localized centered probes are defined before \Cref{prop:primitive-identification}.

		\begin{corollary}[Consistent recovery of primitive coefficients]\label[corollary]{cor:primitive-recovery}
			Fix \(i\in S\), \(p\ge0\), and \(x\in\mathbb R^d\). Suppose \Cref{ass:wellposedness,ass:mixing} and part~\textup{(i)} of \Cref{ass:local-densities} hold on \(U_x\). Suppose also that \Cref{ass:first-order-smoothness} holds on \(U_x\) for \((i,j,g)=(i,j,1)\), \(j\in S\), and for \((i,i,\widetilde g_{r,x})\) and \((i,i,\widetilde g_{rs,x})\), \(1\le r,s\le d\). {Under \eqref{eq:recovery-mesh-bandwidth},} jointly over the finite collection of indices,
			\[
			\begin{aligned}
			\widehat q_{ij,n}^{(p)}(x)&\xrightarrow{\mathbb P}q_{ij}(x), &&j\ne i,\\
			\widehat q_{i,n}^{(p)}(x)&\xrightarrow{\mathbb P}q_i(x),\\
			\widehat b_{r,n}^{(p)}(x,i)&\xrightarrow{\mathbb P}b_r(x,i),
			&&1\le r\le d,\\
			\widehat a_{rs,n}^{(p)}(x,i)&\xrightarrow{\mathbb P}a_{rs}(x,i),
			&&1\le r,s\le d.
			\end{aligned}
			\]
		\end{corollary}

		\begin{proof}
			Apply the consistency assertion of \Cref{thm:first-order-recovery-clt} to \(g\equiv1\), \(\widetilde g_{r,x}\), and \(\widetilde g_{rs,x}\), and then use \Cref{prop:primitive-identification}. Joint convergence follows because the displayed collection is finite.
		\end{proof}

		Recall that
		\(\widetilde g_{r,x}(y)=\chi_x(y)(y_r-x_r)\) and
		\(\widetilde g_{rs,x}(y)=\chi_x(y)(y_r-x_r)(y_s-x_s)\)
		denote the centered linear and centered quadratic probes, respectively. For \(g\equiv1\) and \(g=\widetilde g_{r,x}\), \Cref{thm:first-order-recovery-clt} gives the corresponding scalar limits whenever the variance in \eqref{eq:first-order-gamma} is positive. For \(g=\widetilde g_{rs,x}\), one has \(g(x)=0\) and \(\nabla g(x)=0\), so \(\gamma_{ii,g}(x)=0\); a nondegenerate diffusion-matrix limit requires the next short-time variance term and a faster normalization.

		The variance formula displays the diagonal/off-diagonal distinction explicitly. For \(i=j\), the infinitesimal noise consists of the diffusion increment of \(g(X)\) and the loss of \(g(X)\) when the process exits regime \(i\). For \(i\ne j\), it is generated by the rare \(i\to j\) transition. Both mechanisms produce the first-order rate \(\sqrt{n\Delta_nh_n^d}\).
		{For the transition-probability probe \(g\equiv1\), the theorem directly estimates \(q_{ij}(x)\) when \(i\ne j\) and \(-q_i(x)\) when \(i=j\). In this case \(\gamma_{ij,1}(x)=q_{ij}(x)\) off the diagonal and \(\gamma_{ii,1}(x)=q_i(x)\) on the diagonal.}
			If \(\gamma_{ij,g}(x)=0\), the conclusion of \Cref{thm:first-order-recovery-clt} is interpreted as the degenerate limit \(N(0,0)\). In particular, for \(g\equiv1\), an off-diagonal point satisfying \(q_{ij}(x)=0\) has no nondegenerate studentized limit at the first-order normalization.

		Write \(\widehat S_n:=\widehat S_n^{i,(p)}(x)\). Define the local density and infinitesimal-variance estimators
		\[
		\begin{aligned}
		\widehat\varpi_{i,n}(x)
		&:=
		\frac{1}{nh_n^d}\sum_{s=0}^{n-1}
		K\!\left(\frac{Y_{s,n}-x}{h_n}\right)
		\mathbf 1_{\{\Lambda_{t_{s,n}}=i\}},\\
		\widehat\gamma_{ij,g,n}(x)
		&:=
		e_0^\top\widehat S_n^{-1}
		\frac{1}{nh_n^d\Delta_n}\sum_{s=0}^{n-1}
		K\!\left(\frac{Y_{s,n}-x}{h_n}\right)
		\mathbf 1_{\{\Lambda_{t_{s,n}}=i\}}
		\bigl(D_{s,n}^{ij,g}\bigr)^2
		\psi_p\!\left(\frac{Y_{s,n}-x}{h_n}\right),\\
		\widehat\sigma_{B,ij,p}^{2,g}(x)
		&:=
		\frac{\widehat\gamma_{ij,g,n}(x)}{\widehat\varpi_{i,n}(x)}
		e_0^\top M_p(K)^{-1}Q_p(K)M_p(K)^{-1}e_0.
		\end{aligned}
		\]
		When \(\widehat\sigma_{B,ij,p}^{2,g}(x)>0\), write
		\(\widehat\sigma_{B,ij,p}^{g}(x)
		:=\{\widehat\sigma_{B,ij,p}^{2,g}(x)\}^{1/2}\).

		\begin{corollary}[First-order studentization]\label[corollary]{cor:first-order-studentization}
			Under the assumptions of \Cref{thm:first-order-recovery-clt},
			\[
			\widehat\varpi_{i,n}(x)\xrightarrow{\mathbb P}\varpi_i(x),
			\qquad
			\widehat\gamma_{ij,g,n}(x)\xrightarrow{\mathbb P}\gamma_{ij,g}(x).
			\]
			If \(\gamma_{ij,g}(x)>0\) and \eqref{eq:first-order-clt-rates} holds, then
			\[
			\frac{\sqrt{n\Delta_nh_n^d}
			\{\widehat B_n^{ij,(p)}g(x)-B_{ij}g(x)\}}
			{\widehat\sigma_{B,ij,p}^{g}(x)}
			\xrightarrow{d}N(0,1).
			\]
		\end{corollary}

		\begin{proof}
			The change-of-variables formula gives
			\(\mathbb E[\widehat\varpi_{i,n}(x)]\to\varpi_i(x)\); \Cref{lem:shrinking-localization-vnorm} makes the centered term \(o_{\mathbb P}(1)\).
			\begingroup
			For the second estimator, define the regime-weighted innovation
			\[
			\eta_{s,n}
			:=
			\mathbf 1_{\{\Lambda_{t_{s,n}}=i\}}
			\left\{
			\frac{(D_{s,n}^{ij,g})^2}{\Delta_n}
			-\kappa_{\Delta_n,g}^{ij}(Y_{s,n})
			\right\}.
			\]
			By the Markov property and the definition of \(\kappa_{\Delta_n,g}^{ij}\),
			\(\eta_{s,n}\) is \(\mathcal F_{t_{s+1,n}}\)-measurable and
			\(\mathbb E[\eta_{s,n}\mid\mathcal F_{t_{s,n}}]=0\). Thus martingale orthogonality applies with the one-step shift from \(\mathcal F_{t_{s,n}}\) to \(\mathcal F_{t_{s+1,n}}\).
			\endgroup
			By the fourth-moment bound in \Cref{lem:first-order-short-time-moments},
			$\mathbb E\!\left[\eta_{s,n}^2\mid\mathcal F_{t_{s,n}}\right]
			\lesssim\Delta_n^{-1}$
			on the kernel support. Martingale orthogonality and localization therefore give
			\[
			\left\|
			\frac{1}{nh_n^d}\sum_{s=0}^{n-1}
			K\!\left(\frac{Y_{s,n}-x}{h_n}\right)
			\psi_p\!\left(\frac{Y_{s,n}-x}{h_n}\right)\eta_{s,n}
			\right\|
			=O_{\mathbb P}\!\left((n\Delta_nh_n^d)^{-1/2}\right).
			\]
			By the last assertion of \Cref{lem:first-order-short-time-moments} and the continuity of \(\gamma_{ij,g}\),
			\[
			\sup_{u\in\operatorname{supp}(K)}
			\left|\kappa_{\Delta_n,g}^{ij}(x+h_nu)-\gamma_{ij,g}(x)\right|\to0.
			\]
			Since \(\widehat S_ne_0\) is the local response vector for the constant function \(1\), the preceding two displays and \eqref{eq:first-order-design-limit} imply
			\(\widehat\gamma_{ij,g,n}(x)\to\gamma_{ij,g}(x)\) in probability. The consistency of \(\widehat\sigma_{B,ij,p}^{2,g}(x)\) follows. Slutsky's theorem and \Cref{thm:first-order-recovery-clt} prove the studentized limit.
		\end{proof}

	\begingroup
	\subsection{Second-order recovery and studentization}
	{Using every starting index makes adjacent second-difference responses share a one-step increment, so their leading covariances cancel and the first-order martingale proof cannot be repeated. We instead use nonoverlapping two-step blocks, retaining the lag-one/lag-two covariance within each response while producing a martingale-difference array.}

	Put \(N_n:=\lfloor n/2\rfloor\). For \(\Delta>0\), define the normalized second-order block
	\[
	c_{\Delta,g}^{ij}(y)
	:=
	\frac{
	P_{2\Delta}^{ij}g(y)-2P_\Delta^{ij}g(y)+\delta_{ij}g(y)
	}{\Delta^2}.
	\]

	\begin{assumption}[Second-order local smoothness]\label[assumption]{ass:second-order-smoothness}
		For the polynomial degree \(p\), regime pair \((i,j)\), test function \(g\), and neighborhood \(U_x\),
		\[
		c_{\Delta_n,g}^{ij}\in C^{p+1}(U_x),
		\qquad
		\sup_{n\ge1}
		\left\|c_{\Delta_n,g}^{ij}\right\|_{C^{p+1}(U_x)}
		<\infty.
		\]
	\end{assumption}

	{A model-level sufficient condition for Assumptions~A5--A6 is given in Proposition~\ref{prop:appendix-normalized-smoothness}; its proof and the numerical-model verification are given in Supplementary Sections~S.5--S.6.}

	\begin{lemma}[Second-order short-time moments]\label[lemma]{lem:second-order-short-time-moments}
		Suppose the coefficient hypotheses of \Cref{prop:kth} hold through order three and either \(g\in C_c^6(\mathbb R^d)\) or \(g\equiv1\). For \(z=(y,i)\), set
		\[
		\mathcal D_{\Delta,g}^{(j)}
		:=
		f_g^{(j)}(Z_{2\Delta})
		-2f_g^{(j)}(Z_\Delta)
		+f_g^{(j)}(Z_0),
		\qquad Z_0=z.
		\]
		Then, uniformly for \(y\in U_x\),
		\begin{align}
		&c_{\Delta,g}^{ij}(y)
		=
		C_{ij}g(y)+\Delta A_3^{ij}g(y)+o(\Delta),
		\label{eq:second-order-temporal-bias}\\
		\frac{1}{\Delta}
		&\operatorname{Var}_{(y,i)}
		\left(\mathcal D_{\Delta,g}^{(j)}\right)
		\to 2\gamma_{ij,g}(y),
		\label{eq:second-order-variance-limit}\\
		&\mathbb E_{(y,i)}
		\left[\left|\mathcal D_{\Delta,g}^{(j)}\right|^4\right]
		\lesssim\Delta.
		\label{eq:second-order-fourth-moment}\\
		&\kappa_{\Delta,g}^{ij,[2]}(y)
		:=
		\frac{1}{\Delta}
		\mathbb E_{(y,i)}
		\left[\left(\mathcal D_{\Delta,g}^{(j)}\right)^2\right]
		\to 2\gamma_{ij,g}(y).\notag
		\end{align}
	\end{lemma}

	\begin{proof}
		The order-three expansion in \Cref{prop:kth}, or \Cref{lem:cutoff-localization} when \(g\equiv1\), gives
		\[
		P_{2\Delta}^{ij}g-2P_\Delta^{ij}g+\delta_{ij}g
		=
		\Delta^2 C_{ij}g+\Delta^3A_3^{ij}g+o(\Delta^3)
		\]
		uniformly on \(\overline U_x\). Division by \(\Delta^2\) proves \eqref{eq:second-order-temporal-bias}.

		We next identify the variance. Write \(f:=f_g^{(j)}\), and let \(P_t\) denote the full semigroup. By the Markov property,
		\[
		\begin{aligned}
		\mathbb E_z\!\left[
		\left\{f(Z_{2\Delta})-2f(Z_\Delta)+f(z)\right\}^2
		\right]
		&=
		P_{2\Delta}(f^2)(z)+4P_\Delta(f^2)(z)+f(z)^2\\
		&\quad-4P_\Delta\{fP_\Delta f\}(z)
		+2f(z)P_{2\Delta}f(z)-4f(z)P_\Delta f(z).
		\end{aligned}
		\]
		Set \(r_\Delta:=\Delta^{-1}(P_\Delta f-f)\). Then
		\[
		P_\Delta\{fP_\Delta f\}
		=
		P_\Delta(f^2)+\Delta P_\Delta(fr_\Delta).
		\]
		Dynkin's formula and the boundedness of \(\mathcal Lf\) give \(\|r_\Delta\|_\infty\le\|\mathcal Lf\|_\infty\), while the order-one expansion gives \(r_\Delta\to\mathcal Lf\) locally uniformly. If \(g\) is compactly supported, \(f(r_\Delta-\mathcal Lf)\) has fixed compact support; if \(g\equiv1\), the bounded total switching rate gives the required global bound. The linear-growth bounds following from \Cref{ass:wellposedness} imply uniform compact containment for initial states in \(\overline U_x\times S\). Together with Feller continuity applied to the fixed function \(f\mathcal Lf\), this yields \(P_\Delta(fr_\Delta)\to f\mathcal Lf\) uniformly on \(\overline U_x\times S\). Combining this limit with the order-one expansion of \(P_\Delta(f^2)\) gives
		\[
		P_\Delta\{fP_\Delta f\}
		=
		f^2+\Delta\{\mathcal L(f^2)+f\mathcal Lf\}+o(\Delta)
		\]
		uniformly on \(\overline U_x\times S\). Substitution in the preceding display yields
		\[
		\mathbb E_z\!\left[
		\left\{f(Z_{2\Delta})-2f(Z_\Delta)+f(z)\right\}^2
		\right]
		=
		2\Delta\{\mathcal L(f^2)(z)-2f(z)\mathcal Lf(z)\}
		+o(\Delta).
		\]
		The expression in braces is \(\gamma_{ij,g}(y)\). The conditional mean of the second difference is \(O(\Delta^2)\), so subtracting its square does not affect the limit after division by \(\Delta\). This proves \eqref{eq:second-order-variance-limit}.

		For the fourth moment, write the second difference as the difference of the two consecutive increments of \(f(Z_t)\). The inequality
		$|a-b|^4\le8(|a|^4+|b|^4)$
		reduces the assertion to fourth moments of one-step increments. For compactly supported \(g\), the semimartingale decomposition of \(f(Z_t)\), the boundedness of \(\mathcal Lf\) and of its diffusion and jump carré-du-champ coefficients, and the Burkholder--Davis--Gundy inequality give
		\[
		\sup_{z\in E}
		\mathbb E_z\!\left[
		|f(Z_\Delta)-f(z)|^4
		\right]\lesssim\Delta.
		\]
		For \(g\equiv1\), the same estimate follows directly from the bounded total switching rate. Conditioning at time \(\Delta\) proves \eqref{eq:second-order-fourth-moment}. Finally, the squared conditional mean is \(O(\Delta^4)\); hence \eqref{eq:second-order-variance-limit} also gives the asserted limit for \(\kappa_{\Delta,g}^{ij,[2]}\).
	\end{proof}

	For \(s=0,\ldots,N_n-1\), define the nonoverlapping second-difference response
	\[
	D_{s,n}^{ij,g,[2]}
	:=
		f_g^{(j)}(Z_{2s+2,n})
		-2f_g^{(j)}(Z_{2s+1,n})
		+f_g^{(j)}(Z_{2s,n}).
	\]
	On \(\{\Lambda_{t_{2s,n}}=i\}\), this response is
	\[
	\begin{aligned}
	D_{s,n}^{ij,g,[2]}
	=
	g(Y_{2s+2,n})\mathbf 1_{\{\Lambda_{t_{2s+2,n}}=j\}}
	-2g(Y_{2s+1,n})\mathbf 1_{\{\Lambda_{t_{2s+1,n}}=j\}}
	+\delta_{ij}g(Y_{2s,n}),
	\end{aligned}
	\]
	the second forward difference of the terminal-regime block response.
	For later use, put
	\[
	v_{\Delta,g}^{ij,[2]}(y)
	:=
	\operatorname{Var}_{(y,i)}
	\left(\mathcal D_{\Delta,g}^{(j)}\right).
	\]
	Define the regime-weighted second-order innovation
	\begin{equation}\label{eq:second-order-weighted-innovation}
	\xi_{s,n}^{ij,g,[2]}
	:=
	\mathbf 1_{\{\Lambda_{t_{2s,n}}=i\}}
	\left\{
	D_{s,n}^{ij,g,[2]}
	-\mathbb E\!\left[
	D_{s,n}^{ij,g,[2]}
	\,\middle|\,\mathcal G_{s,n}
	\right]
	\right\},
	\qquad
	\mathcal G_{s,n}:=\mathcal F_{t_{2s,n}}.
	\end{equation}
	Then \(\xi_{s,n}^{ij,g,[2]}\) is \(\mathcal G_{s+1,n}\)-measurable and
	\(\mathbb E[\xi_{s,n}^{ij,g,[2]}\mid\mathcal G_{s,n}]=0\). Moreover,
	\begin{equation}\label{eq:second-order-weighted-decomposition}
	\mathbf 1_{\{\Lambda_{t_{2s,n}}=i\}}D_{s,n}^{ij,g,[2]}
	=
	\Delta_n^2\mathbf 1_{\{\Lambda_{t_{2s,n}}=i\}}
	c_{\Delta_n,g}^{ij}(Y_{2s,n})
	+\xi_{s,n}^{ij,g,[2]}.
	\end{equation}
	Define the common design matrix
	\[
	\widehat S_{n,\mathrm{nb}}^{i,(p)}(x)
	:=
	\frac{1}{N_nh_n^d}
	\sum_{s=0}^{N_n-1}
	K\!\left(\frac{Y_{2s,n}-x}{h_n}\right)
	\mathbf 1_{\{\Lambda_{t_{2s,n}}=i\}}
	\psi_p\!\left(\frac{Y_{2s,n}-x}{h_n}\right)
	\psi_p\!\left(\frac{Y_{2s,n}-x}{h_n}\right)^\top.
	\]
	On its invertibility event, set
	\begin{equation}\label{eq:second-order-nb-estimator}
	\begin{aligned}
	\widehat C_{n,\mathrm{nb}}^{ij,(p)}g(x)
	&:=
	e_0^\top
	\left(\widehat S_{n,\mathrm{nb}}^{i,(p)}(x)\right)^{-1}
	\frac{1}{N_nh_n^d\Delta_n^2}
	\sum_{s=0}^{N_n-1}
	K\!\left(\frac{Y_{2s,n}-x}{h_n}\right)
	\mathbf 1_{\{\Lambda_{t_{2s,n}}=i\}}\\
	&\qquad\qquad\times
	D_{s,n}^{ij,g,[2]}
	\psi_p\!\left(\frac{Y_{2s,n}-x}{h_n}\right).
	\end{aligned}
	\end{equation}
	This estimator uses no preliminary estimate of \(B_{ij}g\). The first-order term cancels inside every response before smoothing and normalization.

	Define
	\begin{equation}\label{eq:second-order-asymptotic-variance}
	\sigma_{C,\mathrm{nb},ij,p}^{2,g}(x)
	:=
	\frac{2\gamma_{ij,g}(x)}{\varpi_i(x)}
	e_0^\top M_p(K)^{-1}Q_p(K)M_p(K)^{-1}e_0.
	\end{equation}

	\begin{theorem}[Second-order recovery central limit theorem]\label{thm:second-order-recovery-clt}
		Fix \(p\ge0\), a regime pair \((i,j)\), a test function \(g\in C_c^6(\mathbb R^d)\) or \(g\equiv1\), and a design point \(x\). Suppose \Cref{ass:mixing}, part~\textup{(i)} of \Cref{ass:local-densities}, and \Cref{ass:second-order-smoothness} hold on \(U_x\), and suppose the coefficient hypotheses of \Cref{prop:kth} hold through order three. If
		$N_n\Delta_n^3h_n^d\to\infty$,
		then
		$\widehat C_{n,\mathrm{nb}}^{ij,(p)}g(x)
		\xrightarrow{\mathbb P}C_{ij}g(x)$.
		If, in addition,
		\begin{equation}\label{eq:second-order-clt-rates}
		\sqrt{N_n\Delta_n^3h_n^d}
		\bigl(\Delta_n+h_n^{p+1}\bigr)\to0,
		\end{equation}
		then
		\[
		\sqrt{N_n\Delta_n^3h_n^d}
		\left\{
		\widehat C_{n,\mathrm{nb}}^{ij,(p)}g(x)-C_{ij}g(x)
		\right\}
		\xrightarrow{d}
		N\!\left(0,\sigma_{C,\mathrm{nb},ij,p}^{2,g}(x)\right).
		\]
	\end{theorem}

	\begin{proof}
		The stationary subsample \((Z_{2s,n})_{s\ge0}\) has step size \(2\Delta_n\). Applying the proof of \Cref{lem:shrinking-localization-vnorm} with this step size coordinatewise gives
		\[
		\widehat S_{n,\mathrm{nb}}^{i,(p)}(x)
		-\mathbb E\!\left[\widehat S_{n,\mathrm{nb}}^{i,(p)}(x)\right]
		=
		O_{\mathbb P}\!\left((N_n\Delta_nh_n^d)^{-1/2}\right).
		\]
		The kernel change-of-variables formula therefore gives
		\begin{equation}\label{eq:second-order-design-limit}
		\left(\widehat S_{n,\mathrm{nb}}^{i,(p)}(x)\right)^{-1}
		\xrightarrow{\mathbb P}
		\varpi_i(x)^{-1}M_p(K)^{-1}.
		\end{equation}

		Taylor's theorem, \Cref{ass:second-order-smoothness}, and \eqref{eq:second-order-temporal-bias} give
		\begin{equation}\label{eq:second-order-linearization}
		\begin{aligned}
		\widehat C_{n,\mathrm{nb}}^{ij,(p)}g(x)-C_{ij}g(x)
		&=
		e_0^\top
		\left(\widehat S_{n,\mathrm{nb}}^{i,(p)}(x)\right)^{-1}
		\frac{1}{N_nh_n^d\Delta_n^2}
		\sum_{s=0}^{N_n-1}
		K\!\left(\frac{Y_{2s,n}-x}{h_n}\right)\\
		&\qquad\qquad\times
		\psi_p\!\left(\frac{Y_{2s,n}-x}{h_n}\right)
		\xi_{s,n}^{ij,g,[2]}
		+O_{\mathbb P}\!\left(h_n^{p+1}+\Delta_n\right).
		\end{aligned}
		\end{equation}
		By \eqref{eq:second-order-variance-limit},
		\[
		\mathbb E\!\left[
		(\xi_{s,n}^{ij,g,[2]})^2
		\,\middle|\,\mathcal G_{s,n}
		\right]
		=
		\mathbf 1_{\{\Lambda_{t_{2s,n}}=i\}}
		v_{\Delta_n,g}^{ij,[2]}(Y_{2s,n})
		=O(\Delta_n)
		\]
		uniformly on the kernel support. Martingale orthogonality and localization show that the first term in \eqref{eq:second-order-linearization} is
		$O_{\mathbb P}\!\left((N_n\Delta_n^3h_n^d)^{-1/2}\right)$,
		which proves consistency.

		For the distributional limit, put
		\[
		\mathcal M_n^{[2]}
		:=
		\frac{1}{\sqrt{N_n\Delta_nh_n^d}}
		\sum_{s=0}^{N_n-1}
		K\!\left(\frac{Y_{2s,n}-x}{h_n}\right)
		\psi_p\!\left(\frac{Y_{2s,n}-x}{h_n}\right)
		\xi_{s,n}^{ij,g,[2]}.
		\]
		Its predictable quadratic-variation matrix is
		\[
		\frac{1}{N_nh_n^d}
		\sum_{s=0}^{N_n-1}
		K\!\left(\frac{Y_{2s,n}-x}{h_n}\right)^2
		\mathbf 1_{\{\Lambda_{t_{2s,n}}=i\}}
		\psi_p\!\left(\frac{Y_{2s,n}-x}{h_n}\right)
		\psi_p\!\left(\frac{Y_{2s,n}-x}{h_n}\right)^\top
		\frac{v_{\Delta_n,g}^{ij,[2]}(Y_{2s,n})}{\Delta_n}.
		\]
		By \eqref{eq:second-order-variance-limit}, the last factor converges uniformly to \(2\gamma_{ij,g}\) on \(U_x\). Applying the localized \(V\)-norm argument to the subsampled chain makes the centered empirical average negligible, and the change-of-variables formula identifies the limit as
		$2\varpi_i(x)\gamma_{ij,g}(x)Q_p(K)$.

		Fix \(a\in\mathbb R^{q_p}\), and let \(X_{s,n}^{[2]}(a)\) denote the \(s\)-th summand of \(a^\top\mathcal M_n^{[2]}\). Supplementary Section~S.4 verifies the shifted-filtration, maximum-increment, realized-square, and zero-variance conditions for these projected summands, and also records the conditional Lindeberg check. Lemma~\ref{lem:appendix-martingale-clt} yields the limit for each Cramér--Wold projection, and the Cramér--Wold device gives
		\[
		\mathcal M_n^{[2]}
		\xrightarrow{d}
		N\!\left(0,2\varpi_i(x)\gamma_{ij,g}(x)Q_p(K)\right).
		\]
		Combining this limit with \eqref{eq:second-order-design-limit}, \eqref{eq:second-order-linearization}, and \eqref{eq:second-order-clt-rates} and applying Slutsky's theorem proves the assertion.
	\end{proof}

	Define
	\[
	\begin{aligned}
		\widehat\varpi_{i,n}^{\mathrm{nb}}(x)
		&:=
		\frac{1}{N_nh_n^d}
		\sum_{s=0}^{N_n-1}
		K\!\left(\frac{Y_{2s,n}-x}{h_n}\right)
		\mathbf 1_{\{\Lambda_{t_{2s,n}}=i\}},\\
		\widehat\gamma_{ij,g,n}^{[2]}(x)
	&:=
	e_0^\top
	\left(\widehat S_{n,\mathrm{nb}}^{i,(p)}(x)\right)^{-1}
	\frac{1}{N_nh_n^d\Delta_n}
	\sum_{s=0}^{N_n-1}
	K\!\left(\frac{Y_{2s,n}-x}{h_n}\right)
	\mathbf 1_{\{\Lambda_{t_{2s,n}}=i\}}
	\bigl(D_{s,n}^{ij,g,[2]}\bigr)^2\\
	&\quad\times
	\psi_p\!\left(\frac{Y_{2s,n}-x}{h_n}\right),\\
		\widehat\sigma_{C,\mathrm{nb},ij,p}^{2,g}(x)
		&:=
		\frac{\widehat\gamma_{ij,g,n}^{[2]}(x)}
		{\widehat\varpi_{i,n}^{\mathrm{nb}}(x)}
		e_0^\top M_p(K)^{-1}Q_p(K)M_p(K)^{-1}e_0.
	\end{aligned}
	\]
	When the last quantity is positive, let
	\(\widehat\sigma_{C,\mathrm{nb},ij,p}^{g}(x)\)
	denote its positive square root.

	\begin{corollary}[Feasible second-order studentization]\label[corollary]{cor:second-order-studentization}
		Under the assumptions of \Cref{thm:second-order-recovery-clt},
		\[
		\widehat\varpi_{i,n}^{\mathrm{nb}}(x)
		\xrightarrow{\mathbb P}\varpi_i(x),
		\qquad
		\widehat\gamma_{ij,g,n}^{[2]}(x)
		\xrightarrow{\mathbb P}2\gamma_{ij,g}(x).
		\]
		If \(\gamma_{ij,g}(x)>0\) and \eqref{eq:second-order-clt-rates} holds, then
		\[
		\frac{
		\sqrt{N_n\Delta_n^3h_n^d}
		\{\widehat C_{n,\mathrm{nb}}^{ij,(p)}g(x)-C_{ij}g(x)\}
		}{
		\widehat\sigma_{C,\mathrm{nb},ij,p}^{g}(x)
		}
		\xrightarrow{d}N(0,1).
		\]
	\end{corollary}

	\begin{proof}
		The density convergence follows from the change-of-variables formula and the localized \(V\)-norm covariance bound for \((Z_{2s,n})\). For the variance coefficient, set
		\[
		\eta_{s,n}^{[2]}
		:=
		\mathbf 1_{\{\Lambda_{t_{2s,n}}=i\}}
		\left\{
		\frac{(D_{s,n}^{ij,g,[2]})^2}{\Delta_n}
		-\kappa_{\Delta_n,g}^{ij,[2]}(Y_{2s,n})
		\right\}.
		\]
		By the Markov property and the definition of \(\kappa_{\Delta_n,g}^{ij,[2]}\), \(\eta_{s,n}^{[2]}\) is \(\mathcal G_{s+1,n}\)-measurable and
		$\mathbb E[\eta_{s,n}^{[2]}\mid\mathcal G_{s,n}]=0$.
		By \eqref{eq:second-order-fourth-moment},
		$
		\mathbb E\!\left[
		(\eta_{s,n}^{[2]})^2
		\,\middle|\,\mathcal G_{s,n}
		\right]\lesssim\Delta_n^{-1}
		$
		on the kernel support. Martingale orthogonality and localization therefore make the corresponding local empirical average
		\(O_{\mathbb P}((N_n\Delta_nh_n^d)^{-1/2})\). The last assertion of \Cref{lem:second-order-short-time-moments}, continuity of \(\gamma_{ij,g}\), and \eqref{eq:second-order-design-limit} yield
		\(\widehat\gamma_{ij,g,n}^{[2]}(x)\to2\gamma_{ij,g}(x)\) in probability. The result follows from \Cref{thm:second-order-recovery-clt} and Slutsky's theorem.
	\end{proof}

	If \(\gamma_{ij,g}(x)=0\), the second-order limit is again degenerate. A nondegenerate studentized limit then requires a different normalization determined by the first nonzero short-time variance coefficient.

	\endgroup

	\section{Numerical example}
	\label{sec:examples}
	{We use a smooth two-regime model to illustrate the deterministic expansion and the fixed- and shrinking-mesh distribution theories. The detailed verification of the model assumptions and the secondary numerical tables are given in Supplementary Sections~S.6--S.7, so the main text focuses on the common-design first- and second-order estimators.}

	\subsection{A two-regime model}
	{Let \(d=1\) and \(S=\{1,2\}\). The numerical model is
		\[
		\begin{aligned}
			b(x,i)&=-\beta_i x, & \sigma(x,i)&=\sigma_i,\\
			q_{12}(x)&=0.55+0.25\tanh x,
			& q_{21}(x)&=0.45-0.20\tanh x,
		\end{aligned}
		\]
		with
		\[
		(\beta_1,\beta_2)=(1,2),
		\qquad
		(\sigma_1,\sigma_2)=(1,1.5).
		\]
		Let \(\chi\in C_c^\infty(\mathbb R)\) satisfy \(0\le\chi\le1\), \(\chi=1\) on \([-1.25,1.25]\), and \(\chi=0\) outside \([-1.5,1.5]\), and set
		\[
		g_0=\chi,\qquad g_1(x)=x\chi(x),\qquad g_2(x)=x^2\chi(x).
		\]
	}

	{Supplementary Section~S.6 verifies Assumptions~A1--A6 for every configuration reported below, including the fixed-mesh density, small-ball, and variance conditions, the normalized first- and second-order smoothness bounds, and the order-three coefficient regularity required by \Cref{thm:second-order-recovery-clt}.}

	\subsection{Explicit coefficients for the numerical targets}
	For this subsection write \(q=q_{12}\) and \(r=q_{21}\). Since there are only two regimes, the intermediate-regime sum in \eqref{eq:Coff} is absent. Therefore
	\[
	B_{12}g(x)=q(x)g(x),
	\qquad
	C_{12}g
	=\mathscr L_1(qg)+q\mathscr L_2g-q(q+r)g.
	\]
	The product rule gives
	\[
	\mathscr L_1(qg)=q\mathscr L_1g+g\mathscr L_1q+\sigma_1^2q'g',
	\]
	so the cross-gradient interaction is already visible in one dimension. {The three probes defined above identify different pieces of the hybrid expansion. At each reported design point, where the cutoff equals one, they have the same pointwise coefficients as \(1,x,x^2\).} For the off-diagonal block \(1\to2\),
	\[
	C_{12}1=\mathscr L_1q-q(q+r)
	=-\beta_1xq'(x)+\frac{\sigma_1^2}{2}q''(x)-q(x)\{q(x)+r(x)\},
	\]
	\[
	C_{12}x-xC_{12}1
	=-q(x)(\beta_1+\beta_2)x+\sigma_1^2q'(x).
	\]
	For the quadratic probe,
	\[
	C_{12}(x^2)-x^2C_{12}1
	=q(x)\{-2(\beta_1+\beta_2)x^2+\sigma_1^2+\sigma_2^2\}
	+2\sigma_1^2xq'(x).
	\]
	These identities give analytic targets for the expansion and recovery diagnostics. The diagonal targets are computed from \eqref{eq:Bops} and \eqref{eq:Cdiag}.

	\subsection{Observation and computational scheme}
	Only the grid observations \((X_{k\Delta},\Lambda_{k\Delta})\) enter the estimators. Between observation times, paths are simulated exactly by combining the Ornstein--Uhlenbeck transition
	\[
	X_{t+s}=e^{-\beta_is}X_t+
	\sigma_i\left(\frac{1-e^{-2\beta_is}}{2\beta_i}\right)^{1/2}\xi,
	\qquad \xi\sim N(0,1),
	\]
	with thinning. Candidate switches are proposed at rates \(0.80\) in regime \(1\) and \(0.65\) in regime \(2\), and are accepted with probabilities \(q_{12}(X_\tau)/0.80\) and \(q_{21}(X_\tau)/0.65\), respectively. Long paths are initialized from a numerical approximation of the invariant distribution and use an additional burn-in of \(5\) time units.

	Reference block values are computed independently of the Monte Carlo samples. For fixed \(j\) and \(g\), the vector \(u(t,x)=(u_1(t,x),u_2(t,x))^\top\), with \(u_i(t,x)=P_t^{ij}g(x)\), solves
	\[
	\partial_tu_i
	=
	\mathscr L_i u_i
	+\sum_{\ell\ne i}q_{i\ell}(u_\ell-u_i),
	\qquad
	u_i(0,x)=g(x)\mathbf 1_{\{i=j\}}.
	\]
	\begingroup
	We discretize this system on \([-8,8]\) by a conservative block-tridiagonal Markov generator with spatial mesh \(0.0125\), advance it by the Crank--Nicolson scheme of \citet{cranknicolson1947} with time steps at most \(1.25\times10^{-4}\), and compute the invariant grid mass from the adjoint generator. The spatial boundary closure is reflecting in the finite-state-generator sense: at the left endpoint the outward left-jump rate is set to zero, at the right endpoint the outward right-jump rate is set to zero, and each boundary diagonal is the negative sum of the retained inward spatial rate and the switching rate. Thus no probability is lost through the truncated boundary and every row of the coupled generator sums to zero; the closure is neither absorbing nor an OU extrapolation. The discrete semigroup preserves constants to \(5.8\times10^{-15}\). Halving the spatial mesh and maximum time step changes the block values used in \Cref{tab:simulation-expansion} by at most \(1.8\times10^{-7}\).
	\endgroup

	All experiments use the Epanechnikov kernel
	$K(u)=\frac34(1-u^2)\mathbf 1_{\{|u|\le1\}}$,
	and the cutoff probes \(g_0,g_1,g_2\). For independent Monte Carlo replications \(W_1,\ldots,W_R\), we write
	\[
	\overline W_R=\frac1R\sum_{b=1}^R W_b,
	\qquad
	\widehat{\operatorname{se}}(\overline W_R)
	=
	\left\{
	\frac{1}{R(R-1)}
	\sum_{b=1}^R(W_b-\overline W_R)^2
	\right\}^{1/2}.
	\]
	Here, s.e.\ denotes this Monte Carlo standard error. For an empirical coverage \(\widehat c_R\), its Monte Carlo standard error is \(\{\widehat c_R(1-\widehat c_R)/R\}^{1/2}\).

	\subsection{Deterministic expansion benchmark}
	To resolve the expansion remainder, define
	\[
	\begin{aligned}
	D_{1,g}(\Delta)
	&:=
	\frac{P_\Delta^{12}g(0)-A_0^{12}g(0)-\Delta B_{12}g(0)}
	{\Delta^2},\\
	D_{2,g}(\Delta)
	&:=
	\frac{P_\Delta^{12}g(0)-A_0^{12}g(0)-\Delta B_{12}g(0)
	-\tfrac12\Delta^2C_{12}g(0)}
	{\Delta^2}.
	\end{aligned}
	\]
	The second-order expansion predicts \(D_{1,g}(\Delta)\to C_{12}g(0)/2\) and \(D_{2,g}(\Delta)\to0\). For \(g_0\), the limiting value is \(-0.275\). For \(g_1\), it is \(0.125\); at \(x=0\), this coefficient includes the cross-gradient contribution \(\sigma_1^2q_{12}'(0)=0.25\).

	\begin{table}[H]
		\centering
		\small
		\caption{Deterministic normalized expansion residuals at \(x=0\) for the block \(1\to2\). The limits of \(D_{1,g_0}\) and \(D_{1,g_1}\) are \(-0.275\) and \(0.125\), respectively.}
		\label{tab:simulation-expansion}
		\begin{tabular}{crrrr}
			\toprule
			\(\Delta\) & \(D_{1,g_0}(\Delta)\) & \(D_{2,g_0}(\Delta)\)
			& \(D_{1,g_1}(\Delta)\) & \(D_{2,g_1}(\Delta)\)\\
			\midrule
			\(0.10000\) & \(-0.268530\) & \(0.006470\) & \(0.100548\) & \(-0.024452\)\\
			\(0.05000\) & \(-0.270612\) & \(0.004388\) & \(0.112068\) & \(-0.012932\)\\
			\(0.02500\) & \(-0.272774\) & \(0.002226\) & \(0.118298\) & \(-0.006702\)\\
			\(0.01250\) & \(-0.273884\) & \(0.001116\) & \(0.121584\) & \(-0.003416\)\\
			\(0.00625\) & \(-0.274441\) & \(0.000559\) & \(0.123273\) & \(-0.001727\)\\
			\bottomrule
		\end{tabular}
	\end{table}
	The convergence of both \(D_{1,g}\) columns to their analytic limits verifies the second-order block-generator coefficient, including the cross-gradient term for \(g_1\). Moreover, \(D_{2,g}(\Delta)/\Delta\) stabilizes along the three finest meshes: its values range from \(0.0890\) to \(0.0895\) for \(g_0\), and from \(-0.2681\) to \(-0.2764\) for \(g_1\). This is the expected \(O(\Delta^3)\) signed remainder after the second-order correction.

	\subsection{Fixed-mesh estimation and studentization}
	The fixed-mesh experiment uses \(\Delta=0.05\), \(500\) independent replications, and
	\[
	n\in\{50000,100000,200000,400000\},
	\qquad
	h_n=
	\begin{cases}
	1.6n^{-0.40},&p=0,\\
	0.9n^{-0.25},&p=1.
	\end{cases}
	\]
	These deterministic bandwidths were fixed before simulation and satisfy
	\(nh_n\to\infty\) and \(\sqrt{nh_n}\,h_n^{p+1}\to0\).
	The deterministic target for the displayed block and probe is
	\[
	P_{0.05}^{12}g_0(0)=0.026823469,
	\qquad
	\sigma_{12,0.05,1}^{g_0}(0)=0.248794.
	\]
	For the local-linear estimator in replication \(b\), define the oracle- and plug-in-standardized statistics
	\[
	Z_b^{\mathrm{or}}
	=
	\frac{\sqrt{n h_n}\{\widehat P_{0.05,h_n,b}^{12}g_0(0)
	-P_{0.05}^{12}g_0(0)\}}
	{\sigma_{12,0.05,1}^{g_0}(0)},
	\qquad
	Z_b^{\mathrm{pl}}
	=
	\frac{\sqrt{n h_n}\{\widehat P_{0.05,h_n,b}^{12}g_0(0)
	-P_{0.05}^{12}g_0(0)\}}
	{\widehat\sigma_b}.
	\]
	\begingroup
	Supplementary Proposition~S.4 defines \(\widehat\sigma_b^2\) and proves its consistency for \(\sigma_{12,0.05,1}^{2,g_0}(0)\). All components use the same Epanechnikov kernel and bandwidth \(h_n\), the block moments use \(p=1\), and the raw variance is truncated at zero. The implementation floors both the variance estimate and the density denominator at \(10^{-12}\) to avoid division by zero. The kernel factor in \eqref{eq:fixedmesh-scalar-variance} is \(3/5\).
	\endgroup

	\begin{table}[H]
		\centering
		\scriptsize
		\caption{Fixed-mesh RMSE and studentization for \(P_{0.05}^{12}g_0(0)\), based on \(500\) replications. The two coverage columns use the nominal \(95\%\) normal interval.}
		\label{tab:simulation-fixedmesh}
		\begin{tabular}{rrrrrrrrr}
			\toprule
			\(n\) & RMSE \(p=0\) & RMSE \(p=1\)
			& \(\overline Z^{\mathrm{or}}\) & sd\((Z^{\mathrm{or}})\) & cov.\({}^{\mathrm{or}}\)
			& \(\overline Z^{\mathrm{pl}}\) & sd\((Z^{\mathrm{pl}})\) & cov.\({}^{\mathrm{pl}}\)\\
			\midrule
			\(50000\)  & \(0.007506\) & \(0.004617\) & \(-0.033\) & \(1.018\) & \(0.956\) & \(-0.124\) & \(1.044\) & \(0.948\)\\
			\(100000\) & \(0.005948\) & \(0.003362\) & \(-0.057\) & \(0.961\) & \(0.960\) & \(-0.118\) & \(0.979\) & \(0.962\)\\
			\(200000\) & \(0.004974\) & \(0.002637\) & \(-0.103\) & \(0.973\) & \(0.964\) & \(-0.153\) & \(0.983\) & \(0.954\)\\
			\(400000\) & \(0.004128\) & \(0.002057\) & \(-0.058\) & \(0.988\) & \(0.952\) & \(-0.096\) & \(0.998\) & \(0.946\)\\
			\bottomrule
		\end{tabular}
	\end{table}
	\begingroup
	The local-linear RMSE is smaller at each displayed \(n\), although the bandwidths differ across polynomial degrees. The standard deviations and coverages of the oracle and plug-in statistics are close to one and \(0.95\), respectively; the Monte Carlo standard error of a \(0.95\) coverage estimate is \(0.0097\). At \(n=400000\), the analogous checks for \(g_1,g_2\) and both blocks give oracle standard deviations in \([0.928,1.035]\), oracle coverages in \([0.936,0.972]\), and plug-in coverages in \([0.936,0.946]\).
	\endgroup

	\begin{figure}
		\centering
		\includegraphics[width=0.94\textwidth]{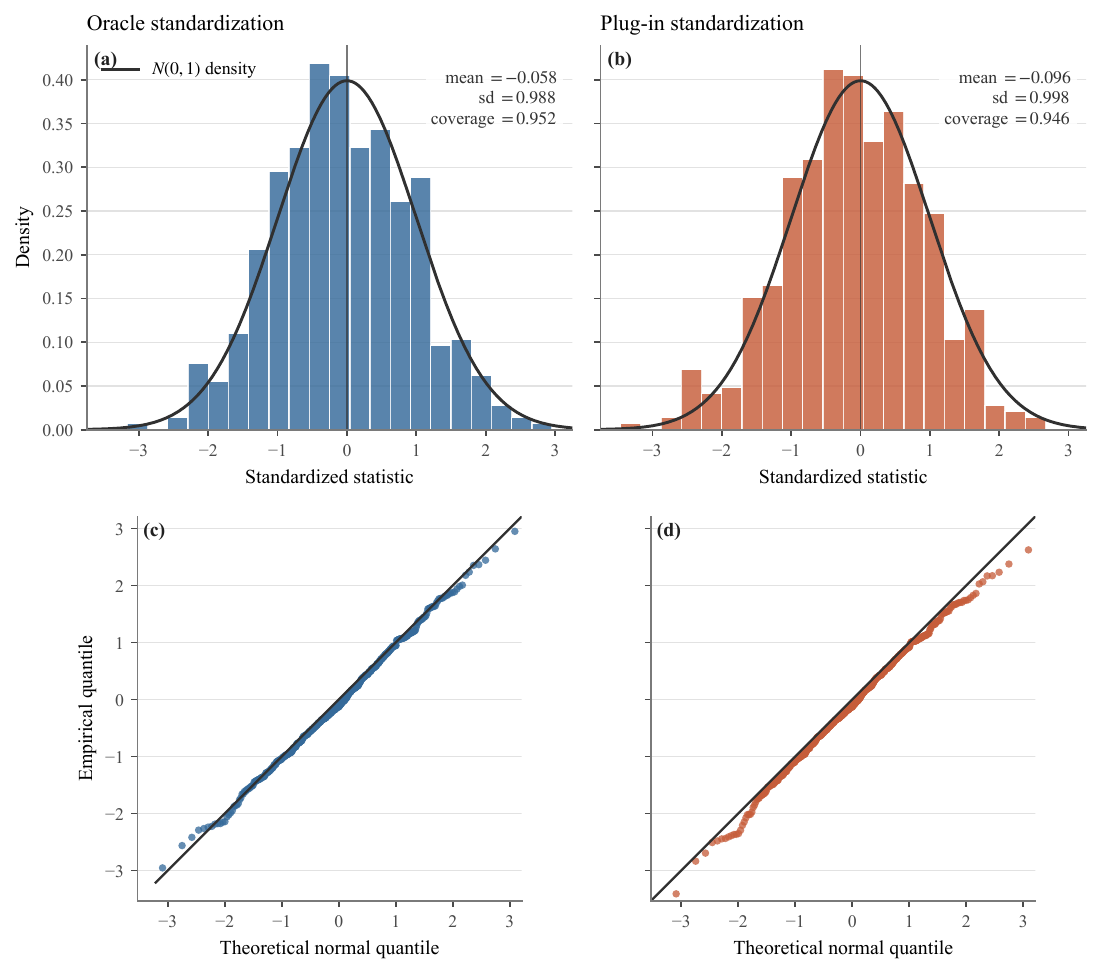}
		\caption{Distribution diagnostics for the oracle- and plug-in-standardized local-linear estimator of \(P_{0.05}^{12}g_0(0)\), based on \(500\) replications with \(n=400000\). The upper panels are density histograms with the standard-normal density superimposed; the lower panels are normal Q--Q plots.}
		\label{fig:fixedmesh-clt-diagnostics}
	\end{figure}
	\Cref{fig:fixedmesh-clt-diagnostics} also shows agreement in shape and tails: the oracle and plug-in statistics have skewness \(0.065\) and \(-0.156\), and excess kurtosis \(-0.170\) and \(-0.106\), respectively.

	\subsection{Shrinking-mesh recovery and error decomposition}
	The recovery experiment uses local-linear smoothing and \(100\) replications at each displayed triplet \((n,\Delta_n,h_n)\). For the off-diagonal block at \(x=0\), define the deterministic finite-difference oracles
	\[
	B_{\Delta}^{\mathrm{FD}}g
	:=
	\frac{P_\Delta^{12}g(0)}{\Delta},
	\qquad
	C_{\Delta}^{\mathrm{FD}}g
	:=
	\frac{P_{2\Delta}^{12}g(0)-2P_\Delta^{12}g(0)}{\Delta^2}.
	\]

	{We report the two common-design estimators; Supplementary Section~S.7.3 gives the secondary separate-lag results.}

	For \(g_0\), the coefficient targets are \(B_{12}g_0(0)=0.55\) and \(C_{12}g_0(0)=-0.55\). The first oracle satisfies
	\[
	B_{\Delta}^{\mathrm{FD}}g_0
	=
	0.55-0.275\Delta+o(\Delta),
	\]
	so its displacement below \(0.55\) is predicted truncation bias.
	\begingroup
	For the first-order distributional check, the deterministic stationary solver gives
	\(\varpi_1(0)=0.253033\). Since \(q_{12}(0)=0.55\), \(g_0(0)=1\), and the local-linear Epanechnikov design factor is \(3/5\), \eqref{eq:first-order-asymptotic-variance} gives
	\[
	\sigma_{B,12,1}^{g_0}(0)
	=
	\left\{\frac{0.55}{0.253033}\frac35\right\}^{1/2}
	=1.142007.
	\]
	For replication \(b\), define the stochastic-oracle statistic
	\[
	Z_{B,b}^{\mathrm{or}}
	:=
	\frac{\sqrt{n\Delta_nh_n}
	\{\widehat B_{n,b}^{12}g_0(0)-B_{\Delta_n}^{\mathrm{FD}}g_0\}}
	{\sigma_{B,12,1}^{g_0}(0)}.
	\]
	The spatial undersmoothing condition in \eqref{eq:first-order-clt-rates} justifies centering at \(B_{\Delta_n}^{\mathrm{FD}}g_0\); adding its temporal undersmoothing condition permits centering at \(B_{12}g_0(0)\).
	\endgroup

	\begingroup
	\begin{table}[H]

		\centering
		\scriptsize
		\caption{First-order common-design distributional check for \(g_0\), \(x=0\), block \(1\to2\), based on \(100\) replications. Parentheses contain Monte Carlo standard errors of the replication means; RMSE is computed relative to \(B_{\Delta_n}^{\mathrm{FD}}g_0\).}
		\label{tab:simulation-first-order-clt}
		\begin{tabular}{rrrrrrrr}
			\toprule
			\(n\) & \(\Delta_n\) & \(n\Delta_nh_n\)
			& \(B_{\Delta_n}^{\mathrm{FD}}\) & mean \(\widehat B_n\) (s.e.) & RMSE
			& mean/sd \(Z_B^{\mathrm{or}}\) & cov.\\
			\midrule
			\(300000\)  & \(0.15\) & \(2308.50\)  & \(0.5094\) & \(0.5131\;(0.0022)\) & \(0.0219\) & \(0.156/0.911\) & \(0.99\)\\
			\(1500000\) & \(0.12\) & \(7074.00\)  & \(0.5177\) & \(0.5190\;(0.0013)\) & \(0.0129\) & \(0.097/0.951\) & \(0.98\)\\
			\(6400000\) & \(0.10\) & \(20160.00\) & \(0.5231\) & \(0.5229\;(0.0007)\) & \(0.0070\) & \(-0.036/0.880\) & \(0.94\)\\
			\bottomrule
		\end{tabular}
	\end{table}
	The replication means track the finite-difference oracles, the RMSE decreases, and the oracle statistics are centered near zero with standard deviations near one. Their \(95\%\) coverages are \(0.99,0.98,0.94\), with Monte Carlo standard errors \(0.010,0.014,0.024\).

	For \Cref{thm:second-order-recovery-clt}, we use the nonoverlapping estimator \(\widehat C_{n,\mathrm{nb}}^{12,(1)}g_0(0)\). By \eqref{eq:second-order-asymptotic-variance},
	\[
	\sigma_{C,\mathrm{nb},12,1}^{g_0}(0)
	=
	\sqrt{2}\,\sigma_{B,12,1}^{g_0}(0)
	=1.615042.
	\]
	Define
	\[
	Z_{C,b}^{\mathrm{or}}
	:=
	\frac{
		\sqrt{N_n\Delta_n^3h_n}
		\{\widehat C_{n,\mathrm{nb},b}^{12,(1)}g_0(0)
		-C_{\Delta_n}^{\mathrm{FD}}g_0\}
	}{
		\sigma_{C,\mathrm{nb},12,1}^{g_0}(0)
	},
	\]
	and define \(Z_{C,b}^{\mathrm{pl}}\) by replacing the denominator with the squared-second-difference estimator in \Cref{cor:second-order-studentization}.

	\begin{table}[H]

		\centering
		\scriptsize
		\caption{Second-order distributional check for the nonoverlapping common-design estimator, based on \(100\) replications. The pointwise oracle \(C_{\Delta_n}^{\mathrm{FD}}g_0\) is used for centering; the remaining local-polynomial smoothing bias is asymptotically negligible under \eqref{eq:second-order-clt-rates}.}
		\label{tab:simulation-second-order-clt}
		\begin{tabular}{rrrrrrrrrr}
			\toprule
			\(n\) & \(\Delta_n\) & \(N_n\Delta_n^3h_n\)
			& \(C_{\Delta_n}^{\mathrm{FD}}\) & mean \(\widehat C_{n,\mathrm{nb}}\) (s.e.) & RMSE
			& mean/sd \(Z_C^{\mathrm{or}}\) & cov.
			& mean/sd \(Z_C^{\mathrm{pl}}\) & cov.\\
			\midrule
			\(300000\)  & \(0.15\) & \(25.97\)  & \(-0.5613\) & \(-0.5484\;(0.0276)\) & \(0.2750\)
			& \(0.041/0.871\) & \(0.97\) & \(0.047/0.863\) & \(0.97\)\\
			\(1500000\) & \(0.12\) & \(50.93\)  & \(-0.5626\) & \(-0.5853\;(0.0196)\) & \(0.1960\)
			& \(-0.100/0.865\) & \(0.98\) & \(-0.098/0.868\) & \(0.98\)\\
			\(6400000\) & \(0.10\) & \(100.80\) & \(-0.5571\) & \(-0.5360\;(0.0173)\) & \(0.1737\)
			& \(0.131/1.077\) & \(0.94\) & \(0.133/1.076\) & \(0.94\)\\
			\bottomrule
		\end{tabular}
	\end{table}
	The RMSE decreases across the three designs. Both standardized statistics are centered near zero, their standard deviations approach one, and their coverages remain close to \(0.95\), supporting the factor-two variance and feasible studentization.

	\begin{figure}[H]
		\vspace{\baselineskip}
		\centering
		\includegraphics[width=0.94\textwidth]{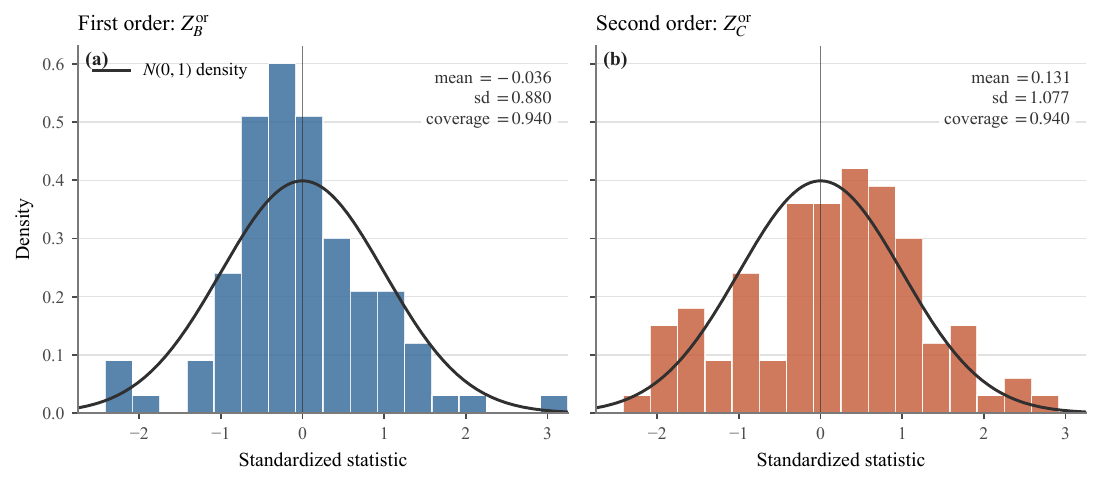}
		\caption{Shrinking-mesh distribution diagnostics for the largest design \((n,\Delta_n,h_n)=(6400000,0.10,0.0315)\), based on \(100\) replications. The panels show density histograms of the first-order oracle statistic \(Z_B^{\mathrm{or}}\) and the nonoverlapping second-order oracle statistic \(Z_C^{\mathrm{or}}\), with the standard normal density superimposed.}
		\label{fig:recovery-clt-histograms}
	\end{figure}
	\Cref{fig:recovery-clt-histograms} provides a direct shape comparison with the Gaussian limits. Both statistics are centered near zero and have \(0.94\) empirical coverage; the first-order statistic is modestly underdispersed, whereas the second-order statistic is close to unit dispersion.
	\endgroup

	{Because \(g_1(0)=g_2(0)=0\), their limits under the normalization of \Cref{thm:second-order-recovery-clt} are degenerate; hence \(g_0\) supplies the distributional check, while the expansion experiment validates the analytic terms for \(g_1,g_2\).}

	{Supplementary Section~S.7 reports all probes, both blocks, and the separate-lag experiment.}

	\section{Discussion}
	\label{sec:discussion}
	{The three targets have distinct roles: fixed-mesh blocks determine normalized conditional moments, first-order coefficients identify \(q,b,a\), and second-order coefficients give the \(O(\Delta^2)\) correction to block and regime-conditioned expectations. Because \(C_{ij}g\) contains interactions such as \((\nabla q_{ij})^\top a(\cdot,i)\nabla g\), its direct estimate also supplies a specification check against the value implied by separately fitted primitive coefficients and a target for second-order weak approximation or short-time bias correction.}

	{The common design cancels the localized level before smoothing and exposes a martingale array. Nonoverlapping second differences preserve within-block covariance, producing the factor two in \eqref{eq:second-order-asymptotic-variance}; both orders admit feasible confidence intervals. The supplementary arbitrary-order result remains an algebraic consistency benchmark.}

	\appendix
	\section{Auxiliary limit and regularity results}
	\label{sec:appendix-auxiliary-results}

	\subsection{A localized triangular-array central limit theorem}

	The following scalar result is the probabilistic input used in the
	fixed-mesh local-polynomial central limit theorem. Its proof is given
	in Supplementary Section~S.1.

	\begin{lemma}[Localized triangular-array CLT under exponential
	absolute regularity]
	\label{lem:appendix-kernel-clt}
	Let \(W=(W_k)_{k\ge1}\) be strictly stationary and absolutely regular
	with coefficients
	\[
	  \beta_W(r)\le C_\beta e^{-c_\beta r},\qquad r\ge1.
	\]
	Let \(h_n\downarrow0\) satisfy
	\[
	  n h_n^d\to\infty,
	  \qquad
	  \frac{(\log n)^2}{n h_n^d}\to0.
	\]
	For each \(n\), let \(\zeta_{k,n}=\phi_n(W_k)\), \(k\ge1\), be centered
	real random variables. Assume that, for a constant
	\(C_\zeta<\infty\) independent of \(n\),
	\[
	  |\zeta_{k,n}|\le C_\zeta,
	  \qquad
	  \mathbb E[\zeta_{1,n}^2]\le C_\zeta h_n^d,
	\]
	and, for every \(a\ge1\) and \(\ell\ge1\),
	\[
	  \operatorname{Var}\!\left(
	    \sum_{k=a}^{a+\ell-1}\zeta_{k,n}
	  \right)
	  \le C_\zeta\ell h_n^d.
	\]
	If
	\[
	  \operatorname{Var}\!\left(
	    (n h_n^d)^{-1/2}\sum_{k=1}^n\zeta_{k,n}
	  \right)
	  \to \tau^2
	\]
	for some \(\tau^2\in[0,\infty)\), then
	\[
	  (n h_n^d)^{-1/2}\sum_{k=1}^n\zeta_{k,n}
	  \xrightarrow{d}N(0,\tau^2).
	\]
	\end{lemma}

	\subsection{A martingale-difference-array central limit theorem}

	The first- and second-order recovery proofs use the following scalar
	form of the martingale-difference-array central limit theorem. The
	positive-variance assertion is the realized-square formulation of
	\citet[Theorem~2.3]{mcleish1974}, and the zero-variance assertion
	follows from Lenglart's inequality. Supplementary Section~S.4 verifies
	the hypotheses for the two recovery arrays.

	\begin{lemma}[Martingale-difference-array CLT]
	\label{lem:appendix-martingale-clt}
	For each \(n\), let \(m_n\ge1\), let
	\((\mathcal H_{s,n})_{s=0}^{m_n}\) be a filtration, and let
	\((X_{s,n})_{s=0}^{m_n-1}\) be a square-integrable
	martingale-difference array.
	Suppose that, for some \(\tau^2\in(0,\infty)\),
	\[
	  \max_{0\le s<m_n}|X_{s,n}|\xrightarrow{\mathbb P}0,
	  \qquad
	  \sup_n\mathbb E\!\left[
	    \max_{0\le s<m_n}X_{s,n}^2
	  \right]<\infty, \qquad
	  \sum_{s=0}^{m_n-1}X_{s,n}^2
	  \xrightarrow{\mathbb P}\tau^2.
	\]
	Then
	\[
	  \sum_{s=0}^{m_n-1}X_{s,n}
	  \xrightarrow{d}N(0,\tau^2).
	\]
	If instead
	\[
	  \sum_{s=0}^{m_n-1}
	  \mathbb E[X_{s,n}^2\mid\mathcal H_{s,n}]
	  \xrightarrow{\mathbb P}0,
	\]
	then the sum converges to zero in probability.
	\end{lemma}

	\subsection{A semigroup criterion for normalized smoothness}

	For a regime-indexed function \(h=(h_\ell)_{\ell\in S}\), write
	\[
	  \|h\|_{C_b^r}
	  :=
	  \max_{\ell\in S}
	  \sum_{|\alpha|\le r}
	  \sup_{y\in\mathbb R^d}|D^\alpha h_\ell(y)|.
	\]
	Let \(BUC^{p+1}\) denote the corresponding space in which all
	derivatives through order \(p+1\) are bounded and uniformly
	continuous.

	\begin{proposition}[Semigroup criterion for normalized smoothness]
	\label{prop:appendix-normalized-smoothness}
	Fix \(p\ge0\), \(j\in S\), and a probe \(g\), and put
	\(f:=f_g^{(j)}\). Suppose that \(f\), \(\mathcal Lf\), and
	\(\mathcal L^2f\) are bounded and continuous, that Dynkin's formula
	applies to \(f\) and \(\mathcal Lf\), and that, for some \(t_0>0\),
	the maps
	\[
	  s\longmapsto P_s\mathcal L^r f,
	  \qquad r=1,2,
	\]
	are continuous from \([0,2t_0]\) into \(BUC^{p+1}\) and satisfy
	\[
	  \max_{r=1,2}
	  \sup_{0\le s\le2t_0}
	  \|P_s\mathcal L^r f\|_{C_b^{p+1}}
	  <\infty.
	\]
	Then, for every \(i\in S\) and every open set
	\(U_x\subset\mathbb R^d\),
	\[
	  \sup_{0<t\le t_0}
	  \left\|
	    \frac{P_t^{ij}g-\delta_{ij}g}{t}
	  \right\|_{C^{p+1}(U_x)}
	  <\infty, \quad
	  \sup_{0<t\le t_0}
	  \left\|
	    \frac{P_{2t}^{ij}g-2P_t^{ij}g+\delta_{ij}g}{t^2}
	  \right\|_{C^{p+1}(U_x)}
	  <\infty.
	\]
	Consequently, Assumptions~A5 and~A6 hold for every sequence
	\((\Delta_n)\) contained in \((0,t_0]\).
	\end{proposition}

	\begin{remark}
	A sufficient model-level condition is that, for every
	\(\ell,m\in S\), the functions \(b(\cdot,\ell)\),
	\(a(\cdot,\ell)\), and \(q_{\ell m}\) belong to
	\(C_b^{p+4}(\mathbb R^d)\), each \(a(\cdot,\ell)\) is uniformly
	elliptic, and \(g\in C_b^{p+6}(\mathbb R^d)\). The proof and the
	verification for the numerical Ornstein--Uhlenbeck model are given
	in Supplementary Sections~S.5--S.6.
	\end{remark}

\providecommand{\bysame}{\leavevmode\hbox to3em{\hrulefill}\thinspace}
\providecommand{\MR}{\relax\ifhmode\unskip\space\fi MR }
\providecommand{\MRhref}[2]{%
	\href{http://www.ams.org/mathscinet-getitem?mr=#1}{#2}
}
\providecommand{\href}[2]{#2}

\clearpage
\phantomsection
\begin{center}
	{\Large\bfseries Online Supplement to\par}
	\vspace{0.5em}
	{\Large\itshape Nonparametric Inference for Semigroup Blocks of
	Switching Diffusions\par}
\end{center}
\vspace{2em}

\counterwithout{theorem}{section}
\setcounter{section}{0}
\setcounter{theorem}{0}
\setcounter{assumption}{0}
\setcounter{equation}{0}
\setcounter{figure}{0}
\setcounter{table}{0}
\mathtoolsset{showonlyrefs=false}
\renewcommand{\sectionname}{}
\renewcommand{\thesection}{S.\arabic{section}}
\renewcommand{\thetheorem}{S.\arabic{theorem}}
\renewcommand{\theassumption}{S.A\arabic{assumption}}
\renewcommand{\theequation}{S.\arabic{equation}}
\renewcommand{\thefigure}{S.\arabic{figure}}
\renewcommand{\thetable}{S.\arabic{table}}
\renewcommand{\theHsection}{supplement.\arabic{section}}
\renewcommand{\theHtheorem}{supplement.\arabic{theorem}}
\renewcommand{\theHassumption}{supplement.\arabic{assumption}}
\renewcommand{\theHequation}{supplement.\arabic{equation}}
\renewcommand{\theHfigure}{supplement.\arabic{figure}}
\renewcommand{\theHtable}{supplement.\arabic{table}}

\section{Proof of the localized triangular-array central limit theorem}
\label{sec:supp-auxiliary-clt}

This section proves Lemma~\ref{lem:appendix-kernel-clt} of the main
paper.

	\begin{proof}[Proof of Lemma~\ref{lem:appendix-kernel-clt}]
		Set \(N_n:=n h_n^d\), and define
		\[
		r_n:=\left\lfloor N_n^{1/4}(\log n)^{1/2}\right\rfloor.
		\]
		Choose \(A>0\) so large that, with \(s_n:=\lceil A\log n\rceil\),
		\[
		n\beta_W(s_n)\to0,
		\qquad
		\frac{r_n}{h_n^d}\beta_W(s_n)\to0.
		\]
		This is possible because \(N_n/(\log n)^2\to\infty\) implies
		\[
		\frac{r_n}{h_n^d}
		=\frac{nr_n}{N_n}
		\lesssim\frac{n}{\log n}
		\]
		eventually, while \(\beta_W(s_n)\lesssim n^{-Ac_\beta}\). Thus any \(A\) with \(Ac_\beta>1\) gives both displayed convergences.
		The condition \((\log n)^2/N_n\to0\) implies
		\[
		s_n=o(r_n),\qquad r_n=o(N_n^{1/2}).
		\]

		Let \(m_n:=\lfloor n/(r_n+s_n)\rfloor\). Split \(\{1,\ldots,n\}\) into \(m_n\) consecutive big blocks \(I_{\nu,n}\) of length \(r_n\), followed by small blocks \(J_{\nu,n}\) of length \(s_n\), and a terminal remainder of length at most \(r_n+s_n\). Define
		\[
		B_{\nu,n}:=\sum_{k\in I_{\nu,n}}\zeta_{k,n},
		\qquad
		S_{\nu,n}:=\sum_{k\in J_{\nu,n}}\zeta_{k,n},
		\qquad
		Q_n:=\sum_{k\in R_n}\zeta_{k,n},
		\]
		where \(R_n\) denotes the terminal remainder.

		We first prove that the normalized small-block and terminal-remainder sum is negligible in \(L^2\). By the block-variance assumption,
		\[
		\frac{1}{N_n}\sum_{\nu=1}^{m_n}\operatorname{Var}(S_{\nu,n})
		\lesssim
		\frac{m_ns_nh_n^d}{N_n}
		\lesssim \frac{s_n}{r_n}\to0,
		\]
		and
		\[
		\frac{1}{N_n}\operatorname{Var}(Q_n)
		\lesssim
		\frac{(r_n+s_n)h_n^d}{N_n}
		=\frac{r_n+s_n}{n}\to0.
		\]
		It remains only to control covariances between distinct small blocks. The bounded covariance inequality for absolutely regular sequences, in the direct coefficient form used below, is the standard consequence of the strong-mixing covariance inequality together with the comparison between strong and absolute regularity coefficients; see \citet[Sec.~1.1--1.2]{bradley2005strong}. It gives, for \(\mu>\nu\),
		\[
		\left|\operatorname{Cov}(S_{\nu,n},S_{\mu,n})\right|
		\lesssim
		s_n^2
		\beta_W\!\left(r_n+(\mu-\nu-1)(r_n+s_n)\right).
		\]
		Since \(\beta_W(r)\lesssim e^{-c_\beta r}\),
		\[
		\sum_{\ell\ge1}
		\beta_W\!\left(r_n+(\ell-1)(r_n+s_n)\right)
		\lesssim \beta_W(r_n).
		\]
		Therefore
		\[
		\begin{aligned}
			\frac{1}{N_n}\sum_{1\le\nu<\mu\le m_n}
			\left|\operatorname{Cov}(S_{\nu,n},S_{\mu,n})\right|
			&\lesssim
			\frac{m_ns_n^2}{N_n}\beta_W(r_n)\\
			&\lesssim
			\frac{s_n^2}{r_nh_n^d}\beta_W(r_n)
			\to0.
		\end{aligned}
		\]
		The last convergence follows because \(s_n=O(\log n)\), \(h_n^{-d}=n/N_n\le n\) eventually, and \(r_n/\log n\to\infty\), so \(\beta_W(r_n)\) decays faster than any negative power of \(n\). Hence
		\[
		\operatorname{Var}(A+B)
		\le 2\operatorname{Var}(A)+2\operatorname{Var}(B)
		\]
		Applied with
		\(A=N_n^{-1/2}\sum_{\nu=1}^{m_n}S_{\nu,n}\) and
		\(B=N_n^{-1/2}Q_n\), this inequality shows that no separate bound on the cross-covariance with the terminal remainder is needed. Therefore
		\[
		T_{n,S}:=N_n^{-1/2}\left(\sum_{\nu=1}^{m_n}S_{\nu,n}+Q_n\right)
		\xrightarrow{L^2}0.
		\]

		We now treat the big blocks. Distinct big blocks are separated by small blocks of length \(s_n\). The same covariance inequality gives, for \(\mu>\nu\),
		\[
		\left|\operatorname{Cov}(B_{\nu,n},B_{\mu,n})\right|
		\lesssim
		r_n^2
		\beta_W\!\left(s_n+(\mu-\nu-1)(r_n+s_n)\right).
		\]
		Therefore
		\[
		\begin{aligned}
			\frac{1}{N_n}
			\sum_{1\le \nu<\mu\le m_n}
			\left|\operatorname{Cov}(B_{\nu,n},B_{\mu,n})\right|
			&\lesssim
			\frac{m_n r_n^2}{N_n}\beta_W(s_n)\\
			&\lesssim
			\frac{r_n}{h_n^d}\beta_W(s_n)
			\to0.
		\end{aligned}
		\]
		Let \(T_{n,B}:=N_n^{-1/2}\sum_{\nu=1}^{m_n}B_{\nu,n}\).
		By the coupling lemma of \citet{berbee1979random}, possibly on an enlarged probability space, there exist independent random variables \(B_{\nu,n}^\ast\), \(1\le\nu\le m_n\), with \(B_{\nu,n}^\ast\) distributed as \(B_{\nu,n}\), such that
		\[
		\mathbb P\!\left(\exists\,\nu\le m_n:\ B_{\nu,n}\ne B_{\nu,n}^\ast\right)
		\le
		m_n\beta_W(s_n)
		\le n\beta_W(s_n)\to0.
		\]
		Thus \(T_{n,B}\) and \(N_n^{-1/2}\sum_{\nu=1}^{m_n}B_{\nu,n}^\ast\) have the same weak limits.

		The variance transfer is as follows. The original normalized sum satisfies
		\[
		N_n^{-1/2}\sum_{k=1}^n\zeta_{k,n}=T_{n,B}+T_{n,S},
		\qquad T_{n,S}\xrightarrow{L^2}0,
		\]
		so, by the Cauchy--Schwarz inequality and the assumed convergence of the variance of the original normalized sum,
		\[
		\left|\operatorname{Var}\!\left(N_n^{-1/2}\sum_{k=1}^n\zeta_{k,n}\right)
		-\operatorname{Var}(T_{n,B})\right|\to0.
		\]
		Since the variables \(B_{\nu,n}^\ast\) are independent and have the same marginal distributions as \(B_{\nu,n}\),
		\[
		\begin{aligned}
			&\left|
			\operatorname{Var}(T_{n,B})
			-\operatorname{Var}\!\left(N_n^{-1/2}\sum_{\nu=1}^{m_n}B_{\nu,n}^\ast\right)
			\right| \\
			&\qquad =
			\frac{2}{N_n}
			\left|\sum_{1\le\nu<\mu\le m_n}
			\operatorname{Cov}(B_{\nu,n},B_{\mu,n})\right|
			\to0
		\end{aligned}
		\]
		by the big-block covariance estimate. Therefore
		\[
		\operatorname{Var}\!\left(N_n^{-1/2}\sum_{\nu=1}^{m_n}B_{\nu,n}^\ast\right)
		\to \tau^2.
		\]

		Each independent big-block sum satisfies \(|B_{\nu,n}^\ast|\lesssim r_n\). Since \(r_n/N_n^{1/2}=N_n^{-1/4}(\log n)^{1/2}\to0\),
		\[
		\max_{\nu\le m_n}
		\frac{|B_{\nu,n}^\ast|}{N_n^{1/2}}
		\lesssim\frac{r_n}{N_n^{1/2}}
		\to0.
		\]
		Hence, for every \(\varepsilon>0\), the Lindeberg condition for the independent triangular array is automatic.
		Indeed, for all large \(n\),
		\[
		\sum_{\nu=1}^{m_n}
		\mathbb E\!\left[
		\left(\frac{B_{\nu,n}^\ast}{N_n^{1/2}}\right)^2
		\mathbf 1_{\{|B_{\nu,n}^\ast|>\varepsilon N_n^{1/2}\}}
		\right]
		=0.
		\]

		If \(\tau^2>0\), the Lindeberg--Feller theorem yields
		\[
		N_n^{-1/2}\sum_{\nu=1}^{m_n}B_{\nu,n}^\ast
		\xrightarrow{d}N(0,\tau^2).
		\]
		The coupling step and \(T_{n,S}\to0\) in \(L^2\) then give the asserted convergence for the original normalized sum. If \(\tau^2=0\), the assumed variance convergence and Chebyshev's inequality give \(N_n^{-1/2}\sum_{k=1}^n\zeta_{k,n}\xrightarrow{\mathbb P}0\), which is the degenerate normal law \(N(0,0)\).
	\end{proof}

\clearpage
\section{Detailed fixed-mesh proofs}
\label{sec:supp-fixedmesh}
\begingroup

This section supplies the complete arguments behind \Cref{thm:lp-fixedmesh-consistency,thm:lp-fixedmesh-clt,thm:fixedmesh-normalized-moment} of the main paper. All notation and assumption labels are those of the main paper. The presentation separates the localized law of large numbers, the scalar fluctuation limit, the deterministic smoothing bias, and the final matrix and ratio linearizations.

	\subsection{Localized law of large numbers}

		\begingroup
		\begin{lemma}[Fixed-mesh localized law of large numbers]\label[lemma]{lem:lp-fixedmesh-lln}
			Fix \(\Delta>0\), \(p\ge0\), \(g\in C_b(\mathbb R^d)\), a regime pair \((i,j)\), and \(x\in\mathbb R^d\). Suppose \Cref{ass:mixing} and part~\textup{(i)} of \Cref{ass:local-densities} hold on \(U_x\). If \(h_n\to0\), then
			\[
			\begin{aligned}
				\widehat S_{n,h_n}^{i,(p)}(x)-S_{h_n,p}^{i}(x)
				&=O_{\mathbb P}\!\left((nh_n^d)^{-1/2}\right),\\
				\widehat T_{n,\Delta,h_n}^{ij,g,(p)}(x)
				-\mathbb E[\widehat T_{n,\Delta,h_n}^{ij,g,(p)}(x)]
				&=O_{\mathbb P}\!\left((nh_n^d)^{-1/2}\right).
			\end{aligned}
			\]
		\end{lemma}

		\begin{proof}
			A generic unscaled coordinate of either moment has the form
			\[
			U_{k,n}=\Phi_n(Z_{k-1},Z_k),
			\]
			where \(\Phi_n\) is uniformly bounded and vanishes unless \(Y_{k-1}\in x+h_nD_K\) and \(\Lambda_{t_{k-1}}=i\), for a fixed compact set \(D_K\). Part~\textup{(i)} of \Cref{ass:local-densities} and the local boundedness of \(V\) give
			\[
			\mathbb E[U_{1,n}^2]=O(h_n^d).
			\]
			For \(z\in E\), define \(d_n(z):=\mathbb E_z[\Phi_n(Z_0,Z_1)]\). Since \(\Phi_n\) is uniformly bounded and \(V\ge1\), one has \(\|d_n\|_{V,\infty}\lesssim1\). For \(r\ge2\), the Markov property at time \(\Delta\) and stationarity yield
			\[
			\operatorname{Cov}(U_{1,n},U_{1+r,n})
			=
			\mathbb E\!\left[
			U_{1,n}\left\{P_{(r-1)\Delta}d_n(Z_1)-\nu(d_n)\right\}
			\right].
			\]
			By \Cref{ass:mixing},
			\[
			\left|P_{(r-1)\Delta}d_n(z)-\nu(d_n)\right|
			\lesssim e^{-\lambda_{\rm erg}(r-1)\Delta}V(z).
			\]
			Moreover, \Cref{ass:mixing} implies \(P_\Delta V(z)\lesssim V(z)\). Conditioning on \(Z_0\), using the localization of \(U_{1,n}\), and then applying part~\textup{(i)} of \Cref{ass:local-densities} and the local boundedness of \(V\), we obtain
			\[
			\mathbb E[|U_{1,n}|V(Z_1)]
			\lesssim
			\mathbb E\!\left[
			\mathbf 1_{\{Y_0\in x+h_nD_K,\,\Lambda_{t_0}=i\}}V(Z_0)
			\right]
			\lesssim h_n^d.
			\]
			Consequently,
			\[
			|\operatorname{Cov}(U_{1,n},U_{1+r,n})|
			\lesssim h_n^d e^{-\lambda_{\rm erg}(r-1)\Delta},
			\qquad r\ge2.
			\]
			For \(r=1\), boundedness and localization give a covariance of order \(O(h_n^d)\). The stationary covariance decomposition therefore gives
			\[
			\operatorname{Var}\!\left(
			\frac{1}{nh_n^d}\sum_{k=1}^nU_{k,n}
			\right)
			\lesssim \frac{1}{nh_n^d}.
			\]
			Applying Chebyshev's inequality coordinatewise proves both assertions.
		\end{proof}
		\endgroup

		\subsection{Scalar fluctuation limit}
		\begin{lemma}[Scalar fixed-mesh central limit theorem]\label[lemma]{lem:lp-scalar-clt}
			Fix \(\Delta>0\), \(p\ge0\), \(g\in C_b(\mathbb R^d)\), a regime pair \((i,j)\), and \(x\in\mathbb R^d\). Suppose \Cref{ass:wellposedness,ass:mixing}, parts~\textup{(i)}--\textup{(ii)} of \Cref{ass:local-densities}, and parts~\textup{(i)} and \textup{(iii)} of \Cref{ass:fixedmesh-smoothness} hold on \(U_x\), and suppose
			\[
				h_n\to0,
				\qquad
				\frac{(\log n)^2}{nh_n^d}\to0.
			\]
			Then
			\[
			\frac{1}{\sqrt{nh_n^d}}
			\sum_{k=1}^n
			\left\{
			\zeta_{k,n}^{ij,g,(p)}(x)
			-\mathbb E[\zeta_{k,n}^{ij,g,(p)}(x)]
			\right\}
			\xrightarrow{d}
			N\!\left(0,\sigma_{ij,\Delta,p}^{2,g}(x)\right).
			\]
		\end{lemma}

		\begin{proof}
			Since \(U_x\) is open, \(K\) is compactly supported, and \(x\in U_x\), one has \(x+h_n\operatorname{supp}(K)\subset U_x\) for all sufficiently large \(n\). By the change of variables \(y=x+h_nu\),
			\[
			S_{h_n,p}^{i}(x)
			=\int_{\mathbb R^d}K(u)\psi_p(u)\psi_p(u)^\top \varpi_i(x+h_nu)\,du.
			\]
			By part~\textup{(i)} of \Cref{ass:local-densities}, \(\varpi_i\) is continuous at \(x\), and
			\[
			S_{h_n,p}^{i}(x)\to \varpi_i(x)M_p(K),
			\]
			as \(n\to\infty\). Part~\textup{(i)} of \Cref{ass:local-densities} also gives \(\varpi_i(x)>0\), so the limit is positive definite. Hence \(\|(S_{h_n,p}^{i}(x))^{-1}\|\lesssim1\) for all sufficiently large \(n\). By \Cref{lem:lp-fixedmesh-lln}, the empirical design and response moments satisfy the stochastic bounds used below. For the scalar central limit theorem, we sharpen the covariance estimate to show that the nonzero-lag covariances are \(o(h_n^d)\). A generic coordinate of either unscaled summand can be written as
			\[
			U_{k,n}
			=
			K\!\left(\frac{Y_{k-1}-x}{h_n}\right)
			\mathbf 1_{\{\Lambda_{t_{k-1}}=i\}}A_{k,n},
			\]
			where \(A_{k,n}\) is uniformly bounded whenever the kernel factor is nonzero. For a design coordinate, \(A_{k,n}\) is a product of two components of \(\psi_p((Y_{k-1}-x)/h_n)\); for a response coordinate, it is \(R_k^{j,g}\) times one component of this vector. Because \(K\) is compactly supported, there is a compact set \(D_K\) such that \(U_{k,n}=0\) unless \(Y_{k-1}\in x+h_nD_K\). The local stationary-density bound therefore gives the zero-lag estimate
			\[
			\mathbb E[U_{1,n}^2]
			\le
			C\int_{x+h_nD_K}\varpi_i(y)\,dy
			\le Ch_n^d.
			\]

			\begingroup
			For \(r\ge1\), both \(U_{1,n}\) and \(U_{1+r,n}\) can be nonzero only if \(Y_0,Y_r\in x+h_nD_K\) and the corresponding initial regimes equal \(i\). Hence the direct small-ball bound in Assumption~A3\textup{(ii)} yields, uniformly in \(r\ge1\),
			\[
			\mathbb E\!\left[|U_{1,n}U_{1+r,n}|\right]
			\le
			C\mathbb P\!\left(
			Y_0,Y_r\in x+h_nD_K,
			\Lambda_{t_0}=\Lambda_{t_r}=i
			\right)
			\le Ch_n^{2d}.
			\]
			\endgroup
			Since \(\mathbb E|U_{1,n}|=O(h_n^d)\), the same order holds for the covariance. This is the near-lag estimate. For the far lags, \(U_{1,n}\) is measurable with respect to \(\sigma(Z_0,Z_1)\), while \(U_{1+r,n}\) is measurable with respect to \(\sigma(Z_r,Z_{r+1})\). The bounded covariance inequality for absolutely regular sequences (see \citet[Sec.~1]{bradley2005strong}) and the exponential bound for \((Z_k)\) give constants \(C,c>0\) such that
			\[
			\left|\operatorname{Cov}(U_{1,n},U_{1+r,n})\right|
			\le Ce^{-cr},
			\qquad r\ge2.
			\]
			Set
			\[
			M_n:=\left\lfloor\frac{d+1}{c}|\log h_n|\right\rfloor.
			\]
			Then
			\[
			\sum_{r=1}^{M_n}
			\left|\operatorname{Cov}(U_{1,n},U_{1+r,n})\right|
			\lesssim M_nh_n^{2d}=o(h_n^d),
			\]
			whereas
			\[
			\sum_{r>M_n}
			\left|\operatorname{Cov}(U_{1,n},U_{1+r,n})\right|
			\lesssim e^{-cM_n}=O(h_n^{d+1})=o(h_n^d).
			\]
			Hence
			\[
				\sum_{r=1}^{\infty}
				\left|\operatorname{Cov}(U_{1,n},U_{1+r,n})\right|
				=o(h_n^d).
			\]
			This sharper estimate is applied below to the scalar array.

			We next identify the scalar variance. Put \(m:=m_{\Delta,g}^{ij}(x)\) and
			\[
			H(y)
			:=
			P_\Delta^{ij}(g^2)(y)
			-2mP_\Delta^{ij}g(y)+m^2.
			\]
			Conditional on \(Y_{k-1}=y\) and \(\Lambda_{t_{k-1}}=i\),
			\[
			\mathbb E\!\left[
			\left(R_k^{j,g}-m\right)^2
			\,\middle|\,
			Y_{k-1}=y,\Lambda_{t_{k-1}}=i
			\right]
			=H(y).
			\]
			{By parts~\textup{(i)} and \textup{(iii)} of \Cref{ass:fixedmesh-smoothness}, \(H\) is continuous at \(x\) and \(H(x)=v_{\Delta,g}^{ij}(x)\).} Define
			\[
			\ell_n(u):=
			e_0^\top\bigl(S_{h_n,p}^{i}(x)\bigr)^{-1}\psi_p(u).
			\]
			The inverse-matrix convergence above implies that \(\ell_n\) is uniformly bounded on \(D_K\). Conditioning on \((Y_0,\Lambda_{t_0})\) and changing variables \(y=x+h_nu\) give the exact zero-lag expression
			\[
			\frac{1}{h_n^d}
			\mathbb E\!\left[
			\bigl(\zeta_{1,n}^{ij,g,(p)}(x)\bigr)^2
			\right]
			=
			\int_{\mathbb R^d}
			K(u)^2\ell_n(u)^2
			H(x+h_nu)\varpi_i(x+h_nu)\,du.
			\]
			Since
			\[
			\ell_n(u)
			\to
			\varpi_i(x)^{-1}
			e_0^\top M_p(K)^{-1}\psi_p(u)
			\]
			uniformly on \(D_K\), the dominated convergence theorem yields
			\[
			\frac{1}{h_n^d}
			\mathbb E\!\left[
			\bigl(\zeta_{1,n}^{ij,g,(p)}(x)\bigr)^2
			\right]
			\to
			\sigma_{ij,\Delta,p}^{2,g}(x).
			\]
			Localization also gives \(\mathbb E[\zeta_{1,n}^{ij,g,(p)}(x)]=O(h_n^d)\). Therefore, with
			\[
			\xi_{k,n}
			:=
			\zeta_{k,n}^{ij,g,(p)}(x)
			-\mathbb E[\zeta_{k,n}^{ij,g,(p)}(x)],
			\]
			we have
			\[
			h_n^{-d}\operatorname{Var}(\xi_{1,n})
			\to
			\sigma_{ij,\Delta,p}^{2,g}(x).
			\]

			The scalar variables satisfy the same covariance estimates as the generic coordinates above:
			\[
			\left|\operatorname{Cov}(\xi_{1,n},\xi_{1+r,n})\right|
			\lesssim
			\begin{cases}
			h_n^{2d},&1\le r\le M_n,\\
			e^{-cr},&r>M_n.
			\end{cases}
			\]
			Consequently,
			\[
			\frac{1}{h_n^d}
			\sum_{r=1}^{\infty}
			\left|\operatorname{Cov}(\xi_{1,n},\xi_{1+r,n})\right|
			\to0.
			\]
			The stationary variance decomposition also gives, uniformly over all consecutive blocks of length \(\ell\),
			\[
			\operatorname{Var}\!\left(
			\sum_{k=a}^{a+\ell-1}\xi_{k,n}
			\right)
			\le
			\ell\operatorname{Var}(\xi_{1,n})
			+
			2\ell\sum_{r=1}^{\infty}
			\left|\operatorname{Cov}(\xi_{1,n},\xi_{1+r,n})\right|
			\lesssim \ell h_n^d.
			\]

			Finally, set \(W_k:=(Z_{k-1},Z_k)\). Then \(\xi_{k,n}\) is a bounded measurable function of \(W_k\), and \((W_k)\) is strictly stationary. Moreover, for \(r\ge2\), its absolute-regularity coefficients satisfy \(\beta_W(r)\le\beta_\Delta(r-1)\), and hence decay exponentially. The preceding zero- and nonzero-lag estimates imply
			\[
			\begin{aligned}
			\operatorname{Var}\!\left(
			\frac{1}{\sqrt{nh_n^d}}\sum_{k=1}^n\xi_{k,n}
			\right)
			&=
			\frac{\operatorname{Var}(\xi_{1,n})}{h_n^d}
			+
			\frac{2}{h_n^d}
			\sum_{r=1}^{n-1}
			\left(1-\frac{r}{n}\right)
			\operatorname{Cov}(\xi_{1,n},\xi_{1+r,n})\\
			&\to
			\sigma_{ij,\Delta,p}^{2,g}(x).
			\end{aligned}
			\]
			All conditions of Lemma~\ref{lem:appendix-kernel-clt} are therefore satisfied, and the asserted convergence follows.
		\end{proof}

	\subsection{Smoothing bias}
	{\begin{lemma}[Smoothing bias of the deterministic surrogate]\label[lemma]{lem:fixedmesh-bias}
		{Fix \(\Delta>0\), \(p\ge0\), \(g\in C_b(\mathbb R^d)\), a regime pair \((i,j)\), and a design point \(x\in\mathbb R^d\). Suppose part~\textup{(i)} of \Cref{ass:local-densities} and part~\textup{(ii)} of \Cref{ass:fixedmesh-smoothness} hold on \(U_x\).} Then
		\[
		R_{\Delta,h_n,p}^{ij}g(x)-P_\Delta^{ij}g(x)=O(h_n^{p+1}),
		\qquad n\to\infty.
		\]
	\end{lemma}

	\begin{proof}
		Set \(T_{\Delta,h_n,p}^{ij,g}(x):=\mathbb E[\widehat T_{n,\Delta,h_n}^{ij,g,(p)}(x)]\). For \(\alpha\in\mathcal A_p\), define
		\[
		b_{h_n,p,\alpha}(x):=\frac{h_n^{|\alpha|}}{\alpha!}D^\alpha m_{\Delta,g}^{ij}(x),
		\qquad
		b_{h_n,p}(x):=(b_{h_n,p,\alpha}(x))_{\alpha\in\mathcal A_p}.
		\]
		Taylor's theorem gives, uniformly for \(u\in\operatorname{supp}(K)\),
		\[
		m_{\Delta,g}^{ij}(x+h_nu)=\psi_p(u)^\top b_{h_n,p}(x)+h_n^{p+1}\rho_{h_n}(u;x),
		\]
		where \(\rho_{h_n}(\cdot;x)\) is uniformly bounded on \(\operatorname{supp}(K)\). Substitution yields
		\[
		T_{\Delta,h_n,p}^{ij,g}(x)=S_{h_n,p}^{i}(x)b_{h_n,p}(x)+h_n^{p+1}r_{h_n,p}(x),
		\qquad
		\|r_{h_n,p}(x)\|\lesssim1.
		\]
		Therefore
		\[
		\begin{aligned}
			R_{\Delta,h_n,p}^{ij}g(x)
			&=e_0^\top (S_{h_n,p}^{i}(x))^{-1}T_{\Delta,h_n,p}^{ij,g}(x)\\
			&=e_0^\top b_{h_n,p}(x)
			+h_n^{p+1}e_0^\top (S_{h_n,p}^{i}(x))^{-1}r_{h_n,p}(x).
		\end{aligned}
		\]
		Since \(e_0^\top b_{h_n,p}(x)=m_{\Delta,g}^{ij}(x)=P_\Delta^{ij}g(x)\), the \(O(h_n^{p+1})\) bound follows. This completes the proof.
	\end{proof}}

\subsection{Proofs of the main-text fixed-mesh results}

	\begin{proof}[Proof of \Cref{thm:lp-fixedmesh-consistency}]
			Write \(S_n:=S_{h_n,p}^{i}(x)\) and \(T_n:=\mathbb E[\widehat T_{n,\Delta,h_n}^{ij,g,(p)}(x)]\). By stationarity and the kernel change of variables,
			\[
			S_n\to\varpi_i(x)M_p(K),
			\qquad
			T_n\to\varpi_i(x)P_\Delta^{ij}g(x)\eta_p(K).
			\]
			Part~\textup{(i)} of \Cref{ass:local-densities} makes the limiting design matrix positive definite, while part~\textup{(i)} of \Cref{ass:fixedmesh-smoothness} gives the displayed limit for \(T_n\). The stochastic bounds in \Cref{lem:lp-fixedmesh-lln} and \(nh_n^d\to\infty\) imply
			\[
			\widehat S_{n,h_n}^{i,(p)}(x)\xrightarrow{\mathbb P}\varpi_i(x)M_p(K),
			\qquad
			\widehat T_{n,\Delta,h_n}^{ij,g,(p)}(x)
			\xrightarrow{\mathbb P}
			\varpi_i(x)P_\Delta^{ij}g(x)\eta_p(K).
			\]
			Therefore the design matrix is invertible with probability tending to one. Since \(M_p(K)e_0=\eta_p(K)\), direct substitution into \eqref{eq:lp-est} and the continuous mapping theorem prove consistency.
	\end{proof}

	\begin{proof}[Proof of \Cref{thm:lp-fixedmesh-clt}]
			Write \(S_n:=S_{h_n,p}^{i}(x)\), \(T_n:=\mathbb E[\widehat T_{n,\Delta,h_n}^{ij,g,(p)}(x)]\), and \(b_n:=S_n^{-1}T_n\). By \Cref{thm:lp-fixedmesh-consistency}, the design matrix is invertible with probability tending to one. On this event, the exact identity
			\[
			\begin{aligned}
				\widehat P_{\Delta,h_n}^{ij}g(x)
				-R_{\Delta,h_n,p}^{ij}g(x)
				&=
				e_0^\top
				\bigl(\widehat S_{n,h_n}^{i,(p)}(x)\bigr)^{-1}\\
				&\quad\times
				\left[
				\widehat T_{n,\Delta,h_n}^{ij,g,(p)}(x)-T_n
				-\left\{\widehat S_{n,h_n}^{i,(p)}(x)-S_n\right\}b_n
				\right]
			\end{aligned}
			\]
			holds. The vector expansion in the proof of \Cref{lem:fixedmesh-bias} gives
			\[
			b_n=P_\Delta^{ij}g(x)e_0+O(h_n).
			\]
			Moreover, \Cref{lem:lp-scalar-clt} and the inverse identity
			\(A^{-1}-B^{-1}=B^{-1}(B-A)A^{-1}\) yield
			\[
			\bigl(\widehat S_{n,h_n}^{i,(p)}(x)\bigr)^{-1}
			-S_n^{-1}
			=
			O_{\mathbb P}\!\left((nh_n^d)^{-1/2}\right).
			\]
			Substituting these two estimates into the exact identity and using the definition of \(\zeta_{k,n}^{ij,g,(p)}(x)\) gives the scalar stochastic expansion
			\[
			\begin{aligned}
				\widehat P_{\Delta,h_n}^{ij}g(x)
				-R_{\Delta,h_n,p}^{ij}g(x)
				=
				\frac{1}{nh_n^d}
				\sum_{k=1}^n
				\left\{
				\zeta_{k,n}^{ij,g,(p)}(x)
				-\mathbb E[\zeta_{k,n}^{ij,g,(p)}(x)]
				\right\}
				+
				o_{\mathbb P}\!\left((nh_n^d)^{-1/2}\right).
			\end{aligned}
			\]
			Thus \Cref{lem:lp-scalar-clt} implies
			\[
			\sqrt{nh_n^d}
			\left(
			\widehat P_{\Delta,h_n}^{ij}g(x)
			-R_{\Delta,h_n,p}^{ij}g(x)
			\right)
			\xrightarrow{d}
			N\!\left(0,\sigma_{ij,\Delta,p}^{2,g}(x)\right).
			\]
			Finally, \Cref{lem:fixedmesh-bias} and the undersmoothing condition give
			\[
			\sqrt{nh_n^d}
			\left(
			R_{\Delta,h_n,p}^{ij}g(x)-P_\Delta^{ij}g(x)
			\right)
			=o(1).
			\]
			Slutsky's theorem completes the proof.
	\end{proof}

\subsubsection{Normalized conditional moment}

	\begin{proof}
		For \(c=(c_1,c_2)^\top\in\mathbb R^2\), the linear combination of the two terminal responses is
		\[
		c_1R_k^{j,g}+c_2R_k^{j,1}
		=R_k^{j,c_1g+c_2}.
		\]
		The regression function of the combined probe is \(c_1P_\Delta^{ij}g+c_2P_\Delta^{ij}1\), and its conditional second moment is
		\[
		c_1^2P_\Delta^{ij}(g^2)
		+2c_1c_2P_\Delta^{ij}g
		+c_2^2P_\Delta^{ij}1.
		\]
		Thus the smoothness and continuity conditions required in the proof of \Cref{lem:lp-scalar-clt} hold for every fixed \(c\). That proof also covers zero limiting variance, in which case the scalar limit is degenerate. Applying its scalar linearization and \Cref{lem:fixedmesh-bias} to \(c_1g+c_2\), then using the Cram\'er--Wold device, gives the displayed bivariate limit. Direct conditional covariance calculation gives \(\Omega_{\Delta,g}^{ij}(x)\). By \Cref{thm:lp-fixedmesh-consistency} applied to the probe \(1\), \(\widehat P_{\Delta,h_n}^{ij}1(x)\xrightarrow{\mathbb P}\pi_\Delta^{ij}(x)>0\). Hence the ratio is defined with probability tending to one. The gradient of \((u,v)\mapsto u/v\) at \((\mu_{\Delta,g}^{ij}(x),\pi_\Delta^{ij}(x))\) is
		\[
		\frac{1}{\pi_\Delta^{ij}(x)}
		\begin{pmatrix}1\\-M_\Delta^{ij}g(x)\end{pmatrix}.
		\]
		The multivariate delta method therefore gives the ratio limit. Multiplication of this gradient by \(\Sigma_{\Delta,g}^{ij}(x)\) yields the displayed expression for \(\tau_{ij,\Delta,g}^2(x)\).
	\end{proof}

\subsubsection{Plug-in variance consistency}

Use the same kernel \(K\), bandwidth \(h_n\), and degree \(p\) in
\(\widehat P_{\Delta,h_n}^{ij}g(x)\) and
\(\widehat P_{\Delta,h_n}^{ij}(g^2)(x)\), and define the local stationary-density estimator
\[
\widehat\varpi_{i,n}(x)
:=
\frac{1}{nh_n^d}\sum_{k=1}^n
K\!\left(\frac{Y_{k-1}-x}{h_n}\right)
\mathbf 1_{\{\Lambda_{t_{k-1}}=i\}}.
\]
Write \([u]_+:=\max\{u,0\}\), recall
\[
\kappa_p(K)
:=
e_0^\top M_p(K)^{-1}Q_p(K)M_p(K)^{-1}e_0,
\]
and let
\[
\mathcal E_n
:=
\left\{
\widehat S_{n,h_n}^{i,(p)}(x)\ \text{is invertible},
\widehat\varpi_{i,n}(x)>0
\right\}.
\]
On \(\mathcal E_n\), define
\begin{equation}\label{eq:supp-fixedmesh-plugin-variance}
\widehat\sigma_{ij,\Delta,p,n}^{2,g}(x)
:=
\frac{
\left[
\widehat P_{\Delta,h_n}^{ij}(g^2)(x)
-\{\widehat P_{\Delta,h_n}^{ij}g(x)\}^2
\right]_+
}{\widehat\varpi_{i,n}(x)}
\kappa_p(K).
\end{equation}
Thus a negative raw conditional-variance estimate is truncated at zero.

\begin{proposition}[Fixed-mesh plug-in variance consistency]
\label{prop:supp-fixedmesh-plugin-variance}
Fix \(\Delta>0\), \(p\ge0\), \(g\in C_b(\mathbb R^d)\), a regime pair \((i,j)\), and a design point \(x\). Suppose Assumptions~A1--A2, A3\textup{(i)}, and A4\textup{(i), (iii), and (iv)} hold on \(U_x\). If \(h_n\to0\) and \(nh_n^d\to\infty\), then
\[
\widehat\sigma_{ij,\Delta,p,n}^{2,g}(x)
\xrightarrow{\mathbb P}
\sigma_{ij,\Delta,p}^{2,g}(x).
\]
Moreover, the expression inside the positive part in \eqref{eq:supp-fixedmesh-plugin-variance} is positive with probability tending to one.
\end{proposition}

\begin{proof}
The localized law of large numbers in \Cref{lem:lp-fixedmesh-lln}, applied to the scalar design coordinate, gives
\(\widehat\varpi_{i,n}(x)\xrightarrow{\mathbb P}\varpi_i(x)>0\).
Applying \Cref{thm:lp-fixedmesh-consistency} first to \(g\), and then to \(g^2\) using A4\textup{(iii)}, yields
\[
\widehat P_{\Delta,h_n}^{ij}g(x)
\xrightarrow{\mathbb P}P_\Delta^{ij}g(x),
\qquad
\widehat P_{\Delta,h_n}^{ij}(g^2)(x)
\xrightarrow{\mathbb P}P_\Delta^{ij}(g^2)(x).
\]
Hence the raw variance estimate converges in probability to
\(v_{\Delta,g}^{ij}(x)>0\). It follows that \(\mathbb P(\mathcal E_n)\to1\), the positive-part truncation is inactive with probability tending to one, and the continuous mapping theorem applied to \eqref{eq:supp-fixedmesh-plugin-variance} proves the assertion.
\end{proof}

\endgroup

\clearpage
\section{Arbitrary-order separate-lag recovery}
\label{sec:supp-arbitrary-recovery}
\begingroup

For an arbitrary fixed order \(k\), ordinary forward differences of separately estimated blocks give a consistent algebraic recovery principle under strong sufficient rate conditions. This section collects the separate-lag estimator, its recovery-specific assumptions, the weighted rare-switch lemma, the finite-lag stochastic and smoothing bounds, and the complete consistency proof. The common-design first- and second-order central limit theorems remain in Section~5 of the main paper.

	{A single scaled block estimator does not isolate $A_k^{ij}g(x)$, because its short-time expansion also contains the lower-order block-generator coefficients $A_0^{ij}g(x),\ldots,A_{k-1}^{ij}g(x)$. Estimating and subtracting those terms separately would require lower-order plug-in errors to remain negligible after multiplication by $\Delta_n^{-k}$. Instead, we combine block estimators at equidistant lags with the ordinary forward-difference weights: these weights annihilate every lower-order power of the lag and retain the order-$k$ block-generator coefficient. The construction below implements this cancellation directly.}

	{The recovery estimator uses the sampled blocks at the equidistant lags \(r\Delta_n\), \(r=0,\ldots,k\). For each \(r=1,\ldots,k\), let}
	$(Y_{s,n},\Lambda_{t_{s,n}})$,
	$(Y_{s+r,n},\Lambda_{t_{s+r,n}})$,
	$0\le s\le n-r$,
	be the \(r\)-step pairs. Define the corresponding local-polynomial design matrix and response vector by
	\[
	\begin{aligned}
		\widehat S_{n,r,h_n}^{i,(p)}(x)
		&:=\frac{1}{(n-r+1)h_n^d}\sum_{s=0}^{n-r}
		K\!\left(\frac{Y_{s,n}-x}{h_n}\right)
		\mathbf 1_{\{\Lambda_{t_{s,n}}=i\}}
		\psi_p\!\left(\frac{Y_{s,n}-x}{h_n}\right)
		\psi_p\!\left(\frac{Y_{s,n}-x}{h_n}\right)^\top,\\
		\widehat T_{n,r\Delta_n,h_n}^{ij,g,(p)}(x)
		&:=\frac{1}{(n-r+1)h_n^d}\sum_{s=0}^{n-r}
		K\!\left(\frac{Y_{s,n}-x}{h_n}\right)
		\mathbf 1_{\{\Lambda_{t_{s,n}}=i,\ \Lambda_{t_{s+r,n}}=j\}}
		g(Y_{s+r,n})
		\psi_p\!\left(\frac{Y_{s,n}-x}{h_n}\right).
	\end{aligned}
	\]
	Let
	$S_{0,n,r}^{i,(p)}(x):=\mathbb E[\widehat S_{n,r,h_n}^{i,(p)}(x)]$ and
	$T_{0,n,r}^{ij,g,(p)}(x):=
	\mathbb E[\widehat T_{n,r\Delta_n,h_n}^{ij,g,(p)}(x)]$.
	Whenever \(\widehat S_{n,r,h_n}^{i,(p)}(x)\) and \(S_{0,n,r}^{i,(p)}(x)\) are invertible, set
	\begin{align*}
		&\widehat P_{r\Delta_n,h_n}^{ij,(p)}g(x)
		:=
		e_0^\top\left(\widehat S_{n,r,h_n}^{i,(p)}(x)\right)^{-1}
		\widehat T_{n,r\Delta_n,h_n}^{ij,g,(p)}(x), \\
		&R_{r\Delta_n,h_n,p}^{ij}g(x)
		:=
		e_0^\top\left(S_{0,n,r}^{i,(p)}(x)\right)^{-1}
		T_{0,n,r}^{ij,g,(p)}(x).
	\end{align*}
	{Put
		\[
		\widehat P_{0,h_n}^{ij,(p)}g(x):=A_0^{ij}g(x)=\delta_{ij}g(x).
		\]}
	{For \(r=0,\ldots,k\), define
		\[
		c_{k,r}:=(-1)^{k-r}\binom{k}{r}.
		\]
		These coefficients satisfy
		\[
		\sum_{r=0}^k c_{k,r}r^m=0,
		\qquad 0\le m<k,
		\qquad
		\sum_{r=0}^k c_{k,r}r^k=k!,
		\]
		where \(0^0=1\). The moment equations form a nonsingular Vandermonde system on the distinct nodes \(0,1,\ldots,k\), and hence determine the weights uniquely; see \citet{fornberg1988}.

		\begin{definition}[Finite-difference recovery estimator]
			Define the order-\(k\) recovery estimator by the ordinary forward difference
			\begin{equation}\label{eq:forward-difference-recovery}
				\widehat A_{k,n}^{(p),ij}g(x)
				:=
				\frac{1}{\Delta_n^k}
				\sum_{r=0}^k c_{k,r}
				\widehat P_{r\Delta_n,h_n}^{ij,(p)}g(x).
			\end{equation}
		\end{definition}
	}

	{The remaining two assumptions are recovery-specific.}

	{\begin{assumption}[Local recovery regularity]\label[assumption]{ass:recovery-regularity}
			For the recovery order \(k\), polynomial degree \(p\), regime pair \((i,j)\), bounded test function \(g\), and design point \(x\), the neighborhood \(U_x\) satisfies:
			\begin{enumerate}[label=\textup{(\roman*)}]
				\item the finite-lag sampled blocks satisfy
				\[
				\sup_{1\le r\le k}\sup_{n\ge1}
				\|P_{r\Delta_n}^{ij}g\|_{C^{p+1}(U_x)}<\infty;
				\]
				\item the coefficients \(A_\ell^{ij}g\), \(\ell=0,\ldots,k\), are well defined on \(U_x\), and
				\[
				\max_{1\le r\le k}\sup_{y\in U_x}
				\left|
				P_{r\Delta_n}^{ij}g(y)
				-\sum_{\ell=0}^k\frac{(r\Delta_n)^\ell}{\ell!}A_\ell^{ij}g(y)
				\right|
				=o(\Delta_n^k).
				\]
			\end{enumerate}
	\end{assumption}}

	{Unlike the one-mesh smoothness condition A4, part~\textup{(i)} of \Cref{ass:recovery-regularity} requires derivative control uniformly over \(r\Delta_n\downarrow0\), \(1\le r\le k\).}

	{\begin{corollary}\label[corollary]{cor:model-recovery-expansion}
			Fix \(k\ge1\), \(j\in S\), and \(g\in C_c^{2k}(\mathbb R^d)\). Suppose \Cref{ass:wellposedness} and the coefficient hypotheses of \Cref{prop:kth} hold. Then, for every \(i\in S\), design point \(x\), and neighborhood \(U_x\Subset\mathbb R^d\), part~\textup{(ii)} of \Cref{ass:recovery-regularity} holds.
		\end{corollary}

		\begin{proof}
			Apply the compact-uniform expansion in \Cref{prop:kth} on \(\overline U_x\) with \(\Delta=r\Delta_n\). For each fixed \(r=1,\ldots,k\), its remainder is \(o((r\Delta_n)^k)=o(\Delta_n^k)\), uniformly on \(U_x\). Taking the maximum over the finite set of lags preserves this order and proves part~\textup{(ii)} of \Cref{ass:recovery-regularity}.
		\end{proof}

		The recovery theorem is therefore stated under the abstract local expansion in part~\textup{(ii)} of \Cref{ass:recovery-regularity}, while \Cref{cor:model-recovery-expansion} supplies a model-level route from smooth coefficients. Part~\textup{(i)} of \Cref{ass:recovery-regularity} remains an independent smoothing requirement.}

	{The recovery argument combines the two parts of \Cref{ass:recovery-regularity} with a single finite-lag local-polynomial bound. The binomial weights in \eqref{eq:forward-difference-recovery} then cancel all lower-order Taylor terms.}

	{\begin{assumption}\label[assumption]{ass:rare-switch}
			{Let \(V:E\to[1,\infty)\) be the Lyapunov function in \Cref{ass:mixing}. There exists a constant \(C_{\rm sw}<\infty\) such that}
			\[
			q_{\ell m}(y)V(y,m)
			\le C_{\rm sw}V(y,\ell),
			\qquad y\in\mathbb R^d,\quad \ell\ne m.
			\]
		\end{assumption}

		\begin{lemma}\label[lemma]{lem:weighted-rare-switch}
			Under \Cref{ass:wellposedness,ass:mixing} and \Cref{ass:rare-switch}, there exists \(C<\infty\) such that, for every \(t\ge0\), \(y\in\mathbb R^d\), and \(i\ne j\),
			\[
			\mathbb E_{(y,i)}\!\left[
			V(X_t,\Lambda_t)\mathbf 1_{\{\Lambda_t=j\}}
			\right]
			\le CtV(y,i).
			\]
		\end{lemma}

		\begin{proof}
			First, \Cref{ass:mixing} and \(V\ge1\) imply, for every \(s\ge0\) and \(z\in E\),
			\[
			P_sV(z)
			\le \nu(V)+C_{\rm erg}e^{-\lambda_{\rm erg}s}V(z)
			\le C_VV(z),
			\qquad C_V:=\nu(V)+C_{\rm erg}.
			\]
			Thus the separate short-time semigroup bound previously used in the rare-switch condition follows already from exponential \(V\)-ergodicity.

			Let \(\tau:=\inf\{s>0:\Lambda_s\ne i\}\) be the first switching time for a process started from \(z=(y,i)\). If \(\Lambda_t=j\ne i\), then \(\tau\le t\). The strong Markov property at \(\tau\) therefore gives
			\[
			\mathbb E_z\!\left[V(Z_t)\mathbf 1_{\{\Lambda_t=j\}}\right]
			\le
			C_V\mathbb E_z\!\left[V(Z_\tau)\mathbf 1_{\{\tau\le t\}}\right].
			\]
			By the compensator formula for the first regime jump, continuity of \(X\) at \(\tau\), and \Cref{ass:rare-switch},
			\[
			\begin{aligned}
				\mathbb E_z\!\left[V(Z_\tau)\mathbf 1_{\{\tau\le t\}}\right]
				&=
				\mathbb E_z\!\left[
				\int_0^{t\wedge\tau}
				\sum_{m\ne i}q_{im}(X_s)V(X_s,m)\,ds
				\right]\\
				&\le (|S|-1)C_{\rm sw}
				\int_0^t
				\mathbb E_z\!\left[V(X_s,i)\mathbf 1_{\{s<\tau\}}\right]ds\\
				&\le (|S|-1)C_{\rm sw}\int_0^tP_sV(z)\,ds
				\le (|S|-1)C_{\rm sw}C_VtV(z).
			\end{aligned}
			\]
			Combining the last two displays proves the claim. In particular, for fixed \(k\), the estimate at \(t=r\Delta_n\), \(1\le r\le k\), is uniform in \(r\) and has order \(\Delta_nV(y,i)\).
	\end{proof}}

	{
		\begin{theorem}[Preliminary arbitrary-order recovery consistency]\label{thm:recovery}
			Fix \(k\ge1\), \(p\ge0\), a regime pair \((i,j)\), a bounded test function \(g\in C_b(\mathbb R^d)\), and a design point \(x\in\mathbb R^d\). {Suppose \Cref{ass:wellposedness,ass:mixing}, \Cref{ass:recovery-regularity,ass:rare-switch}, and only part~\textup{(i)} of \Cref{ass:local-densities} hold, with the local conditions imposed on \(U_x\).} If
			\[
			h_n^{p+1}=o(\Delta_n^k),
			\qquad
			\begin{cases}
				n\Delta_n^{2k+1}h_n^d\to\infty, & i=j,\\[1mm]
				n\Delta_n^{2k-1}h_n^d\to\infty, & i\ne j,
			\end{cases}
			\]
			then
			\[
			\widehat A_{k,n}^{(p),ij}g(x)\xrightarrow{\mathbb P}A_k^{ij}g(x).
			\]

		\end{theorem}

		The proof of \Cref{thm:recovery} is given after the following finite-lag stochastic and smoothing bounds for the sampled-block estimators.}

	\begin{lemma}[Finite-lag bounds]\label[lemma]{lem:finite-lag-recovery-bounds}
		{Fix \(k\ge1\), a polynomial degree \(p\ge0\), a regime pair \((i,j)\), a bounded test function \(g\in C_b(\mathbb R^d)\), and a design point \(x\in\mathbb R^d\). {Suppose \Cref{ass:wellposedness,ass:mixing} and \Cref{ass:rare-switch}, part~\textup{(i)} of \Cref{ass:local-densities}, and part~\textup{(i)} of \Cref{ass:recovery-regularity} hold, with the local conditions imposed on \(U_x\).}} Then, for the \(r\)-step estimators constructed above,
		\[
		\max_{1\le r\le k}
		\left|
		\widehat P_{r\Delta_n,h_n}^{ij,(p)}g(x)
		-R_{r\Delta_n,h_n,p}^{ij}g(x)
		\right|
		=O_{\mathbb P}\!\left(a_{n,p}^{ij}(x)\right),
		\]
		and
		\[
		\max_{1\le r\le k}
		\left|
		R_{r\Delta_n,h_n,p}^{ij}g(x)
		-P_{r\Delta_n}^{ij}g(x)
		\right|
		=O(h_n^{p+1}),
		\]
		{where the displayed sufficient stochastic upper bound is
			\[
			a_{n,p}^{ij}(x):=
			\begin{cases}
				(n\Delta_n h_n^d)^{-1/2}, & i=j,\\[1mm]
				\left(\dfrac{\Delta_n}{n h_n^d}\right)^{1/2}, & i\ne j.
			\end{cases}
			\]}
		Moreover,
		\[
		\mathbb P\!\left(
		\widehat S_{n,r,h_n}^{i,(p)}(x)
		\ \text{is invertible for all }1\le r\le k
		\right)\to1.
		\]

	\end{lemma}

	\begin{proof}
		{The proof follows the coordinatewise moment decomposition in Section~S.2.2 of the online supplement; we record the changes caused by the shrinking mesh, the fixed finite lags \(r=1,\ldots,k\), and the off-diagonal rare-switch factor. The covariance estimates below use the shrinking-mesh local \(V\)-norm conditioning argument, not the fixed-mesh two-location small-ball condition. Since \(r\le k\) and \(k\) is fixed, replacing \(n\) by \(n-r+1\) only changes constants uniformly over \(1\le r\le k\). Set}
		\[
		U_{s,n}(x):=\frac{Y_{s,n}-x}{h_n}.
		\]

		\begingroup
		We first control the design matrices. For \(\alpha,\beta\in\mathcal A_p\), let
		\[
		D_{s,n}^{\alpha\beta}
		:=
		K(U_{s,n}(x))\mathbf 1_{\{\Lambda_{t_{s,n}}=i\}}
		U_{s,n}(x)^{\alpha+\beta}.
		\]
		The \((\alpha,\beta)\)-coordinate of \(\widehat S_{n,r,h_n}^{i,(p)}(x)\) is \(((n-r+1)h_n^d)^{-1}\sum_{s=0}^{n-r}D_{s,n}^{\alpha\beta}\). The coordinate is uniformly bounded and vanishes unless
		\(Y_{s,n}\in x+h_n\operatorname{supp}(K)\) and \(\Lambda_{t_{s,n}}=i\). Thus \Cref{lem:shrinking-localization-vnorm} gives
		\[
		\operatorname{Var}(D_{0,n}^{\alpha\beta})=O(h_n^d),
		\qquad
		\sum_{m=1}^{n-1}
		\left|\operatorname{Cov}(D_{0,n}^{\alpha\beta},D_{m,n}^{\alpha\beta})\right|
		\lesssim \frac{h_n^d}{\Delta_n}.
		\]
		Here the factor \(h_n^d\) follows by integrating the \(V\)-norm mixing bound against the first localized coordinate, as shown explicitly in the proof of that lemma. The stationary variance decomposition and Chebyshev's inequality yield
		\[
		\max_{1\le r\le k}
		\left\|
		\widehat S_{n,r,h_n}^{i,(p)}(x)
		-\mathbb E[\widehat S_{n,r,h_n}^{i,(p)}(x)]
		\right\|
		=O_{\mathbb P}\!\left((n\Delta_n h_n^d)^{-1/2}\right).
		\]
		By the kernel change-of-variables formula,
		\[
		S_{0,n,r}^{i,(p)}(x)
		\to \varpi_i(x)M_p(K),
		\qquad 1\le r\le k,
		\]
		uniformly over \(1\le r\le k\). {Since \(\varpi_i(x)>0\), the limiting matrix is positive definite.} Combining this convergence with the preceding stochastic bound gives
		\[
		\mathbb P\!\left(
		\widehat S_{n,r,h_n}^{i,(p)}(x)
		\ \text{is invertible for all }1\le r\le k
		\right)\to1
		\]
		and
		\[
		\max_{1\le r\le k}
		\left\|
		\left(\widehat S_{n,r,h_n}^{i,(p)}(x)\right)^{-1}
		\right\|
		=O_{\mathbb P}(1).
		\]
		\endgroup

		We next control the response vectors. For \(\alpha\in\mathcal A_p\), define the uncentered \(r\)-step response coordinate
		\[
		H_{s,n,r}^{ij,\alpha}
		:=
		K(U_{s,n}(x))
		\mathbf 1_{\{\Lambda_{t_{s,n}}=i,\ \Lambda_{t_{s+r,n}}=j\}}
		g(Y_{s+r,n})U_{s,n}(x)^\alpha,
		\]
		and write \(\bar H_{s,n,r}^{ij,\alpha}:=H_{s,n,r}^{ij,\alpha}-\mathbb E[H_{s,n,r}^{ij,\alpha}]\). The \(\alpha\)-coordinate of \(\widehat T_{n,r\Delta_n,h_n}^{ij,g,(p)}(x)-T_{0,n,r}^{ij,g,(p)}(x)\) is \(((n-r+1)h_n^d)^{-1}\sum_{s=0}^{n-r}\bar H_{s,n,r}^{ij,\alpha}\).

		\begingroup
		Let \(D_K\subset\mathbb R^d\) be compact and contain \(\operatorname{supp}(K)\). For all large \(n\), \(x+h_nD_K\subset U_x\). Suppose first that \(i=j\). Localization gives
		\(\operatorname{Var}(H_{0,n,r}^{ii,\alpha})=O(h_n^d)\). For the overlapping lags \(1\le m\le r\), the same bound gives
		\[
		\left|\operatorname{Cov}(H_{0,n,r}^{ii,\alpha},H_{m,n,r}^{ii,\alpha})\right|
		\lesssim h_n^d.
		\]
		For \(m>r\), define
		\[
		\varphi_{n,r,\alpha}^{ii}(y,\ell)
		:=
		K\!\left(\frac{y-x}{h_n}\right)
		\mathbf 1_{\{\ell=i\}}
		\left(\frac{y-x}{h_n}\right)^\alpha
		P_{r\Delta_n}^{ii}g(y).
		\]
		The function \(\varphi_{n,r,\alpha}^{ii}\) is uniformly bounded and satisfies
		\[
		\operatorname{supp}(\varphi_{n,r,\alpha}^{ii})
		\subset (x+h_nD_K)\times\{i\},
		\qquad
		\|\varphi_{n,r,\alpha}^{ii}\|_{V,\infty}\lesssim1.
		\]
		Conditioning at time \(r\Delta_n\), applying the \(V\)-norm estimate in \Cref{ass:mixing}, and using the Markov property yield
		\[
		\left|
		\mathbb E\!\left[\bar H_{m,n,r}^{ii,\alpha}\mid Z_{r,n}\right]
		\right|
		\lesssim
		e^{-\lambda_{\rm erg}(m-r)\Delta_n}V(Z_{r,n}).
		\]
		Moreover, \(P_{r\Delta_n}V(z)\lesssim V(z)\) by \Cref{ass:mixing}. Hence the kernel localization, the local boundedness of \(V\), and the stationary density give
		\[
		\mathbb E\!\left[|H_{0,n,r}^{ii,\alpha}|V(Z_{r,n})\right]
		\lesssim
		\int_{x+h_nD_K}\varpi_i(y)P_{r\Delta_n}V(y,i)\,dy
		\lesssim h_n^d.
		\]
		Also \(|\mathbb E[H_{0,n,r}^{ii,\alpha}]|\lesssim h_n^d\), while stationarity implies \(\mathbb E[V(Z_{r,n})]=\nu(V)<\infty\). Therefore centering gives
		\[
		\mathbb E\!\left[|\bar H_{0,n,r}^{ii,\alpha}|V(Z_{r,n})\right]
		\lesssim h_n^d.
		\]
		The tower property now implies
		\[
		\left|\operatorname{Cov}(H_{0,n,r}^{ii,\alpha},H_{m,n,r}^{ii,\alpha})\right|
		\lesssim
		h_n^d e^{-\lambda_{\rm erg}(m-r)\Delta_n},
		\qquad m>r.
		\]
		Since \(r\le k\) and \(k\) is fixed, summing the overlapping and separated bounds gives
		\[
		\sum_{m=1}^{n-1}
		\left|\operatorname{Cov}(H_{0,n,r}^{ii,\alpha},H_{m,n,r}^{ii,\alpha})\right|
		\lesssim \frac{h_n^d}{\Delta_n}.
		\]
		The stationary variance decomposition gives
		\[
		\max_{1\le r\le k}
		\left\|
		\widehat T_{n,r\Delta_n,h_n}^{ii,g,(p)}(x)
		-\mathbb E[\widehat T_{n,r\Delta_n,h_n}^{ii,g,(p)}(x)]
		\right\|
		=O_{\mathbb P}\!\left((n\Delta_n h_n^d)^{-1/2}\right),
		\]
		because the number of response coordinates is fixed.
		\endgroup

		Now suppose \(i\ne j\). This is the only part not present in the fixed-mesh proof: the response coordinate contains a short-time transition from \(i\) to \(j\), which contributes an additional factor of order \(\Delta_n\). {The local boundedness in \Cref{ass:mixing} and part~\textup{(i)} of \Cref{ass:local-densities}, together with \Cref{lem:weighted-rare-switch}, gives}
		\[
		\mathbb E[(H_{0,n,r}^{ij,\alpha})^2]
		\lesssim
		\int_{x+h_nD_K}
		\varpi_i(y)\mathbb P_{(y,i)}(\Lambda_{r\Delta_n}=j)\,dy
		\lesssim \Delta_n h_n^d.
		\]
		For the overlapping lags \(1\le m\le r\), the same estimate yields
		\[
		\left|
		\operatorname{Cov}(H_{0,n,r}^{ij,\alpha},H_{m,n,r}^{ij,\alpha})
		\right|
		\lesssim \Delta_n h_n^d.
		\]
		Since \(r\le k\) and \(k\) is fixed, the overlapping lags contribute \(O(\Delta_n h_n^d)\).

		{For the separated lags \(m>r\), we use the Markov-conditioning step in \Cref{lem:shrinking-localization-vnorm}. Define, for \(z=(y,\ell)\in E\),}
		\[
		\varphi_{n,r,\alpha}(z)
		:=
		\mathbb E_z\!\left[
		K\!\left(\frac{X_0-x}{h_n}\right)
		\mathbf 1_{\{\Lambda_0=i,\ \Lambda_{r\Delta_n}=j\}}
		g(X_{r\Delta_n})
		\left(\frac{X_0-x}{h_n}\right)^\alpha
		\right].
		\]
		Conditioning on \(Z_{r,n}\) and using the Markov property gives
		\[
		\mathbb E\!\left[
		\bar H_{m,n,r}^{ij,\alpha}
		\mid Z_{r,n}
		\right]
		=
		P_{(m-r)\Delta_n}\varphi_{n,r,\alpha}(Z_{r,n})
		-\nu(\varphi_{n,r,\alpha}).
		\]
		For a measurable function \(f\), write
		\[
		\|f\|_{V,\infty}:=\sup_{z\in E}\frac{|f(z)|}{V(z)}.
		\]
		If \(z=(y,\ell)\) with \(\ell\ne i\), then \(\varphi_{n,r,\alpha}(z)=0\). If \(z=(y,i)\), then \(K((y-x)/h_n)\) vanishes unless \((y-x)/h_n\in\operatorname{supp}(K)\). On this set the kernel and polynomial factor are bounded, and \(V\ge1\) gives
		\[
		|\varphi_{n,r,\alpha}(y,i)|
		\lesssim
		\mathbb E_{(y,i)}
		\!\left[
		V(X_{r\Delta_n},\Lambda_{r\Delta_n})
		\mathbf 1_{\{\Lambda_{r\Delta_n}=j\}}
		\right].
		\]
		After division by \(V(y,i)\), {\Cref{lem:weighted-rare-switch}} yields
		\[
		\|\varphi_{n,r,\alpha}\|_{V,\infty}\lesssim \Delta_n.
		\]
		For any finite signed measure \(\mu\), the definition of the measure norm in \Cref{ass:mixing} implies
		\[
		|\mu(f)|\le \|f\|_{V,\infty}\|\mu\|_V.
		\]
		Indeed, if \(a=\|f\|_{V,\infty}\), then \(|f|\le aV\), and the assertion follows from the homogeneity of \(\|\mu\|_V=\sup_{|\psi|\le V}|\mu(\psi)|\).
		Applying this with \(f=\varphi_{n,r,\alpha}\) and
		\(\mu=P_{(m-r)\Delta_n}(z,\cdot)-\nu\), and then using \Cref{ass:mixing}, gives
		\[
		\left|
		P_{(m-r)\Delta_n}\varphi_{n,r,\alpha}(z)
		-\nu(\varphi_{n,r,\alpha})
		\right|
		\lesssim
		\Delta_n e^{-\lambda_{\mathrm{erg}}(m-r)\Delta_n}V(z).
		\]
		{The same estimate from \Cref{lem:weighted-rare-switch} controls the first response coordinate.} Indeed,
		\[
		\begin{aligned}
			\mathbb E\!\left[
			|H_{0,n,r}^{ij,\alpha}|\,V(Z_{r,n})
			\right]
			&\lesssim
			\int_{x+h_nD_K}\varpi_i(y)
			\mathbb E_{(y,i)}\!\left[
			V(X_{r\Delta_n},\Lambda_{r\Delta_n})
			\mathbf 1_{\{\Lambda_{r\Delta_n}=j\}}
			\right]dy  \\
			&\lesssim \Delta_n h_n^d.
		\end{aligned}
		\]
		Since \(V\ge1\) and stationarity gives \(\mathbb E[V(Z_{r,n})]=\nu(V)<\infty\), centering yields
		\[
		\mathbb E\!\left[
		|\bar H_{0,n,r}^{ij,\alpha}|\,V(Z_{r,n})
		\right]
		\lesssim \Delta_n h_n^d.
		\]
		For \(m>r\), \(H_{0,n,r}^{ij,\alpha}\) is measurable with respect to \(\sigma(Z_{u,n}:0\le u\le r)\). Hence the tower property, the Markov property, and the preceding conditional bound imply
		\[
		\begin{aligned}
			\left|
			\operatorname{Cov}(H_{0,n,r}^{ij,\alpha},H_{m,n,r}^{ij,\alpha})
			\right|
			&=
			\left|
			\mathbb E\!\left[
			\bar H_{0,n,r}^{ij,\alpha}
			\mathbb E\!\left[
			\bar H_{m,n,r}^{ij,\alpha}
			\mid \sigma(Z_{u,n}:0\le u\le r)
			\right]
			\right]
			\right|  \\
			&\lesssim
			\Delta_n e^{-\lambda_{\mathrm{erg}}(m-r)\Delta_n}
			\mathbb E\!\left[
			|\bar H_{0,n,r}^{ij,\alpha}|\,V(Z_{r,n})
			\right]  \\
			&\lesssim
			\Delta_n^2h_n^d
			e^{-\lambda_{\mathrm{erg}}(m-r)\Delta_n}.
		\end{aligned}
		\]
		Summing over \(m>r\) yields
		\[
		\sum_{m=r+1}^{n-1}
		\left|
		\operatorname{Cov}(H_{0,n,r}^{ij,\alpha},H_{m,n,r}^{ij,\alpha})
		\right|
		\lesssim
		\Delta_n^2h_n^d
		\sum_{m=r+1}^{n-1}e^{-\lambda_{\mathrm{erg}}(m-r)\Delta_n}
		\lesssim \Delta_n h_n^d.
		\]
		Combining the overlapping and separated lag estimates,
		\[
		\sum_{m=1}^{n-1}
		\left|
		\operatorname{Cov}(H_{0,n,r}^{ij,\alpha},H_{m,n,r}^{ij,\alpha})
		\right|
		\lesssim \Delta_n h_n^d.
		\]
		Set \(N_{n,r}:=n-r+1\). The stationary variance decomposition gives, for each \(\alpha\in\mathcal A_p\),
		\[
		\begin{aligned}
			&\operatorname{Var}\!\left(
			\frac{1}{N_{n,r}h_n^d}
			\sum_{s=0}^{N_{n,r}-1}\bar H_{s,n,r}^{ij,\alpha}
			\right)\\
			&\quad\lesssim
			\frac{1}{N_{n,r}h_n^{2d}}
			\left[
			\operatorname{Var}(H_{0,n,r}^{ij,\alpha})
			+\sum_{m=1}^{N_{n,r}-1}
			\left|
			\operatorname{Cov}(H_{0,n,r}^{ij,\alpha},H_{m,n,r}^{ij,\alpha})
			\right|
			\right]
			\lesssim \frac{\Delta_n}{nh_n^d}.
		\end{aligned}
		\]
		Chebyshev's inequality then gives
		\[
		\max_{1\le r\le k}
		\left\|
		\widehat T_{n,r\Delta_n,h_n}^{ij,g,(p)}(x)
		-\mathbb E[\widehat T_{n,r\Delta_n,h_n}^{ij,g,(p)}(x)]
		\right\|
		=O_{\mathbb P}\!\left(
		\left(\frac{\Delta_n}{nh_n^d}\right)^{1/2}
		\right),
		\qquad i\ne j.
		\]

		It remains to transfer these moment bounds to the local-polynomial estimator. The identity
		\[
		\begin{aligned}
			&\widehat P_{r\Delta_n,h_n}^{ij,(p)}g(x)
			-R_{r\Delta_n,h_n,p}^{ij}g(x)\\
			&\quad
			=e_0^\top\left(\widehat S_{n,r,h_n}^{i,(p)}(x)\right)^{-1}
			\left(\widehat T_{n,r\Delta_n,h_n}^{ij,g,(p)}(x)-T_{0,n,r}^{ij,g,(p)}(x)\right)\\
			&\qquad
			+e_0^\top\left[
			\left(\widehat S_{n,r,h_n}^{i,(p)}(x)\right)^{-1}
			-\left(S_{0,n,r}^{i,(p)}(x)\right)^{-1}
			\right]T_{0,n,r}^{ij,g,(p)}(x),
		\end{aligned}
		\]
		holds on the event on which all finite-lag design matrices are invertible; by the design-matrix bound above, the probability of the complement tends to zero. {This is the same inverse-matrix expansion used in the proof of \Cref{thm:lp-fixedmesh-clt}.} Using \(A^{-1}-B^{-1}=B^{-1}(B-A)A^{-1}\) together with the \(O_{\mathbb P}(1)\) inverse bound gives
		\[
		\max_{1\le r\le k}
		\left\|
		\left(\widehat S_{n,r,h_n}^{i,(p)}(x)\right)^{-1}
		-\left(S_{0,n,r}^{i,(p)}(x)\right)^{-1}
		\right\|
		=O_{\mathbb P}\!\left((n\Delta_n h_n^d)^{-1/2}\right).
		\]
		Moreover, the same localization argument used above gives \(T_{0,n,r}^{ii,g,(p)}(x)=O(1)\), while {its off-diagonal version together with \Cref{lem:weighted-rare-switch}} gives \(T_{0,n,r}^{ij,g,(p)}(x)=O(\Delta_n)\) for \(i\ne j\), uniformly over \(1\le r\le k\). Combining these estimates with the response-vector bounds yields
		\[
		\max_{1\le r\le k}
		\left|
		\widehat P_{r\Delta_n,h_n}^{ij,(p)}g(x)
		-R_{r\Delta_n,h_n,p}^{ij}g(x)
		\right|
		=O_{\mathbb P}\!\left(a_{n,p}^{ij}(x)\right).
		\]

		{We now prove the smoothing bound. For each \(r\), set \(m_{r,n}:=P_{r\Delta_n}^{ij}g\). Applying the Taylor argument in Section~S.2.3 of the online supplement with \(h=h_n\) and \(m_{\Delta,g}^{ij}\) replaced by \(m_{r,n}\) gives}
		\[
		R_{r\Delta_n,h_n,p}^{ij}g(x)-P_{r\Delta_n}^{ij}g(x)
		=O(h_n^{p+1}).
		\]
		By {part~\textup{(i)} of \Cref{ass:recovery-regularity}}, the \(C^{p+1}(U_x)\)-norms of \(m_{r,n}\) are uniformly bounded over \(1\le r\le k\) and \(n\). Hence the Taylor remainders in the local-polynomial bias argument are uniform over \(r\) and \(n\). The population design matrices satisfy \(S_{0,n,r}^{i,(p)}(x)\to\varpi_i(x)M_p(K)\) uniformly over \(1\le r\le k\), and their inverses are uniformly bounded for all large \(n\). Since \(k\) is fixed, taking the maximum over \(1\le r\le k\) preserves the order
		\[
		\max_{1\le r\le k}
		\left|
		R_{r\Delta_n,h_n,p}^{ij}g(x)
		-P_{r\Delta_n}^{ij}g(x)
		\right|
		=O(h_n^{p+1}).
		\]
		This completes the proof.
	\end{proof}

	{
		\begin{remark}[Sufficient diagonal and off-diagonal stochastic bounds]\label[remark]{rem:recovery-rates}
			The displayed diagonal and off-diagonal rates are sufficient upper bounds for the general blockwise order-\(k\) construction, obtained by controlling the local-polynomial design and response moments separately. The off-diagonal bound exploits the short-time factor in \Cref{lem:weighted-rare-switch}, derived from \Cref{ass:mixing} and \Cref{ass:rare-switch}. No optimality claim is made for these separate-block bounds. For \(k=1\), the common-design forward difference in the main paper realizes the diagonal cancellation exactly and yields the rate in \Cref{thm:first-order-recovery-clt}.
	\end{remark}}

	{
		\begin{proof}[Proof of \Cref{thm:recovery}]
			By the moment identities for \(c_{k,r}\),
			\[
			\frac{1}{\Delta_n^k}
			\sum_{r=0}^k c_{k,r}
			\sum_{\ell=0}^k
			\frac{(r\Delta_n)^\ell}{\ell!}A_\ell^{ij}g(x)
			=A_k^{ij}g(x).
			\]
			Consequently,
			\[
			\begin{aligned}
				\widehat A_{k,n}^{(p),ij}g(x)-A_k^{ij}g(x)
				&=
				\frac{1}{\Delta_n^k}
				\sum_{r=1}^k c_{k,r}
				\left[
				\widehat P_{r\Delta_n,h_n}^{ij,(p)}g(x)
				-R_{r\Delta_n,h_n,p}^{ij}g(x)
				\right]\\
				&\quad+
				\frac{1}{\Delta_n^k}
				\sum_{r=1}^k c_{k,r}
				\left[
				R_{r\Delta_n,h_n,p}^{ij}g(x)
				-P_{r\Delta_n}^{ij}g(x)
				\right]\\
				&\quad+
				\frac{1}{\Delta_n^k}
				\sum_{r=0}^k c_{k,r}
				\left[
				P_{r\Delta_n}^{ij}g(x)
				-\sum_{\ell=0}^k
				\frac{(r\Delta_n)^\ell}{\ell!}A_\ell^{ij}g(x)
				\right].
			\end{aligned}
			\]
			Let \(a_{n,p}^{ij}(x)\) be the sufficient stochastic upper bound in \Cref{lem:finite-lag-recovery-bounds}. The two sample-size conditions in the theorem are equivalent to
			$a_{n,p}^{ij}(x)=o(\Delta_n^k)$.
			Since \(k\) is fixed, \Cref{lem:finite-lag-recovery-bounds} shows that the first term is
			\[
			O_{\mathbb P}\!\left(
			\frac{a_{n,p}^{ij}(x)}{\Delta_n^k}
			\right)=o_{\mathbb P}(1),
			\]
			and that the second term is
			\[
			O\!\left(\frac{h_n^{p+1}}{\Delta_n^k}\right)=o(1).
			\]
			For the third term, the residual at \(r=0\) is zero because \(P_0^{ij}g(x)=A_0^{ij}g(x)\). By part~\textup{(ii)} of \Cref{ass:recovery-regularity}, the maximum of the remaining \(k\) residuals is \(o(\Delta_n^k)\). The third term is therefore \(o(1)\). Combining the three bounds proves the claimed convergence in probability.
		\end{proof}

		\begingroup
		\begin{corollary}[Blockwise first- and second-order recovery]\label[corollary]{cor:low-order-fd-recovery}
			Assume \Cref{ass:wellposedness}. Fix \(p\ge0\), a regime pair \((i,j)\), a test function \(g\in C_c^4(\mathbb R^d)\), and a design point \(x\in\mathbb R^d\).
			\begin{enumerate}[label=(\roman*)]
				\item If the assumptions and rate conditions of \Cref{thm:recovery} hold with \(k=1\), then
				\[
				\widehat A_{1,n}^{(p),ij}g(x)
				=
				\frac{\widehat P_{\Delta_n,h_n}^{ij,(p)}g(x)-\delta_{ij}g(x)}{\Delta_n}
				\xrightarrow{\mathbb P}B_{ij}g(x).
				\]
				\item If the assumptions and rate conditions of \Cref{thm:recovery} hold with \(k=2\), then
				\[
				\widehat C_{n,\mathrm{sep}}^{ij,(p)}g(x)
				:=
				\frac{
					\widehat P_{2\Delta_n,h_n}^{ij,(p)}g(x)
					-2\widehat P_{\Delta_n,h_n}^{ij,(p)}g(x)
					+\delta_{ij}g(x)}{\Delta_n^2}
				\xrightarrow{\mathbb P}C_{ij}g(x).
				\]
			\end{enumerate}
		\end{corollary}

		\begin{proof}
			Apply \Cref{thm:recovery} with \(k=1\) and \(k=2\), respectively, and use \eqref{eq:coeff-through-two}.
		\end{proof}
		\endgroup
	}

For the numerical model in Section~6 of the main paper, the derivative bound in \eqref{eq:numerical-semigroup-derivative-bound} verifies part~\textup{(i)} of \Cref{ass:recovery-regularity}, and \Cref{cor:model-recovery-expansion} verifies part~\textup{(ii)} for the implemented orders \(k=1,2\). Since the Lyapunov function is regime-independent and the switching rates are bounded above by \(0.80\),
\[
q_{\ell m}(y)V(y,m)\le0.80V(y,\ell),
\qquad \ell\ne m,
\]
which verifies \Cref{ass:rare-switch}.

\endgroup

\clearpage
\section{Verification of the martingale-array CLT conditions}
\label{sec:supp-recovery-martingale-clt}
\begingroup

This section verifies the hypotheses of
Lemma~\ref{lem:appendix-martingale-clt} for the Cram\'er--Wold
projections of the first- and second-order recovery arrays. Let
\(X_{s,n}(a)\) and \(X_{s,n}^{[2]}(a)\) denote the projected summands
defined in the respective main-text proofs.

\begin{enumerate}[label=\textup{(\roman*)}]
\item \emph{Filtration and measurability.}
For first order, put \(\mathcal H_{s,n}^{(1)}:=\mathcal F_{t_{s,n}}\). By \eqref{eq:first-order-weighted-innovation}, \(X_{s,n}(a)\) is \(\mathcal H_{s+1,n}^{(1)}\)-measurable and has conditional mean zero given \(\mathcal H_{s,n}^{(1)}\). For second order, put \(\mathcal H_{s,n}^{(2)}:=\mathcal G_{s,n}=\mathcal F_{t_{2s,n}}\). By \eqref{eq:second-order-weighted-innovation}, \(X_{s,n}^{[2]}(a)\) is \(\mathcal H_{s+1,n}^{(2)}\)-measurable and has conditional mean zero given \(\mathcal H_{s,n}^{(2)}\).

\item \emph{Predictable quadratic variation.}
The conditional-variance calculations following \eqref{eq:first-order-linearization} and \eqref{eq:second-order-linearization}, together with the localized \(V\)-norm law of large numbers, give
\[
\sum_s\mathbb E\!\left[X_{s,n}(a)^2\mid\mathcal H_{s,n}^{(1)}\right]
\xrightarrow{\mathbb P}
a^\top\Gamma_1(x)a,
\]
and
\[
\sum_s\mathbb E\!\left[\{X_{s,n}^{[2]}(a)\}^2\mid\mathcal H_{s,n}^{(2)}\right]
\xrightarrow{\mathbb P}
a^\top\Gamma_2(x)a,
\]
where
\[
\Gamma_1(x):=\varpi_i(x)\gamma_{ij,g}(x)Q_p(K),
\qquad
\Gamma_2(x):=2\varpi_i(x)\gamma_{ij,g}(x)Q_p(K).
\]

\item \emph{Conditional Lindeberg condition and realized-square comparison.}
Boundedness of \(K\), \(\psi_p\), and the responses gives
\[
\max_s|X_{s,n}(a)|
\le \frac{C_a}{\sqrt{n\Delta_nh_n^d}}=o(1),
\qquad
\max_s|X_{s,n}^{[2]}(a)|
\le \frac{C_a}{\sqrt{N_n\Delta_nh_n^d}}=o(1).
\]
The second limit follows from \(N_n\Delta_n^3h_n^d\to\infty\) and \(\Delta_n\to0\). These deterministic bounds give the first two positive-variance conditions in Lemma~\ref{lem:appendix-martingale-clt}; they also make the conditional Lindeberg indicators vanish for every fixed \(\varepsilon>0\) and all sufficiently large \(n\).

For either array, write \(U_{s,n}\) for the corresponding projected summand and \(\mathcal H_{s,n}\) for its conditioning \(\sigma\)-field. The relevant fourth-moment bound, conditional Jensen's inequality, and stationary localization give
\[
\sum_{s=0}^{n-1}\mathbb E|X_{s,n}(a)|^4
\lesssim\frac{1}{n\Delta_nh_n^d}\to0,
\qquad
\sum_{s=0}^{N_n-1}\mathbb E|X_{s,n}^{[2]}(a)|^4
\lesssim\frac{1}{N_n\Delta_nh_n^d}\to0.
\]
The centered squares are martingale differences, so orthogonality yields
\[
\mathbb E\!\left[
\left\{
\sum_s\left(U_{s,n}^2-
\mathbb E[U_{s,n}^2\mid\mathcal H_{s,n}]
\right)
\right\}^2
\right]
\le \sum_s\mathbb E|U_{s,n}|^4
\to0.
\]
Together with the preceding predictable quadratic-variation limits, this proves the realized-square condition in Lemma~\ref{lem:appendix-martingale-clt}. Thus the fourth moments are used for the realized-versus-predictable square comparison, rather than for the immediate maximum-increment and Lindeberg checks.

\item \emph{Degenerate variance.}
If \(\gamma_{ij,g}(x)=0\), then \(a^\top\Gamma_r(x)a=0\), \(r=1,2\), for every \(a\). The predictable quadratic variation therefore converges to zero in probability, and Lenglart's inequality yields convergence of each martingale projection to zero in probability. This is the degenerate normal law \(N(0,0)\). If \(\gamma_{ij,g}(x)>0\), apply Lemma~\ref{lem:appendix-martingale-clt} to each nonzero projection; the Cram\'er--Wold device then gives the vector limits stated in the main paper.
\end{enumerate}

\endgroup

\section{Proof of the normalized-smoothness criterion}
\label{sec:supp-normalized-smoothness}
\begingroup

\begin{proof}[Proof of
Proposition~\ref{prop:appendix-normalized-smoothness}]
By Dynkin's formula applied to \(f\),
\[
P_tf-f=\int_0^tP_s\mathcal Lf\,ds.
\]
The assumed \(BUC^{p+1}\)-continuity makes this a Bochner integral in
\(BUC^{p+1}\). Hence
\[
\left\|\frac{P_tf-f}{t}\right\|_{C_b^{p+1}}
\le
\sup_{0\le s\le t}\|P_s\mathcal Lf\|_{C_b^{p+1}},
\]
which proves the first asserted bound after taking the \(i\)-th regime
component.

For the second difference, the semigroup property gives
\[
P_{2t}f-2P_tf+f=(P_t-I)^2f.
\]
Applying Dynkin's formula first to \(f\) and then to \(\mathcal Lf\), and using the commutation of the semigroup operators, yields
\begin{align*}
(P_t-I)^2f
&=(P_t-I)\int_0^tP_u\mathcal Lf\,du\\
&=\int_0^tP_u(P_t-I)\mathcal Lf\,du\\
&=\int_0^t\int_0^tP_{u+v}\mathcal L^2f\,dv\,du.
\end{align*}
Therefore
\[
\left\|
\frac{P_{2t}f-2P_tf+f}{t^2}
\right\|_{C_b^{p+1}}
\le
\sup_{0\le s\le2t}
\|P_s\mathcal L^2f\|_{C_b^{p+1}},
\]
which proves the second asserted bound.

It remains to verify the stated model-level sufficient condition.
Direct differentiation of the generator gives
\[
\mathcal Lf,\ \mathcal L^2f
\in C_b^{p+2}(\mathbb R^d\times S)
\subset BUC^{p+1}(\mathbb R^d\times S).
\]
The standard derivative estimates and strong continuity of the diagonal
diffusion semigroups on \(BUC^{p+1}\), combined with the Duhamel formula
for the switching part and Gronwall's inequality, give, for
\(h\in\{\mathcal Lf,\mathcal L^2f\}\),
\[
\sup_{0\le s\le T}\|P_sh\|_{C_b^{p+1}}
\le C_T\|h\|_{C_b^{p+1}},
\qquad
\|P_sh-h\|_{C_b^{p+1}}\to0
\quad\text{as }s\downarrow0,
\]
for every \(T<\infty\); see \citet{friedman1967parabolic} for the diagonal
diffusion estimates. The semigroup property and the uniform bound extend
continuity at zero to continuity on every finite time interval. These
estimates verify the derivative-preservation and continuity hypotheses of
the proposition. Boundedness of \(f\), \(\mathcal Lf\), and
\(\mathcal L^2f\) justifies the two applications of Dynkin's formula.
This completes the proof.
\end{proof}

\endgroup

\clearpage
\section{Verification for the numerical model}
\label{sec:supp-numerical-model-verification}
\begingroup

The main paper states the numerical model and the reported configurations. For reference, \(d=1\), \(S=\{1,2\}\),
\[
\begin{aligned}
b(x,i)&=-\beta_i x, & \sigma(x,i)&=\sigma_i,\\
q_{12}(x)&=0.55+0.25\tanh x,
&q_{21}(x)&=0.45-0.20\tanh x,
\end{aligned}
\]
where \((\beta_1,\beta_2)=(1,2)\) and \((\sigma_1,\sigma_2)=(1,1.5)\). The reported probes are \(g_0=\chi\), \(g_1(x)=x\chi(x)\), and \(g_2(x)=x^2\chi(x)\), where \(\chi\in C_c^\infty(\mathbb R)\), \(\chi=1\) on \([-1.25,1.25]\), and \(\chi=0\) outside \([-1.5,1.5]\). The following proposition supplies the detailed density, small-ball, semigroup-derivative, variance-positivity, and coefficient-regularity verification.

		\begingroup
		\begin{proposition}[Verification for the numerical model]\label[proposition]{prop:supp-numerical-assumptions}
			{For each reported design point and a bounded open neighborhood \(U_x\Subset\mathbb R\) of that point, the two-regime model above satisfies Assumptions~A1--A6 in the forms invoked by the main-paper numerical experiments: the fixed mesh is \(\Delta=0.05\), the reported blocks are \(1\to1\) and \(1\to2\), the probes are \(g_0,g_1,g_2\), the implemented polynomial degrees are \(p\in\{0,1\}\), and the recovery orders are \(k\in\{1,2\}\). In particular, the fixed-mesh variance condition A4\textup{(iv)} holds for both the diagonal and off-diagonal reported blocks. The coefficient hypotheses of \Cref{prop:kth} also hold through order three, as required by \Cref{thm:second-order-recovery-clt}.}
		\end{proposition}

		\begin{proof}
			The linear drifts, nonzero constant diffusions, and bounded smooth switching rates satisfy Assumption~A1. Since \(0.30\le q_{12}\le0.80\), \(0.25\le q_{21}\le0.65\), and, for \(V(x,i)=1+x^2\),
			\[
			\mathcal LV(x,i)=-2\beta_i x^2+\sigma_i^2\le-cV(x,i)+C.
			\]
			Uniform ellipticity and the strictly positive switching rates provide the irreducibility and petite-set conditions in \citet[Theorem~6.3]{xi2009asymptotic}. Hence Assumption~A2 holds.

			We next verify the density assertions used in Assumption~A3. Write \(\rho_t^{ab}(y,z)\) for the sub-transition density from \((y,a)\) to \((z,b)\), and put \(\bar q:=0.80\). For a deterministic regime skeleton \(\ell:[0,t]\to\{1,2\}\), the associated linear diffusion has Gaussian variance
			\[
			s_{\ell,t}^2
			=
			\int_0^t
			\exp\!\left\{-2\int_u^t\beta_{\ell(v)}\,dv\right\}
			\sigma_{\ell(u)}^2\,du
			\ge e^{-4t}t.
			\]
			In the switching-path expansion, the term with \(m\) switches is the corresponding Gaussian path density multiplied by the survival factors and by \(m\) switching intensities, and integrated over the ordered switch times. The survival factors are at most one, the intensities are at most \(\bar q\), and the ordered simplex has volume \(t^m/m!\). Consequently, for every fixed \(t>0\),
			\begin{equation}\label{eq:numerical-density-bound}
			\sup_{a,b,y,z}\rho_t^{ab}(y,z)
			\le
			\frac{e^{2t}}{\sqrt{2\pi t}}
			\sum_{m=0}^\infty\frac{(\bar q t)^m}{m!}
			=:
			C_t<\infty.
			\end{equation}
			The same expansion converges locally uniformly and gives continuity of \(\rho_t^{ab}\). Its zero-switch term is strictly positive when \(a=b\), and its one-switch term is strictly positive when \(a\ne b\), because both intensities are bounded away from zero. Thus \(\rho_t^{ab}(y,z)>0\) for every \(t>0\).

			For a fixed sampling mesh \(\Delta>0\), the regime-summed Chapman--Kolmogorov identity makes the uniform-in-lag step explicit. For \(r\ge2\),
			\begin{align*}
			\rho_{r\Delta}^{ab}(y,z)
			&=
			\sum_{c=1}^2\int_{\mathbb R}
			\rho_{(r-1)\Delta}^{ac}(y,w)
			\rho_\Delta^{cb}(w,z)\,dw\\
			&\le
			C_\Delta
			\sum_{c=1}^2\int_{\mathbb R}
			\rho_{(r-1)\Delta}^{ac}(y,w)\,dw
			=C_\Delta,
			\end{align*}
			where the last sum is one because the full switching diffusion is conservative. Invariance now gives
			\[
			\varpi_b(z)
			=
			\sum_{a=1}^2\int_{\mathbb R}
			\rho_t^{ab}(y,z)\,\nu(dy,\{a\}).
			\]
			The preceding continuity, positivity, and bound imply that \(\varpi_b\) is continuous, strictly positive, and bounded. Hence it is bounded above and away from zero on \(\overline U_x\). Finally, the stationary two-point sublaw has density
			\[
			p_{r,\Delta}^{ab}(y,z)
			=
			\varpi_a(y)\rho_{r\Delta}^{ab}(y,z),
			\]
			which is uniformly bounded over \(r\ge1\) on \(U_x\times U_x\). Integration over \((x+hD)^2\) gives the small-ball estimate in A3\textup{(ii)}.

			We give a separate argument for the derivative bounds down to \(t=0\). Let \(T_t^i\) be the Ornstein--Uhlenbeck semigroup in regime \(i\), let \(\mathcal T_t=\operatorname{diag}(T_t^1,T_t^2)\), and let \(Q(x)\) be the \(2\times2\) switching-rate matrix. For \(\mathbf f^{(j)}=(g\mathbf 1_{\{1=j\}},g\mathbf 1_{\{2=j\}})^\top\), the vector \(\mathbf U(t)=(P_t^{1j}g,P_t^{2j}g)^\top\) satisfies
			\[
			\mathbf U(t)
			=
			\mathcal T_t\mathbf f^{(j)}
			+
			\int_0^t
			\mathcal T_{t-s}\{Q\mathbf U(s)\}\,ds.
			\]
			For every integer \(m\ge0\),
			\[
			\partial_x^m T_t^i\phi
			=e^{-m\beta_i t}T_t^i(\partial_x^m\phi).
			\]
			All derivatives of \(Q\) are bounded. Differentiating the preceding Duhamel identity and applying Gronwall's inequality therefore gives, for every integer \(m\ge0\) and every \(T<\infty\),
			\begin{equation}\label{eq:numerical-semigroup-derivative-bound}
			\sup_{0\le t\le T}
			\|\mathbf U(t)\|_{C_b^m(\mathbb R)}<\infty.
			\end{equation}
			{The same conclusion holds with \(g^2\) in place of \(g\). This proves A4\textup{(i)--(iii)} for \(p\in\{0,1\}\), without appealing to an estimate that is stated only for times bounded away from zero.}

			For completeness, the normalized first and second differences follow from exact semigroup identities. Writing \(P_t\) for the full hybrid semigroup and \(\mathcal L\) for its generator,
			\begin{align*}
			\frac{P_t\mathbf f^{(j)}-\mathbf f^{(j)}}{t}
			&=
			\frac1t\int_0^tP_s\mathcal L\mathbf f^{(j)}\,ds,\\
			\frac{P_{2t}\mathbf f^{(j)}-2P_t\mathbf f^{(j)}+\mathbf f^{(j)}}{t^2}
			&=
			\frac1{t^2}\int_0^t\int_0^t
			P_{u+v}\mathcal L^2\mathbf f^{(j)}\,du\,dv.
			\end{align*}
			{Because the probes are smooth and compactly supported, \eqref{eq:numerical-semigroup-derivative-bound} applies to the functions on the right-hand sides. Taking the \(i\)-th component gives the uniform \(C^{p+1}(U_x)\) bounds in Assumptions~A5 and A6.}

			It remains to verify the variance condition, including the diagonal blocks. For \(i=1\), \(j\in\{1,2\}\), and any reported probe, let
			\[
			R_{j,g}:=g(X_\Delta)\mathbf 1_{\{\Lambda_\Delta=j\}}
			\quad\text{under }\mathbb P_{(x,1)}.
			\]
			The zero- and one-switch paths show that both events \(\{\Lambda_\Delta=j\}\) and \(\{\Lambda_\Delta\ne j\}\) have positive probability. Moreover, strict positivity of \(\rho_\Delta^{1j}(x,z)\) and the fact that each \(g_r\) is nonzero on some open interval imply
			\(\mathbb P_{(x,1)}(\Lambda_\Delta=j,\ g_r(X_\Delta)\ne0)>0\).
			Thus \(R_{j,g_r}\) is zero with positive probability and nonzero with positive probability, so it is not almost surely constant. Hence
			\[
			v_{\Delta,g_r}^{1j}(x)=\operatorname{Var}_{(x,1)}(R_{j,g_r})>0,
			\qquad j=1,2,\quad r=0,1,2.
			\]
			\enlargethispage{2\baselineskip}

			{The coefficient hypotheses hold through order three, as required for the temporal-bias expansion in \Cref{thm:second-order-recovery-clt}. Supplementary Section~S.3 verifies the additional assumptions for arbitrary-order separate-lag recovery.}
		\end{proof}
		\endgroup

\endgroup

\clearpage
\section{Secondary numerical results}
\label{sec:supp-secondary-numerics}
\begingroup

The main paper reports representative configurations that directly illustrate the principal theoretical conclusions. The tables below give the complete deterministic expansion, fixed-mesh estimation, and shrinking-mesh recovery summaries. In particular, the earlier separate-lag second-order estimator is reported only here.

\begingroup
\setlength{\tabcolsep}{3.5pt}
\renewcommand{\arraystretch}{1.08}
\scriptsize
\subsection{Deterministic expansion benchmark}
Table~\ref{tab:supp-expansion} reports every combination of block, design point, probe, and mesh used in the deterministic benchmark. Here \(R_1/\Delta^2=\{P_\Delta-A_0-\Delta B\}/\Delta^2\), \(R_2/\Delta^2=\{P_\Delta-A_0-\Delta B-\Delta^2C/2\}/\Delta^2\), and the last column reports \(R_2/\Delta^3\).
\begin{longtable}{cccrrrrr}
\caption{Complete deterministic expansion results.}\label{tab:supp-expansion}\\
\toprule
block & $x$ & probe & $\Delta$ & $R_1/\Delta^2$ & $C/2$ & $R_2/\Delta^2$ & $R_2/\Delta^3$\\
\midrule
\endfirsthead
\multicolumn{8}{c}{\tablename\ \thetable{} continued}\\
\toprule
block & $x$ & probe & $\Delta$ & $R_1/\Delta^2$ & $C/2$ & $R_2/\Delta^2$ & $R_2/\Delta^3$\\
\midrule
\endhead
\midrule \multicolumn{8}{r}{continued on next page}\\
\endfoot
\bottomrule
\endlastfoot
\(1\to 1\) & $-0.50$ & $g_{0}$ & $0.200$ & $-0.0865$ & $0.1176$ & $-0.2041$ & $-1.0206$\\
\(1\to 1\) & $-0.50$ & $g_{1}$ & $0.200$ & $-0.2397$ & $-0.6244$ & $0.3846$ & $1.9231$\\
\(1\to 1\) & $-0.50$ & $g_{2}$ & $0.200$ & $-0.9953$ & $-0.5895$ & $-0.4058$ & $-2.0291$\\
\(1\to 1\) & $-0.50$ & $g_{0}$ & $0.100$ & $0.0056$ & $0.1176$ & $-0.1120$ & $-1.1204$\\
\(1\to 1\) & $-0.50$ & $g_{1}$ & $0.100$ & $-0.4215$ & $-0.6244$ & $0.2028$ & $2.0282$\\
\(1\to 1\) & $-0.50$ & $g_{2}$ & $0.100$ & $-0.7949$ & $-0.5895$ & $-0.2053$ & $-2.0534$\\
\(1\to 1\) & $-0.50$ & $g_{0}$ & $0.050$ & $0.1097$ & $0.1176$ & $-0.0079$ & $-0.1587$\\
\(1\to 1\) & $-0.50$ & $g_{1}$ & $0.050$ & $-0.5923$ & $-0.6244$ & $0.0321$ & $0.6412$\\
\(1\to 1\) & $-0.50$ & $g_{2}$ & $0.050$ & $-0.5906$ & $-0.5895$ & $-0.0011$ & $-0.0216$\\
\(1\to 1\) & $-0.50$ & $g_{0}$ & $0.025$ & $0.1188$ & $0.1176$ & $0.0011$ & $0.0449$\\
\(1\to 1\) & $-0.50$ & $g_{1}$ & $0.025$ & $-0.6153$ & $-0.6244$ & $0.0090$ & $0.3618$\\
\(1\to 1\) & $-0.50$ & $g_{2}$ & $0.025$ & $-0.5802$ & $-0.5895$ & $0.0093$ & $0.3725$\\
\(1\to 1\) & $-0.50$ & $g_{0}$ & $0.013$ & $0.1182$ & $0.1176$ & $5.935\times 10^{-4}$ & $0.0475$\\
\(1\to 1\) & $-0.50$ & $g_{1}$ & $0.013$ & $-0.6199$ & $-0.6244$ & $0.0045$ & $0.3605$\\
\(1\to 1\) & $-0.50$ & $g_{2}$ & $0.013$ & $-0.5848$ & $-0.5895$ & $0.0047$ & $0.3753$\\
\(1\to 1\) & $-0.50$ & $g_{0}$ & $6.25\times 10^{-3}$ & $0.1179$ & $0.1176$ & $3.022\times 10^{-4}$ & $0.0484$\\
\(1\to 1\) & $-0.50$ & $g_{1}$ & $6.25\times 10^{-3}$ & $-0.6221$ & $-0.6244$ & $0.0023$ & $0.3606$\\
\(1\to 1\) & $-0.50$ & $g_{2}$ & $6.25\times 10^{-3}$ & $-0.5872$ & $-0.5895$ & $0.0023$ & $0.3748$\\
\(1\to 1\) & $0$ & $g_{0}$ & $0.200$ & $0.2406$ & $0.2750$ & $-0.0344$ & $-0.1721$\\
\(1\to 1\) & $0$ & $g_{1}$ & $0.200$ & $-0.0873$ & $-0.1250$ & $0.0377$ & $0.1886$\\
\(1\to 1\) & $0$ & $g_{2}$ & $0.200$ & $-1.3214$ & $-1.5500$ & $0.2286$ & $1.1432$\\
\(1\to 1\) & $0$ & $g_{0}$ & $0.100$ & $0.2657$ & $0.2750$ & $-0.0093$ & $-0.0935$\\
\(1\to 1\) & $0$ & $g_{1}$ & $0.100$ & $-0.1043$ & $-0.1250$ & $0.0207$ & $0.2072$\\
\(1\to 1\) & $0$ & $g_{2}$ & $0.100$ & $-1.4082$ & $-1.5500$ & $0.1418$ & $1.4179$\\
\(1\to 1\) & $0$ & $g_{0}$ & $0.050$ & $0.2706$ & $0.2750$ & $-0.0044$ & $-0.0885$\\
\(1\to 1\) & $0$ & $g_{1}$ & $0.050$ & $-0.1140$ & $-0.1250$ & $0.0110$ & $0.2207$\\
\(1\to 1\) & $0$ & $g_{2}$ & $0.050$ & $-1.4758$ & $-1.5500$ & $0.0742$ & $1.4846$\\
\(1\to 1\) & $0$ & $g_{0}$ & $0.025$ & $0.2728$ & $0.2750$ & $-0.0022$ & $-0.0890$\\
\(1\to 1\) & $0$ & $g_{1}$ & $0.025$ & $-0.1193$ & $-0.1250$ & $0.0057$ & $0.2284$\\
\(1\to 1\) & $0$ & $g_{2}$ & $0.025$ & $-1.5122$ & $-1.5500$ & $0.0378$ & $1.5134$\\
\(1\to 1\) & $0$ & $g_{0}$ & $0.013$ & $0.2739$ & $0.2750$ & $-0.0011$ & $-0.0893$\\
\(1\to 1\) & $0$ & $g_{1}$ & $0.013$ & $-0.1221$ & $-0.1250$ & $0.0029$ & $0.2326$\\
\(1\to 1\) & $0$ & $g_{2}$ & $0.013$ & $-1.5309$ & $-1.5500$ & $0.0191$ & $1.5283$\\
\(1\to 1\) & $0$ & $g_{0}$ & $6.25\times 10^{-3}$ & $0.2744$ & $0.2750$ & $-5.591\times 10^{-4}$ & $-0.0895$\\
\(1\to 1\) & $0$ & $g_{1}$ & $6.25\times 10^{-3}$ & $-0.1235$ & $-0.1250$ & $0.0015$ & $0.2352$\\
\(1\to 1\) & $0$ & $g_{2}$ & $6.25\times 10^{-3}$ & $-1.5404$ & $-1.5500$ & $0.0096$ & $1.5360$\\
\(1\to 1\) & $0.50$ & $g_{0}$ & $0.200$ & $0.1981$ & $0.4350$ & $-0.2369$ & $-1.1847$\\
\(1\to 1\) & $0.50$ & $g_{1}$ & $0.200$ & $0.3242$ & $0.7020$ & $-0.3778$ & $-1.8889$\\
\(1\to 1\) & $0.50$ & $g_{2}$ & $0.200$ & $-1.0801$ & $-0.8223$ & $-0.2578$ & $-1.2889$\\
\(1\to 1\) & $0.50$ & $g_{0}$ & $0.100$ & $0.3015$ & $0.4350$ & $-0.1335$ & $-1.3355$\\
\(1\to 1\) & $0.50$ & $g_{1}$ & $0.100$ & $0.4954$ & $0.7020$ & $-0.2066$ & $-2.0657$\\
\(1\to 1\) & $0.50$ & $g_{2}$ & $0.100$ & $-0.9536$ & $-0.8223$ & $-0.1313$ & $-1.3130$\\
\(1\to 1\) & $0.50$ & $g_{0}$ & $0.050$ & $0.4139$ & $0.4350$ & $-0.0211$ & $-0.4219$\\
\(1\to 1\) & $0.50$ & $g_{1}$ & $0.050$ & $0.6646$ & $0.7020$ & $-0.0374$ & $-0.7474$\\
\(1\to 1\) & $0.50$ & $g_{2}$ & $0.050$ & $-0.7880$ & $-0.8223$ & $0.0343$ & $0.6856$\\
\(1\to 1\) & $0.50$ & $g_{0}$ & $0.025$ & $0.4294$ & $0.4350$ & $-0.0057$ & $-0.2273$\\
\(1\to 1\) & $0.50$ & $g_{1}$ & $0.025$ & $0.6899$ & $0.7020$ & $-0.0120$ & $-0.4813$\\
\(1\to 1\) & $0.50$ & $g_{2}$ & $0.025$ & $-0.7949$ & $-0.8223$ & $0.0274$ & $1.0977$\\
\(1\to 1\) & $0.50$ & $g_{0}$ & $0.013$ & $0.4322$ & $0.4350$ & $-0.0028$ & $-0.2276$\\
\(1\to 1\) & $0.50$ & $g_{1}$ & $0.013$ & $0.6959$ & $0.7020$ & $-0.0061$ & $-0.4846$\\
\(1\to 1\) & $0.50$ & $g_{2}$ & $0.013$ & $-0.8084$ & $-0.8223$ & $0.0139$ & $1.1143$\\
\(1\to 1\) & $0.50$ & $g_{0}$ & $6.25\times 10^{-3}$ & $0.4336$ & $0.4350$ & $-0.0014$ & $-0.2287$\\
\(1\to 1\) & $0.50$ & $g_{1}$ & $6.25\times 10^{-3}$ & $0.6989$ & $0.7020$ & $-0.0030$ & $-0.4875$\\
\(1\to 1\) & $0.50$ & $g_{2}$ & $6.25\times 10^{-3}$ & $-0.8153$ & $-0.8223$ & $0.0070$ & $1.1223$\\
\(1\to 2\) & $-0.50$ & $g_{0}$ & $0.200$ & $-0.1628$ & $-0.1176$ & $-0.0451$ & $-0.2257$\\
\(1\to 2\) & $-0.50$ & $g_{1}$ & $0.200$ & $0.4544$ & $0.4830$ & $-0.0286$ & $-0.1430$\\
\(1\to 2\) & $-0.50$ & $g_{2}$ & $0.200$ & $0.0926$ & $0.2524$ & $-0.1598$ & $-0.7992$\\
\(1\to 2\) & $-0.50$ & $g_{0}$ & $0.100$ & $-0.1462$ & $-0.1176$ & $-0.0286$ & $-0.2857$\\
\(1\to 2\) & $-0.50$ & $g_{1}$ & $0.100$ & $0.4742$ & $0.4830$ & $-0.0088$ & $-0.0876$\\
\(1\to 2\) & $-0.50$ & $g_{2}$ & $0.100$ & $0.1624$ & $0.2524$ & $-0.0900$ & $-0.9000$\\
\(1\to 2\) & $-0.50$ & $g_{0}$ & $0.050$ & $-0.1255$ & $-0.1176$ & $-0.0079$ & $-0.1570$\\
\(1\to 2\) & $-0.50$ & $g_{1}$ & $0.050$ & $0.4679$ & $0.4830$ & $-0.0151$ & $-0.3024$\\
\(1\to 2\) & $-0.50$ & $g_{2}$ & $0.050$ & $0.2228$ & $0.2524$ & $-0.0297$ & $-0.5934$\\
\(1\to 2\) & $-0.50$ & $g_{0}$ & $0.025$ & $-0.1190$ & $-0.1176$ & $-0.0014$ & $-0.0554$\\
\(1\to 2\) & $-0.50$ & $g_{1}$ & $0.025$ & $0.4715$ & $0.4830$ & $-0.0114$ & $-0.4574$\\
\(1\to 2\) & $-0.50$ & $g_{2}$ & $0.025$ & $0.2429$ & $0.2524$ & $-0.0095$ & $-0.3814$\\
\(1\to 2\) & $-0.50$ & $g_{0}$ & $0.013$ & $-0.1182$ & $-0.1176$ & $-5.940\times 10^{-4}$ & $-0.0475$\\
\(1\to 2\) & $-0.50$ & $g_{1}$ & $0.013$ & $0.4771$ & $0.4830$ & $-0.0059$ & $-0.4725$\\
\(1\to 2\) & $-0.50$ & $g_{2}$ & $0.013$ & $0.2479$ & $0.2524$ & $-0.0046$ & $-0.3659$\\
\(1\to 2\) & $-0.50$ & $g_{0}$ & $6.25\times 10^{-3}$ & $-0.1179$ & $-0.1176$ & $-3.022\times 10^{-4}$ & $-0.0484$\\
\(1\to 2\) & $-0.50$ & $g_{1}$ & $6.25\times 10^{-3}$ & $0.4800$ & $0.4830$ & $-0.0030$ & $-0.4735$\\
\(1\to 2\) & $-0.50$ & $g_{2}$ & $6.25\times 10^{-3}$ & $0.2502$ & $0.2524$ & $-0.0023$ & $-0.3663$\\
\(1\to 2\) & $0$ & $g_{0}$ & $0.200$ & $-0.2735$ & $-0.2750$ & $0.0015$ & $0.0073$\\
\(1\to 2\) & $0$ & $g_{1}$ & $0.200$ & $0.0793$ & $0.1250$ & $-0.0457$ & $-0.2287$\\
\(1\to 2\) & $0$ & $g_{2}$ & $0.200$ & $0.5700$ & $0.8938$ & $-0.3238$ & $-1.6188$\\
\(1\to 2\) & $0$ & $g_{0}$ & $0.100$ & $-0.2685$ & $-0.2750$ & $0.0065$ & $0.0647$\\
\(1\to 2\) & $0$ & $g_{1}$ & $0.100$ & $0.1005$ & $0.1250$ & $-0.0245$ & $-0.2445$\\
\(1\to 2\) & $0$ & $g_{2}$ & $0.100$ & $0.7280$ & $0.8938$ & $-0.1657$ & $-1.6571$\\
\(1\to 2\) & $0$ & $g_{0}$ & $0.050$ & $-0.2706$ & $-0.2750$ & $0.0044$ & $0.0878$\\
\(1\to 2\) & $0$ & $g_{1}$ & $0.050$ & $0.1121$ & $0.1250$ & $-0.0129$ & $-0.2586$\\
\(1\to 2\) & $0$ & $g_{2}$ & $0.050$ & $0.8083$ & $0.8938$ & $-0.0854$ & $-1.7086$\\
\(1\to 2\) & $0$ & $g_{0}$ & $0.025$ & $-0.2728$ & $-0.2750$ & $0.0022$ & $0.0890$\\
\(1\to 2\) & $0$ & $g_{1}$ & $0.025$ & $0.1183$ & $0.1250$ & $-0.0067$ & $-0.2681$\\
\(1\to 2\) & $0$ & $g_{2}$ & $0.025$ & $0.8497$ & $0.8938$ & $-0.0440$ & $-1.7604$\\
\(1\to 2\) & $0$ & $g_{0}$ & $0.013$ & $-0.2739$ & $-0.2750$ & $0.0011$ & $0.0893$\\
\(1\to 2\) & $0$ & $g_{1}$ & $0.013$ & $0.1216$ & $0.1250$ & $-0.0034$ & $-0.2733$\\
\(1\to 2\) & $0$ & $g_{2}$ & $0.013$ & $0.8714$ & $0.8938$ & $-0.0223$ & $-1.7880$\\
\(1\to 2\) & $0$ & $g_{0}$ & $6.25\times 10^{-3}$ & $-0.2744$ & $-0.2750$ & $5.591\times 10^{-4}$ & $0.0895$\\
\(1\to 2\) & $0$ & $g_{1}$ & $6.25\times 10^{-3}$ & $0.1233$ & $0.1250$ & $-0.0017$ & $-0.2764$\\
\(1\to 2\) & $0$ & $g_{2}$ & $6.25\times 10^{-3}$ & $0.8825$ & $0.8938$ & $-0.0113$ & $-1.8023$\\
\(1\to 2\) & $0.50$ & $g_{0}$ & $0.200$ & $-0.4629$ & $-0.4350$ & $-0.0278$ & $-0.1391$\\
\(1\to 2\) & $0.50$ & $g_{1}$ & $0.200$ & $-0.6102$ & $-0.6184$ & $0.0081$ & $0.0407$\\
\(1\to 2\) & $0.50$ & $g_{2}$ & $0.200$ & $0.1693$ & $0.5719$ & $-0.4026$ & $-2.0129$\\
\(1\to 2\) & $0.50$ & $g_{0}$ & $0.100$ & $-0.4564$ & $-0.4350$ & $-0.0214$ & $-0.2140$\\
\(1\to 2\) & $0.50$ & $g_{1}$ & $0.100$ & $-0.6216$ & $-0.6184$ & $-0.0032$ & $-0.0323$\\
\(1\to 2\) & $0.50$ & $g_{2}$ & $0.100$ & $0.3423$ & $0.5719$ & $-0.2296$ & $-2.2962$\\
\(1\to 2\) & $0.50$ & $g_{0}$ & $0.050$ & $-0.4334$ & $-0.4350$ & $0.0016$ & $0.0321$\\
\(1\to 2\) & $0.50$ & $g_{1}$ & $0.050$ & $-0.6002$ & $-0.6184$ & $0.0182$ & $0.3631$\\
\(1\to 2\) & $0.50$ & $g_{2}$ & $0.050$ & $0.4810$ & $0.5719$ & $-0.0909$ & $-1.8178$\\
\(1\to 2\) & $0.50$ & $g_{0}$ & $0.025$ & $-0.4298$ & $-0.4350$ & $0.0053$ & $0.2110$\\
\(1\to 2\) & $0.50$ & $g_{1}$ & $0.025$ & $-0.6026$ & $-0.6184$ & $0.0158$ & $0.6317$\\
\(1\to 2\) & $0.50$ & $g_{2}$ & $0.025$ & $0.5344$ & $0.5719$ & $-0.0375$ & $-1.4999$\\
\(1\to 2\) & $0.50$ & $g_{0}$ & $0.013$ & $-0.4322$ & $-0.4350$ & $0.0028$ & $0.2275$\\
\(1\to 2\) & $0.50$ & $g_{1}$ & $0.013$ & $-0.6101$ & $-0.6184$ & $0.0083$ & $0.6602$\\
\(1\to 2\) & $0.50$ & $g_{2}$ & $0.013$ & $0.5532$ & $0.5719$ & $-0.0187$ & $-1.4973$\\
\(1\to 2\) & $0.50$ & $g_{0}$ & $6.25\times 10^{-3}$ & $-0.4336$ & $-0.4350$ & $0.0014$ & $0.2287$\\
\(1\to 2\) & $0.50$ & $g_{1}$ & $6.25\times 10^{-3}$ & $-0.6142$ & $-0.6184$ & $0.0042$ & $0.6645$\\
\(1\to 2\) & $0.50$ & $g_{2}$ & $6.25\times 10^{-3}$ & $0.5624$ & $0.5719$ & $-0.0094$ & $-1.5115$\\
\end{longtable}

\subsection{Fixed-mesh estimation}
The next table contains the complete fixed-mesh summaries. The reported standard-deviation ratio is the empirical standard deviation divided by the theorem's asymptotic standard error; the two coverage columns correspond to oracle and plug-in studentization.
\begin{longtable}{rrccrrrrrr}
\caption{Complete fixed-mesh summaries from 500 replications.}\label{tab:supp-fixed}\\
\toprule
$n$ & $p$ & block & probe & target & bias & RMSE & std.\ sd & cov.\ or. & cov.\ pl.\\
\midrule
\endfirsthead
\multicolumn{10}{c}{\tablename\ \thetable{} continued}\\
\toprule
$n$ & $p$ & block & probe & target & bias & RMSE & std.\ sd & cov.\ or. & cov.\ pl.\\
\midrule
\endhead
\midrule \multicolumn{10}{r}{continued on next page}\\
\endfoot
\bottomrule
\endlastfoot
$50000$ & $0$ & \(1\to 1\) & $g_{0}$ & $0.97318$ & $1.673\times 10^{-5}$ & $0.0075$ & $0.981$ & $0.962$ & $0.948$\\
$50000$ & $0$ & \(1\to 1\) & $g_{1}$ & $-0.00028$ & $4.168\times 10^{-4}$ & $0.0100$ & $0.981$ & $0.962$ & $0.964$\\
$50000$ & $0$ & \(1\to 1\) & $g_{2}$ & $0.04631$ & $1.804\times 10^{-4}$ & $0.0032$ & $1.011$ & $0.948$ & $0.932$\\
$50000$ & $0$ & \(1\to 2\) & $g_{0}$ & $0.02682$ & $-1.664\times 10^{-5}$ & $0.0075$ & $0.981$ & $0.962$ & $0.948$\\
$50000$ & $0$ & \(1\to 2\) & $g_{1}$ & $0.00028$ & $-5.121\times 10^{-5}$ & $0.0020$ & $0.916$ & $0.962$ & $0.970$\\
$50000$ & $0$ & \(1\to 2\) & $g_{2}$ & $0.00202$ & $-1.674\times 10^{-5}$ & $0.0010$ & $0.979$ & $0.958$ & $0.834$\\
$50000$ & $1$ & \(1\to 1\) & $g_{0}$ & $0.97318$ & $1.495\times 10^{-4}$ & $0.0046$ & $1.018$ & $0.956$ & $0.948$\\
$50000$ & $1$ & \(1\to 1\) & $g_{1}$ & $-0.00028$ & $5.144\times 10^{-5}$ & $0.0059$ & $0.974$ & $0.952$ & $0.954$\\
$50000$ & $1$ & \(1\to 1\) & $g_{2}$ & $0.04631$ & $7.593\times 10^{-4}$ & $0.0020$ & $0.973$ & $0.930$ & $0.938$\\
$50000$ & $1$ & \(1\to 2\) & $g_{0}$ & $0.02682$ & $-1.494\times 10^{-4}$ & $0.0046$ & $1.018$ & $0.956$ & $0.948$\\
$50000$ & $1$ & \(1\to 2\) & $g_{1}$ & $0.00028$ & $-1.834\times 10^{-5}$ & $0.0011$ & $0.898$ & $0.968$ & $0.978$\\
$50000$ & $1$ & \(1\to 2\) & $g_{2}$ & $0.00202$ & $-1.589\times 10^{-5}$ & $5.793\times 10^{-4}$ & $0.950$ & $0.970$ & $0.900$\\
$100000$ & $0$ & \(1\to 1\) & $g_{0}$ & $0.97318$ & $1.423\times 10^{-4}$ & $0.0059$ & $0.957$ & $0.966$ & $0.936$\\
$100000$ & $0$ & \(1\to 1\) & $g_{1}$ & $-0.00028$ & $4.017\times 10^{-4}$ & $0.0081$ & $0.974$ & $0.958$ & $0.958$\\
$100000$ & $0$ & \(1\to 1\) & $g_{2}$ & $0.04631$ & $5.518\times 10^{-6}$ & $0.0025$ & $0.982$ & $0.956$ & $0.948$\\
$100000$ & $0$ & \(1\to 2\) & $g_{0}$ & $0.02682$ & $-1.422\times 10^{-4}$ & $0.0059$ & $0.957$ & $0.966$ & $0.936$\\
$100000$ & $0$ & \(1\to 2\) & $g_{1}$ & $0.00028$ & $5.926\times 10^{-5}$ & $0.0016$ & $0.952$ & $0.954$ & $0.954$\\
$100000$ & $0$ & \(1\to 2\) & $g_{2}$ & $0.00202$ & $-3.595\times 10^{-5}$ & $7.909\times 10^{-4}$ & $0.945$ & $0.970$ & $0.888$\\
$100000$ & $1$ & \(1\to 1\) & $g_{0}$ & $0.97318$ & $1.981\times 10^{-4}$ & $0.0034$ & $0.961$ & $0.960$ & $0.962$\\
$100000$ & $1$ & \(1\to 1\) & $g_{1}$ & $-0.00028$ & $7.500\times 10^{-5}$ & $0.0047$ & $1.015$ & $0.940$ & $0.944$\\
$100000$ & $1$ & \(1\to 1\) & $g_{2}$ & $0.04631$ & $5.049\times 10^{-4}$ & $0.0016$ & $1.027$ & $0.926$ & $0.938$\\
$100000$ & $1$ & \(1\to 2\) & $g_{0}$ & $0.02682$ & $-1.980\times 10^{-4}$ & $0.0034$ & $0.961$ & $0.960$ & $0.962$\\
$100000$ & $1$ & \(1\to 2\) & $g_{1}$ & $0.00028$ & $9.031\times 10^{-6}$ & $8.958\times 10^{-4}$ & $0.921$ & $0.966$ & $0.958$\\
$100000$ & $1$ & \(1\to 2\) & $g_{2}$ & $0.00202$ & $-3.300\times 10^{-5}$ & $4.381\times 10^{-4}$ & $0.930$ & $0.972$ & $0.938$\\
$200000$ & $0$ & \(1\to 1\) & $g_{0}$ & $0.97318$ & $5.758\times 10^{-5}$ & $0.0050$ & $0.986$ & $0.948$ & $0.942$\\
$200000$ & $0$ & \(1\to 1\) & $g_{1}$ & $-0.00028$ & $4.071\times 10^{-4}$ & $0.0068$ & $1.004$ & $0.954$ & $0.950$\\
$200000$ & $0$ & \(1\to 1\) & $g_{2}$ & $0.04631$ & $3.938\times 10^{-5}$ & $0.0021$ & $1.000$ & $0.964$ & $0.958$\\
$200000$ & $0$ & \(1\to 2\) & $g_{0}$ & $0.02682$ & $-5.849\times 10^{-5}$ & $0.0050$ & $0.985$ & $0.948$ & $0.942$\\
$200000$ & $0$ & \(1\to 2\) & $g_{1}$ & $0.00028$ & $5.230\times 10^{-5}$ & $0.0014$ & $0.983$ & $0.958$ & $0.962$\\
$200000$ & $0$ & \(1\to 2\) & $g_{2}$ & $0.00202$ & $-5.352\times 10^{-5}$ & $6.675\times 10^{-4}$ & $0.980$ & $0.966$ & $0.884$\\
$200000$ & $1$ & \(1\to 1\) & $g_{0}$ & $0.97318$ & $2.772\times 10^{-4}$ & $0.0026$ & $0.973$ & $0.964$ & $0.954$\\
$200000$ & $1$ & \(1\to 1\) & $g_{1}$ & $-0.00028$ & $2.448\times 10^{-4}$ & $0.0036$ & $1.013$ & $0.944$ & $0.948$\\
$200000$ & $1$ & \(1\to 1\) & $g_{2}$ & $0.04631$ & $4.253\times 10^{-4}$ & $0.0012$ & $0.983$ & $0.942$ & $0.946$\\
$200000$ & $1$ & \(1\to 2\) & $g_{0}$ & $0.02682$ & $-2.777\times 10^{-4}$ & $0.0026$ & $0.973$ & $0.964$ & $0.954$\\
$200000$ & $1$ & \(1\to 2\) & $g_{1}$ & $0.00028$ & $-3.050\times 10^{-6}$ & $6.950\times 10^{-4}$ & $0.927$ & $0.966$ & $0.970$\\
$200000$ & $1$ & \(1\to 2\) & $g_{2}$ & $0.00202$ & $-3.672\times 10^{-5}$ & $3.478\times 10^{-4}$ & $0.954$ & $0.964$ & $0.932$\\
$400000$ & $0$ & \(1\to 1\) & $g_{0}$ & $0.97318$ & $8.122\times 10^{-6}$ & $0.0041$ & $1.007$ & $0.948$ & $0.940$\\
$400000$ & $0$ & \(1\to 1\) & $g_{1}$ & $-0.00028$ & $1.208\times 10^{-4}$ & $0.0056$ & $1.019$ & $0.950$ & $0.944$\\
$400000$ & $0$ & \(1\to 1\) & $g_{2}$ & $0.04631$ & $6.217\times 10^{-5}$ & $0.0017$ & $1.005$ & $0.942$ & $0.946$\\
$400000$ & $0$ & \(1\to 2\) & $g_{0}$ & $0.02682$ & $-8.532\times 10^{-6}$ & $0.0041$ & $1.007$ & $0.948$ & $0.940$\\
$400000$ & $0$ & \(1\to 2\) & $g_{1}$ & $0.00028$ & $4.646\times 10^{-5}$ & $0.0012$ & $1.015$ & $0.930$ & $0.942$\\
$400000$ & $0$ & \(1\to 2\) & $g_{2}$ & $0.00202$ & $-3.028\times 10^{-5}$ & $5.353\times 10^{-4}$ & $0.969$ & $0.954$ & $0.910$\\
$400000$ & $1$ & \(1\to 1\) & $g_{0}$ & $0.97318$ & $1.201\times 10^{-4}$ & $0.0021$ & $0.989$ & $0.952$ & $0.946$\\
$400000$ & $1$ & \(1\to 1\) & $g_{1}$ & $-0.00028$ & $1.846\times 10^{-4}$ & $0.0029$ & $1.035$ & $0.936$ & $0.936$\\
$400000$ & $1$ & \(1\to 1\) & $g_{2}$ & $0.04631$ & $3.004\times 10^{-4}$ & $8.959\times 10^{-4}$ & $0.982$ & $0.936$ & $0.946$\\
$400000$ & $1$ & \(1\to 2\) & $g_{0}$ & $0.02682$ & $-1.205\times 10^{-4}$ & $0.0021$ & $0.988$ & $0.952$ & $0.946$\\
$400000$ & $1$ & \(1\to 2\) & $g_{1}$ & $0.00028$ & $7.867\times 10^{-7}$ & $5.641\times 10^{-4}$ & $0.976$ & $0.946$ & $0.944$\\
$400000$ & $1$ & \(1\to 2\) & $g_{2}$ & $0.00202$ & $-2.380\times 10^{-5}$ & $2.605\times 10^{-4}$ & $0.928$ & $0.972$ & $0.942$\\
\end{longtable}

\subsection{Shrinking-mesh recovery}
{\Cref{tab:supp-recovery-B,tab:supp-recovery-C} report every separate-lag recovery run. \Cref{tab:supp-recovery-C-nb} reports the nonoverlapping second-order distributional check used in the revised manuscript. Parentheses give the Monte Carlo standard error of the replication mean. The RMSE is computed relative to the deterministic finite-difference oracle.}
\begin{table}[htbp]

\centering\scriptsize
\caption{Complete first-order recovery results from 100 replications.}\label{tab:supp-recovery-B}
\begin{tabular}{rrrcrrrrr}
\toprule
$n$ & $\Delta$ & $h$ & probe & $B$ & $B_\Delta^{\rm FD}$ & mean (s.e.) & RMSE$^{\rm FD}$ & $\Pp(|\widehat B-B|>0.05)$\\
\midrule
$300000$ & $0.15$ & $0.0513$ & $g_{0}$ & $0.5500$ & $0.5094$ & $0.5131\;(0.0022)$ & $0.0219$ & $0.31$\\
$300000$ & $0.15$ & $0.0513$ & $g_{1}$ & $0$ & $0.0134$ & $0.0124\;(0.0010)$ & $0.0100$ & $0$\\
$300000$ & $0.15$ & $0.0513$ & $g_{2}$ & $0$ & $0.0970$ & $0.0983\;(7.973\times 10^{-4})$ & $0.0080$ & $1.00$\\
$1500000$ & $0.12$ & $0.0393$ & $g_{0}$ & $0.5500$ & $0.5177$ & $0.5190\;(0.0013)$ & $0.0129$ & $5.0\times 10^{-2}$\\
$1500000$ & $0.12$ & $0.0393$ & $g_{1}$ & $0$ & $0.0115$ & $0.0115\;(5.411\times 10^{-4})$ & $0.0054$ & $0$\\
$1500000$ & $0.12$ & $0.0393$ & $g_{2}$ & $0$ & $0.0835$ & $0.0838\;(4.023\times 10^{-4})$ & $0.0040$ & $1.00$\\
$6400000$ & $0.10$ & $0.0315$ & $g_{0}$ & $0.5500$ & $0.5231$ & $0.5229\;(7.079\times 10^{-4})$ & $0.0070$ & $0$\\
$6400000$ & $0.10$ & $0.0315$ & $g_{1}$ & $0$ & $0.0101$ & $0.0102\;(3.197\times 10^{-4})$ & $0.0032$ & $0$\\
$6400000$ & $0.10$ & $0.0315$ & $g_{2}$ & $0$ & $0.0728$ & $0.0728\;(1.921\times 10^{-4})$ & $0.0019$ & $1.00$\\
\bottomrule
\end{tabular}
\end{table}
\begingroup
\begin{table}[htbp]

\centering\scriptsize
\caption{Complete second-order recovery results from 100 replications.}\label{tab:supp-recovery-C}
\begin{tabular}{rrrcrrrrr}
\toprule
$n$ & $\Delta$ & $h$ & probe & $C$ & $C_\Delta^{\rm FD}$ & mean (s.e.) & RMSE$^{\rm FD}$ & $\Pp(|\widehat C-C|>0.10)$\\
\midrule
$300000$ & $0.15$ & $0.0513$ & $g_{0}$ & $-0.5500$ & $-0.5613$ & $-0.5580\;(0.0195)$ & $0.1936$ & $0.63$\\
$300000$ & $0.15$ & $0.0513$ & $g_{1}$ & $0.2500$ & $0.0681$ & $0.0816\;(0.0134)$ & $0.1335$ & $0.70$\\
$300000$ & $0.15$ & $0.0513$ & $g_{2}$ & $1.7875$ & $0.4796$ & $0.4712\;(0.0102)$ & $0.1014$ & $1.00$\\
$1500000$ & $0.12$ & $0.0393$ & $g_{0}$ & $-0.5500$ & $-0.5626$ & $-0.5798\;(0.0153)$ & $0.1534$ & $0.54$\\
$1500000$ & $0.12$ & $0.0393$ & $g_{1}$ & $0.2500$ & $0.0947$ & $0.0952\;(0.0085)$ & $0.0849$ & $0.74$\\
$1500000$ & $0.12$ & $0.0393$ & $g_{2}$ & $1.7875$ & $0.6680$ & $0.6694\;(0.0068)$ & $0.0681$ & $1.00$\\
$6400000$ & $0.10$ & $0.0315$ & $g_{0}$ & $-0.5500$ & $-0.5571$ & $-0.5476\;(0.0108)$ & $0.1083$ & $0.30$\\
$6400000$ & $0.10$ & $0.0315$ & $g_{1}$ & $0.2500$ & $0.1159$ & $0.1085\;(0.0062)$ & $0.0622$ & $0.74$\\
$6400000$ & $0.10$ & $0.0315$ & $g_{2}$ & $1.7875$ & $0.8238$ & $0.8238\;(0.0038)$ & $0.0376$ & $1.00$\\
\bottomrule
\end{tabular}
\end{table}
\begin{table}[htbp]

\centering\scriptsize
\caption{Distributional check for the nonoverlapping common-design second-order estimator for \(g_0\), block \(1\to2\), at \(x=0\).}\label{tab:supp-recovery-C-nb}
\begin{tabular}{rrrrrrrrrr}
\toprule
$n$ & $\Delta$ & $N\Delta^3h$ & $C_\Delta^{\rm FD}$ & mean (s.e.) & RMSE$^{\rm FD}$ & mean/sd $Z^{\rm or}$ & cov. & mean/sd $Z^{\rm pl}$ & cov.\\
\midrule
$300000$ & $0.15$ & $25.97$ & $-0.5613$ & $-0.5484\;(0.0276)$ & $0.2750$ & $0.041/0.871$ & $0.97$ & $0.047/0.863$ & $0.97$\\
$1500000$ & $0.12$ & $50.93$ & $-0.5626$ & $-0.5853\;(0.0196)$ & $0.1960$ & $-0.100/0.865$ & $0.98$ & $-0.098/0.868$ & $0.98$\\
$6400000$ & $0.10$ & $100.80$ & $-0.5571$ & $-0.5360\;(0.0173)$ & $0.1737$ & $0.131/1.077$ & $0.94$ & $0.133/1.076$ & $0.94$\\
\bottomrule
\end{tabular}
\end{table}
\endgroup

\endgroup

\end{document}